\RequirePackage{fix-cm}
\documentclass[]{article}
\usepackage{arxiv}
\usepackage{graphicx}
\usepackage{xcolor}
\usepackage{subfigure}
\usepackage{placeins}
\usepackage{amstext,amsmath,amssymb,amsbsy,bm,amsthm}
\theoremstyle{definition}
\newtheorem{definition}{Definition}[section]
\newtheorem{theorem}{Theorem}[section]
\newtheorem{corollary}{Corollary}[theorem]
\newtheorem{lemma}[theorem]{Lemma}
\newtheorem{remark}{Remark}[section]
\DeclareMathOperator{\sech}{sech}

\begin{document}

\title{Numerical Solution and Bifurcation Analysis of Nonlinear Partial Differential Equations  with Extreme Learning Machines
}


\author{
  Gianluca Fabiani\\
  Scuola Superiore Meridionale\\
  Universit\`a degli Studi di Napoli Federico II\\
  Italy\\
  \texttt{gianluca.fabiani@unina.it} \\
   \And
 Francesco Calabr\`{o} \\
 Dipartimento di Matematica e Applicazioni ``Renato Caccioppoli"\\
 Universit\`a degli Studi di Napoli ``Federico II"\\
 Italy \\
  \texttt{francesco.calabro@unina.it}
  \And
  Lucia Russo\\
  Institute of Sustainable Mobility and Energy\\
  Consiglio Nazionale delle Ricerche\\
  Italy\\
  \texttt{lucia.russo@stems.cnr.it }
  \And
  Constantinos Siettos\\
  Dipartimento di Matematica e Applicazioni ``Renato Caccioppoli"\\
  Universit\`a degli Studi di Napoli ``Federico II"\\
  Italy\\
  \texttt{constantinos.siettos@unina.it}
  }



\maketitle

\begin{abstract}
We address a new numerical scheme based on a class of machine learning methods, the so-called  Extreme Learning Machines (ELM) with both sigmoidal and radial-basis functions, for the computation of steady-state solutions and the construction of (one-dimensional) bifurcation diagrams of nonlinear partial differential equations (PDEs). For our illustrations, we considered two benchmark problems, namely (a) the one-dimensional viscous Burgers with both homogeneous (Dirichlet) and non-homogeneous boundary conditions, and, (b) the one- and two-dimensional Liouville–Bratu–Gelfand PDEs with homogeneous Dirichlet boundary conditions. For the one-dimensional Burgers and Bratu PDEs, exact analytical solutions are available and used for comparison purposes against the numerical derived solutions. Furthermore, the numerical efficiency (in terms of accuracy and size of the grid) of the proposed numerical machine-learning scheme is compared against central finite differences (FD) and Galerkin weighted-residuals finite-element (FEM) methods. We show that the proposed ELM  numerical method outperforms both FD and FEM methods for medium to large sized grids, while provides equivalent results with the FEM for low to medium sized grids; both methods (ELM and FEM) outperform the FD scheme.
\end{abstract}

\section{Introduction}

The solution of partial differential equations (PDEs) with the aid of machine learning as an alternative to conventional numerical analysis methods can been traced back in the early '90s. For example, Lagaris et al. \cite{lagaris1998artificial} presented a method based on feedforward neural networks (FNN) that can be used for the numerical solution of linear and nonlinear PDEs. The method is based on the construction of appropriate trial functions, the analytical derivation of the gradient of the error with respect to the network parameters and collocation. The training of the FNN was achieved iteratively with the quasi-Newton BFGS method. Gonzalez-Garcia et al. \cite{gonzalez1998identification} proposed a multilayer neural network scheme that resembles the Runge-Kutta integrator for the identification of dynamical systems described by nonlinear PDEs.\par
Nowadays, the exponentially increasing- over the last decades- computational power and recent theoretical advances, have allowed further developments at the intersection between machine learning and numerical analysis. In particular, on the side of the numerical solution of PDEs, the development of systematic and robust machine-learning methodologies targeting at the solution of large scale systems of nonlinear problems with steep gradients constitutes an open and challenging problem in the area. Very recently \cite{raissi2018numerical,raissi2019physics} addressed the use of numerical Gaussian Processes  and Deep Neural Networks (DNNs) with collocation to solve time-dependent non-linear PDEs circumventing the need for spatial discretization of the differential operators. The proposed approach is demonstrated through the one-dimensional nonlinear Burgers, the Schrödinger and the Allen–Cahn equations. In \cite{han2018solving}, DNNs were used to solve high-dimensional nonlinear parabolic PDEs including the Black–Scholes, the Hamilton–Jacobi–Bellman and the Allen–Cahn equation. In \cite{samaniego2020energy},  DNNs were used to approximate the solution of PDEs arising in engineering problems by exploiting the variational structure that may arise in some of these problems. In \cite{chan2019machine,gebhardt2020framework,hadash2018estimate} DNNs were used to solve high-dimensional semi-linear PDEs; the efficiency of the method was compared against other deep learning schemes. In \cite{wei2018machine}, the authors used FNN to solve modified high-dimensional diffusion equations: the training of the FNN is achieved iteratively using an unsupervised universal machine-learning solver. Most recently, in \cite{fresca2020comprehensive}, the authors have used DNN to construct non-linear reduced-order models of time-dependent parametrized PDEs.\par 
Over the last few years, extreme learning machines (ELMs) have been used as an alternative to other machine learning schemes, thus providing a good generalization at a low computational cost \cite{huang2006extreme}. The idea behind ELMs is to randomly set the values of the weights between the input and hidden layer, the biases and the parameters of the activation/transfer functions  and determine the weights between the last hidden and output layer by solving a least-squares problem. The solution of such a least-squares problem is the whole ``training" procedure; hence, no iterative training is needed for ELMs, in contrast with what happens with the other aforementioned machine learning methods. Extensions to this basic scheme include 
multilayer ELMs \cite{dai2019multilayer,huang2013,tang2015extreme} and deep ELMs \cite{tissera2016deep}. As with conventional neural networks, convolutional networks and deep learning, ELMs have been mainly used for classification purposes \cite{bai2014sparse,chaturvedi2018bayesian,chen2020unsupervised,huang2010optimization,tang2015extreme,wang2011study}.\par On the other hand, the use of ELMs for ``traditional" numerical analysis tasks and in particular  for the numerical solution of PDEs is still widely unexplored. To the best of our knowledge, the only study on the subject is that of \cite{DWIVEDI202096} where the authors however report a failure of ELMs to deal, for example, with PDEs whose solutions exhibit steep gradients. Recently, we have proposed an ELM scheme to deal with such steep gradients appearing in linear PDEs \cite{calabro2020extreme} demonstrating through several benchmark problems that the proposed approach is efficient.\par
Here, we propose a problem-independent new  numerical scheme based on ELMs for the solution of nonlinear PDEs that may exhibit sharp gradients. As nonlinear PDEs may also exhibit non-uniqueness and/or non-existence of solutions, we also show how one can use ELMs for the construction of (one-dimensional) bifurcation diagrams of PDEs. The efficiency of the proposed numerical scheme is demonstrated and discussed through two well-studied benchmark problems:  the one-dimensional viscous Burgers equation, a representative of the class of advection-diffusion problems and the one- and two-dimensional  Liouville–Bratu–Gelfand PDE, a representative of the class of reaction-diffusion problems.  
The numerical accuracy of the proposed scheme is  compared against the analytical solutions and the exact locations of the limit points that are known for the one-dimensional PDEs, but also against the corresponding numerical approximations obtained with central finite differences (FD) and Galerkin finite elements methods (FEM).


\section{Extreme Learning Machines} \label{sec:headings}
ELMs are a class of machine-learning techniques for defining functions derived by artificial neural networks (ANNs) with fixed internal weights and biases.
Thus, ELMs have the same structure of a single hidden layer FNN with $N$ neurons. Next, we report the definition of ELM functions which we denote by ${v}(\bm{x}):\mathbb{R}^d\to \mathbb{R}$. 
\begin{definition}[ELM network with single hidden layer]
\begin{itemize}
Assuming:
\item An infinitely differentiable non polynomial function $\psi$, \emph{the activation (transfer) function} for the neurons in the hidden layer.
    \item A randomly-generated matrix $A\in \mathbb{R}^{N \times d}$ containing the \emph{internal weights matrix} 
    connecting the input layer and the hidden layer.
    \item A randomly-generated vector $\bm{\beta}\in \mathbb{R}^{N} $, containing the \emph{biases} 
    in the hidden layer.
\end{itemize} 
Then, we say that ${v}$ is an ELM function with a single hidden layer, if there exists a choice of $\bm{w}\in \mathbb{R}^{N}$, (the \emph{external weights vector} 
between the hidden layer and the output layer) such that:
\begin{equation}
    {v}(\bm{x};A,\bm{\beta};\bm{w})=\sum_{j=1}^N w_{j} \psi(\bm{\alpha}_{j} \cdot \bm{x} + \beta_{j}),
    \label{eq:ELMframework}
\end{equation}
where $\bm{x}=(x_1,x_2,\dots,x_d) \in \mathbb{R}^d$ is the input vector. 
\end{definition}
We remark that the regularity assumption in the above definition is not mandatory for the approximation properties, but in our case some regularity is needed to write the collocation method, thus for this case, we  also briefly present the necessary theory.
It is well-known, that for ANNs, where $A$ and $\bm{\beta}$ are not a-priori fixed, holds the universal approximation theorem if $\psi$ is a non-polynomial function: the functional space is spanned by the basis functions $\{\psi(\bm{\alpha} \cdot \bm{x} + \beta) , \bm{\alpha}\in \mathbb{R}^d, \beta\in\mathbb{R} \}$ that is dense in $L_2$. Moreover, with some regularity assumptions on the activation function(s), the approximation holds true also for the derivatives (see e.g. Theorem 3.1 and Theorem 4.1 in \cite{pinkus1999approximation}). Besides, fixing $A$ and $\bm{\beta}$ a priori is not a limitation, because the universal approximation is still valid in the setting of ELMs (see Theorem 2 in \cite{huang2015trends}):

\begin{theorem}[Universal approximation] 
Let the coefficients $\bm{\alpha}, \beta$ in the function sequence $\{ \psi(\bm{\alpha}_{j} \cdot \bm{x} + \beta_{j})\}_{j=1}^N $ be randomly generated according to any continuous sampling distribution and call $\tilde{v}^N\in \text{span}\{ \psi(\bm{\alpha}_{j} \cdot \bm{x} + \beta_{j}) \,,\ j=1\dots N \} $ the ELM function determined by ordinary least square solution of $\| f(\bm{x})- \tilde{v}^N(\bm{x}) \| $, where $f$ is a continuous function.\\ Then, one has with probability one that $\lim_{N\to \infty} \|f-\tilde{v}^N\|= 0 $.
\end{theorem}
We remark that in the ANN framework, the classical way is to optimize the parameters of the network (internal and external weights and biases) iteratively, e.g. by stochastic gradient descent algorithms that have a high computational cost and don't ensure a global but only local convergence.
On the other hand, ELM networks are advantageous because the solution of an interpolation problem leads to a system of linear equations, where the only unknowns are the external weights $\bm{w}$. For example, consider $M$ points $\bm{x}_i$ such that $y_i=v(\bm{x}_i)$ for $i=1,\dots,M$. In the ELM framework (\ref{eq:ELMframework}) the interpolation problem becomes:
\begin{equation*}
    \sum_{j=1}^N w_j \psi_j(\bm{x}_i)=y_i, \qquad i=1,\dots,M
\end{equation*}
where $N$ is the number of neurons and $\psi_j(\bm{x})$ is used to denote $\psi(\bm{\alpha}_j \cdot \bm{x} + \beta_j)$. Thus, this is a system of $M$ equations and $N$ unknowns that in a matrix form can be we written as:
\begin{equation}\label{ESSE}
    S\bm{w}=\bm{y},
\end{equation}
where $\bm{y}=(y_1,\dots,y_M) \in \mathbb{R}^M$ and $S \in \mathbb{R}^{M \times N}$ is the matrix with elements $(S)_{ij}=\psi_j(\bm{x}_i)$.
If the problem is square ($N=M$) and the parameters $\bm{\alpha}_j$ and $\beta_j$ are chosen randomly, it can be proved that the matrix $S$ is invertible with probability 1 (see i.e. Theorem 1 \cite{huang2015trends}) and so, there is a unique solution, than can be numerically found; if one has to deal with an ill-conditioned matrix, one can still attempt to find a numerically robust solution by applying established numerical analysis methods suitable for such a case  (e.g. by constructing the Moore-Penrose pseudoinverse using QR factorization or SVD). If the problem is under-determined ($N>M$), the linear system has (infinite) many solutions and can be solved by applying regularization in order to pick the solution with e.g. the minimal $L_2$ norm. Such an approach provides the best solution to the optimization problem related to the magnitude of the calculated weights (see \cite{huang2011extreme}).\par
Thus, in ELM networks, one has to choose the type of the activation/transfer function and the values of the internal weights and biases. Since the only limitation is that $\psi$ is a non-polynomial function, there are infinitely many choices. The most common choice are the sigmoidal functions (SF) (also referred as ridge functions or plane waves) and the radial basis functions (RBF) \cite{asprone2010particle,pinkus1999approximation}.

Below, we describe the construction procedure and main features of the proposed ELM scheme, based on these two transfer functions. In the case of the logistic sigmoid transfer function this investigation was made in our work for one-dimensional linear PDEs \cite{calabro2020extreme}. Here, we report the fundamental arguments and we extend them to include RBFs and two-dimensional nonlinear problems.\par 
\subsection{ELM with sigmoidal functions}
For the SF case, we select the logistic sigmoid, that is defined by
\begin{equation}
    \psi_j(\bm{x}) \equiv \sigma_j(\bm{x})=\frac{1}{1+\text{exp}(-\bm{\alpha}_j\cdot \bm{x}-\beta_j)}.
    \label{eq:logisticSF}
\end{equation}
For this function, it is straightforward to compute the derivatives. In particular the derivatives with respect to the $x_k$ component are given by:
\begin{equation}\label{eq:der:SF}
\begin{split}
    \frac{\partial}{\partial x_k}\sigma_j(\bm{x})&= \alpha_{j,k}\frac{\text{exp}(z_j)}{(1+\text{exp}(z_j))^2},\\
    \frac{\partial^2}{\partial x_k^2}\sigma_j(\bm{x})&= \alpha_{j,k}^2\frac{\text{exp}(z_j) \cdot (\text{exp}(z_j)-1)}{(1+\text{exp}(z_j))^3},
\end{split}
\end{equation}
where $z_j=\bm{\alpha}_j \cdot \bm{x}+\beta_j$.
\\
A crucial point in the ELM framework is how to fix the values of the internal weights and 
biases in a proper way. Indeed, despite the fact that theoretically any random choice should be good enough, in practice, it is convenient to define an appropriate range of values for the parameters $\alpha_{j,k}$ and $\beta_{j}$ that are strictly related to the selected activation function.  For the one-dimensional case, $\sigma_j$ is a monotonic function such that:
\begin{equation*}
\begin{split}
    \text{ $\alpha_j>0$} & \Rightarrow \qquad \lim_{x\rightarrow +\infty} \sigma_j(x)=1, \qquad \lim_{x\rightarrow -\infty} \sigma_j(x)=0\\
    \text{ $\alpha_j<0$} & \Rightarrow \qquad \lim_{x\rightarrow +\infty} \sigma_j(x)=0, \qquad \lim_{x\rightarrow -\infty} \sigma_j(x)=1.
\end{split}
\end{equation*}
This function has a inflection point, that we call \emph{center} $c_j$ defined by the following property:
\begin{equation}
    \sigma_j(\alpha_j c_j + \beta_j)=\frac{1}{2}.
    \label{eq:cond_center}
\end{equation}
Now since $\sigma(0)=1/2$, the following relation between parameters holds:
\begin{equation*}
    c_j=-\frac{\beta_j}{\alpha_j}.
\end{equation*}
Finally, $\sigma_j$ has a steep transition that is governed by the amplitude of $\alpha_j$: if $|\alpha_j|\rightarrow +\infty$, then $\sigma_j$ approximates the Heaviside function, while if $|\alpha_j|\rightarrow 0$, then $\sigma_j$ becomes a constant function. Now, since in the ELM framework these parameters are fixed a priori, what one needs to avoid is to have some function that can be ``useless"\footnote{
In Huang \cite{Huang} it is suggested to take  in $I=[-1,1]$ the $\alpha_j$  randomly generated in the interval $[-1,1]$ and $\beta_j$ randomly generated in $[0, 1]$. This construction leads to functions that are not well suited for our purposes: ad example if $ \alpha_j=0.1$ and $\beta_j=0.9$, the center is $c_j=-9$. Moreover if $\alpha_j$ is small, the function $\phi_j$  is very similar to a  constant function in $[-1,1]$, therefore this function is useless for our purposes.} in the domain, say $I=[a,b]$.\\
Therefore, for the one-dimensional case, our suggestion is to chose $\alpha_{j}$ uniformly distributed as:
\begin{equation*}
    \alpha_{j} \sim \mathcal{U}\biggl(-\frac{N-55}{10|I|},\frac{N+35}{10|I|}\biggr),
    \label{eq:alphaSF1D}
\end{equation*}
where $N$ is the number of neurons in the the hidden layer and $|I|=b-a$ is the domain length.
Moreover, we also suggest to avoid too small in module coefficients $a_{j}$ by setting:
\begin{equation*}
    |\alpha_{j}|>\frac{1}{2|I|}.
    \label{eq:alphaSF1D_b}
\end{equation*}
Then, for the centers $c_j$, we select equispaced points in the domain $I$, that are given by imposing the $\beta_j$s to be:
\begin{equation*}
    \beta_j=-\alpha_j \cdot c_j.
    \label{eq:betaSF1D}
\end{equation*}
In the two-dimensional case, we denote as $\bm{x}=(x_1,x_2) \in \mathbb{R}^2$ the input and  $A \in \mathbb{R}^{N \times 2}$ the matrix with rows $\bm{\alpha}_j=(\alpha_{j,1},\alpha_{j,2})$. Then, the condition \eqref{eq:cond_center} becomes:
\begin{equation*}
    \sigma_j(x,y)=\sigma(\alpha_{j,1} x_1+\alpha_{j,2} x_2+\beta_j)=\frac{1}{2}
\end{equation*}
So, now we have:
\begin{equation*}
    s \equiv x_2=-\frac{\alpha_{j,1}}{\alpha_{j,2}}x_1-\frac{\beta_j}{\alpha_{j,2}},
\end{equation*}
where $s$ is a straight line of inflection points that we call \emph{central direction}. As the direction parallel to the central direction $\sigma_j$ is constant, while the orthogonal direction to $s$, the sigmoid $\sigma_j$ is exactly the one-dimensional logistic sigmoid.
So considering one point $\bm{c}_j=(c_{j,1},c_{j,2})$ of the straight line $s$, we get the following relation between parameters:
\begin{equation*}
    \beta_j=-\alpha_{j,1}\cdot c_{j,1}-\alpha_{j,2}\cdot c_{j,2}.
    \label{eq:betaSF2D}
\end{equation*}
Now, the difference with the one-dimensional case is the fact that in a domain $I^2 =[a,b]^2$ discretized by a grid of $n \times n$ points, the number of neurons $N=n^2$ grows quadratically, while the distance between two adjacent points decreases linearly, i.e. is given by $|I|/(n-1)$.
Thus, for the two-dimensional case, we take $\alpha_{j,k}$ uniformly distributed as:
\begin{equation*}
    \alpha_{j,k} \sim \mathcal{U}\biggl(-\frac{\sqrt{N}-60}{20|I|},\frac{\sqrt{N}+40}{20|I|}\biggr), \qquad k=1,2
    \label{eq:alphaSF2D}
\end{equation*}
where $N$ is the number of neuron in the network and $|I|=b-a$.

\subsection{ELM with radial basis functions}
Here, for the RBF case, we select the Gaussian kernel, that is defined as follows:
\begin{equation}
    \psi_j(\bm{x}) \equiv \varphi_j(\bm{x})=\text{exp}(-\varepsilon_j^2||\bm{x}-\bm{c}_j||_2^2)=\text{exp}\biggl(-\varepsilon_j^2 \sum_{k=1}^d(x_k-c_{j,k})^2\biggr),
    \label{eq:gaussianRBF}
\end{equation}
where $\bm{c}_j \in \mathbb{R}^d$ is the center point and $\varepsilon_j \in \mathbb{R}$ is the inverse of the standard deviation.
For such functions, we have:
\begin{equation}\label{eq:der:RBF}
\begin{split}
    \frac{\partial}{\partial x_k}\varphi_j(\bm{x})&=-2 \varepsilon_j^2(x_k-c_{j,k}) \text{exp}(-\varepsilon^2_j r_j^2),\\
    \frac{\partial^2}{\partial x_k^2}\varphi_j(\bm{x})&= -2 \varepsilon_j^2 (1-2\varepsilon^2_j (x_k-c_{j,k})^2)\text{exp}(-\varepsilon^2_j r_j^2),
\end{split}
\end{equation}
where $r_j=||\bm{x}-\bm{c}_j||_2$.
In all the directions, the Gaussian kernel is a classical bell function such that:
\begin{equation*}
    \begin{split}
        \lim_{\|\bm{x}-\bm{c}_j\| \rightarrow + \infty} \phi_j(\bm{x})=0, \qquad \phi_j(\bm{c}_j)=1.
    \end{split}
\end{equation*}
Moreover, the parameter $\varepsilon_j^2$ controls the steepness of the amplitude of the bell function: if $\varepsilon_j\rightarrow + \infty$, then $\phi_j$ approximates the Dirac function, while if $\varepsilon \rightarrow 0$, $\phi_j$ approximates a constant function. Thus, in the case of RBFs one can relate the role of $\varepsilon_j$  to the role of $\alpha_{j,k}$ for the case of SF. For RBFs, it is well known that the center has to be chosen as a point internal to the domain and also more preferable to be exactly a grid point, while the steepness parameter $\varepsilon$ is usually chosen to be the same for each function. Here, since we are embedding RBFs in the ELM framework, we take randomly the center $\bm{c}_j$ and the steepness parameter $\varepsilon_j$ in order to have more variability in the functional space. Thus, as for the SF case, we set the parameters $\varepsilon^2_j$ random uniformly distributed as:
\begin{equation*}
    \varepsilon_{j}^2 \sim \mathcal{U}\biggl(\frac{1}{|I|},\frac{N+65}{15|I|}\biggr),
    \label{eq:epsilonRBF1D}
\end{equation*}
where $N$ denotes the number of neurons in the hidden layer and $|I|=b-a$ is the domain length; for the centers $c_j$, we select equispaced points in the domain.
Besides, note that for the RBF case, it is trivial to extend the above into the multidimensional case, since $\varphi_j$ is already expressed with respect to the center. For the two-dimensional case, we do the same reasoning as for the SF taking:
\begin{equation*}
    \varepsilon_{j}^2 \sim \mathcal{U}\biggl(\frac{1}{2|I|},\frac{\sqrt{N}+50}{30|I|}\biggr).
    \label{eq:epsilonRBF2D}
\end{equation*}

\section{Numerical Bifurcation Analysis of Nonlinear Partial Differential Equations with Extreme Learning Machines}\label{Sect:sez2}

In this section, we introduce the general setting for the numerical solution and bifurcation analysis of nonlinear PDEs with ELMs based on basic numerical analysis concepts and tools (see e.g. \cite{brezzi1982finite,chan1982arc,cliffe2000numerical,glowinski1985continuation,quarteroni2008numerical}).
Let's start from a nonlinear PDE of the general form: 
\begin{equation}
L u= f(u,\lambda) \mbox{ in } \Omega,
\label{eq3}
\end{equation}
with boundary conditions:
\begin{equation}
B_l u =g_l, \mbox{ in } \partial \Omega_l\,,\ l=1,2,\dots,m \,,\ \label{eq4}
\end{equation}
where $L$ is the partial differential operator acting on $u$, $f(u,\lambda)$ is a nonlinear function of $u$ and ${\lambda} \in \mathbb{R}^{p}$ is the vector of model parameters, and $\{\partial \Omega_l\}_l$  denotes a partition of the boundary.

\par
A numerical solution $\tilde{u}=\tilde{u}(\lambda)$ to the above problem at particular values of the parameters $\lambda$  is typically found iteratively by applying e.g. Newton-Raphson or matrix-free Krylov-subspace methods (Newton-GMRES) (see e.g. \cite{kelley_2018}) on a finite system of $M$ nonlinear algebraic equations. In general, these equations reflect some zero residual condition, or exactness equation, and thus the numerical solution that is sought is the optimal solution with respect to the condition in the finite dimensional space. Assuming that $\tilde{u}$ is fixed via the degrees of freedom $\bm{w}\in \mathbf{R}^N$ - we  use the notation $\tilde{u}=\tilde{u}(\mathbf{w}) $ - then these degrees of freedom are sought by solving: 
\begin{equation}
F_{k} (w_{1},w_{2}, \dots w_{j} \dots w_{N};\lambda) =0\,,\ k=1,2,...M \ .
\label{Res_Eq}
\end{equation}
Many methods for the numerical solution of Eq. (\ref{eq3}), \eqref{eq4} are written in the above  form after the application of an approximation and discretization technique such as Finite Differences (FD), Finite Elements (FE) and Spectral Expansion (SE), as we detail next. 
\\
The system of $M$ algebraic equations (\ref{Res_Eq}) is solved iteratively (e.g. by Newton's method), that is by solving until a convergence criterion is satisfied, the following linearized system:
\begin{equation}
    \nabla_{\bm{w}} \bm{F}(\bm{w}^{(n)},\lambda) \cdot d\bm{w}^{(n)}=-\bm{F}(\bm{w}^{(n)},\lambda), \quad \bm{w}^{(n+1)}=\bm{w}^{(n)}+d\bm{w}^{(n)}.
    \label{ELMiter}
\end{equation}
$\nabla_{\bm{w}} \bm{F}$ is the Jacobian matrix:
\begin{equation}
\nabla_{\bm{w}} \bm{F} (\bm{w}^{(n)},\lambda)= \left[\frac{\partial F_k}{\partial w_j}\right]_{\big|({\bm{w}^{(n)},\lambda)}}=
\begin{bmatrix}
\frac{\partial F_1}{\partial w_1} & \frac{\partial F_1}{\partial w_2} & \dots & \frac{\partial F_1}{\partial w_j} & \dots &\partial\frac{F_1}{\partial w_N}\\
\frac{\partial F_2}{\partial w_1} & \frac{\partial F_2}{\partial w_2} & \dots & \frac{\partial F_2}{\partial w_j} & \dots &\frac{\partial F_2}{\partial w_N}\\
\vdots  &\vdots & \ddots & \vdots & \ddots &\vdots\\
\frac{\partial F_k}{\partial w_1} & \frac{\partial F_k}{\partial w_2} & \dots & \frac{\partial F_k}{\partial w_j} & \dots &\frac{\partial F_k}{\partial w_N}\\
\vdots  &\vdots & \ddots & \vdots & \ddots &\vdots\\
\frac{\partial F_M}{\partial w_1} & \frac{\partial F_M}{\partial w_2} & \dots & \frac{\partial F_M}{\partial w_j} & \dots &\frac{\partial F_M}{\partial w_N}
\end{bmatrix}_{\big|{(\bm{w}^{(n)},\lambda)}}
\label{jacelm}
\end{equation}
If the system is not square (i.e. when $M\neq N$), then at each iteration, one would perform e.g. QR-factorization of the Jacobian matrix
\begin{equation}
    \nabla_{\bm{w}} \bm{F} (\bm{w}^{(n)},\lambda)=R^T Q^T=\begin{bmatrix}
    R_{1}^T & 0
    \end{bmatrix}
    \begin{bmatrix} Q_1^T \\ Q_2^T \end{bmatrix},
    \label{eq:QR_decomposition}
\end{equation}
where $Q \in \mathbb{R}^{N\times N}$ is an orthogonal matrix and $R \in \mathbb{R}^{N \times M}$ is an upper triangular matrix.
Then, the solution of Eq.(\ref{ELMiter}) is given by:
\begin{equation*}
    d\bm{w}^{(n)}=-Q_1 R_1^{-1} \cdot \bm{F} (\bm{w}^{(n)},\lambda).
\end{equation*}
Branches of solutions in the parameter space past critical points on which the Jacobian matrix $\nabla F$ with elements $\dfrac{\partial F_k}{\partial w_{j}}$  becomes singular can be traced with the aid of numerical bifurcation analysis theory (see e.g. \cite{dhooge2008new,doedel2012numerical,doedel2007auto,govaerts2000numerical,krauskopf2007numerical,kuznetsov2013elements,schilder2017continuation}).
For example, solution branches past saddle-node bifurcations (limit/turning points) can be traced by applying the so called ``pseudo" arc-length continuation method \cite{chan1982arc}. This involves the parametrization of both $\tilde{u}(\bm{w})$ and $\lambda$ by the arc-length $s$ on the solution branch. The solution is sought in terms of both $\tilde{u}(\bm{w};s)$ and $\lambda(s)$ in an iterative manner, by solving until convergence the following augmented system:
\begin{equation}
\begin{bmatrix}
\nabla_{\bm{w}} \bm{F} & \nabla_{\lambda }\bm{F}\\
\nabla_{\bm{w}} {N} & \nabla_{\lambda} {N}
\end{bmatrix} \cdot\begin{bmatrix}d\bm{w}^{(n)}(s)\\d\lambda^{(n)}(s) \end{bmatrix} =-\begin{bmatrix}\bm{F}(\bm{w}^{(n)}(s),\lambda(s))\\N(\tilde{u}(\bm{w}^{(n)};s),\lambda^{(n)}(s))\end{bmatrix},
\label{augmentedarclength}
\end{equation}
where 
\begin{equation*}
\nabla_{\lambda }\bm{F}=\begin{bmatrix}
\frac{\partial F_1}{\partial \lambda} & \frac{\partial F_2}{\partial \lambda} & \dots & \frac{F_M}{\partial \lambda}
\end{bmatrix}^T,
\end{equation*}
and
\begin{equation*}
    \begin{split}
N(\tilde{u}(\bm{w}^{(n)};s),\lambda^{(n)}(s))= & \\ (\tilde{u}(\bm{w}^{(n)};s)-&\tilde{u}(\bm{w};s)_{-2})^T\cdot \frac{(\tilde{u}(\bm{w})_{-2}-\tilde{u}(\bm{w})_{-1})}{ds}+\\ (\lambda^{(n)}(s)-\lambda_{-1})&\cdot \frac{(\lambda_{-2}-\lambda_{-1})}{ds}-ds,
    \end{split}
\end{equation*}
is one of the choices for the so-called ``pseudo arc-length condition" (for more details see e.g. \cite{chan1982arc,doedel2012numerical,glowinski1985continuation,govaerts2000numerical,kuznetsov2013elements}); $\tilde{u}(\bm{w})_{-2}$ and  $\tilde{u}(\bm{w})_{-1}$ are two already found consequent solutions for $\lambda_{-2}$ and $\lambda_{-1}$, respectively and $ds$ is the arc-length step for which a new solution around the previous solution $(\tilde{u}(\bm{w})_{-2},\lambda_{-2})$ along the arc-length of the solution branch is being sought. 

\subsection{Finite Differences and Finite Elements cases: the application of Newton's method}
In FD methods, one aims to find the values of the solution per se (i.e. $u_{j}= w_{j}$) at a finite number of nodes within the domain. The operator in the differential problem \eqref{eq3} and the boundary conditions \eqref{eq4} are approximated by means of some finite difference operator: $L^h\approx L\,;\ B^h_l\approx B_l$: the finite operator revels in some linear combination of the function evaluations for the differential part, while keeping non-linear requirement to be satisfied due to the presence of nonlinearity. Then, approximated equations are collocated in internal and boundary points $\bm{x}_k$ giving equations that can be written as residual equations \eqref{Res_Eq}. \par
With FE and SE methods, the aim is to find the coefficients of a properly chosen basis function expansion of the solution within the domain such that the boundary conditions are satisfied precisely. 
In the Galerkin-FEM with Lagrangian basis (see e.g. \cite{olson1991efficient,quarteroni2008numerical}), the discrete counterpart seeks for a solution of Eq. \eqref{eq3}-\eqref{eq4} in $N$ points $x_j$ of the domain $\Omega$ according to:
\begin{equation}
u=\sum_{j=1}^{N} w_{j} \phi_{j},
\label{eq5}
\end{equation}
where the basis functions $\phi_{j}$ are defined so that they satisfy the completeness requirement and are such that $\phi_{j}(x_k)=\delta_{jk}$. This, again with the choice of nodal variables to be the function approximation at the points, gives that $u(x_j)=w_j$ are exactly the degrees of freedom for the method.  
The scheme can be written as the satisfaction of the zero for the weighted residuals $R_k, k=1,2,\dots N$ defined as:
\begin{equation}
R_k=\int_{\Omega} (L u- f(u,\lambda)) \phi_k\, d\Omega +\sum_{l=1}^{m} \int_{\partial \Omega_k} (B_k u -g_l) \phi_l\, d\sigma
\label{eq6}
\end{equation}
where the weighting functions $\phi_i$ are the same basis functions used in Eq. \eqref{eq5} for the approximation of $u$.
The above constitutes a nonlinear system of $N$ algebraic equations that for a given set of values for $\lambda$ are solved by Newton-Raphson, thus solving until convergence the following linearized system seen in equation \eqref{ELMiter}, where $R_k$ plays the role of $F_k$.

Notice that the border rows and columns of the Jacobian matrix (\ref{jacelm}) are appropriately changed so that Eq. \eqref{ELMiter} satisfy the boundary conditions. Due to the construction of the basis functions, the Jacobian matrix is sparse, thus allowing the significant reduction of the computation cost for the solution of \eqref{ELMiter} at each Newton's iteration.

\subsection{Extreme Learning Machine Collocation: the application of Newton's method}
In an analogous manner to FE methods, Extreme Learning Machines aim at solving the problem \eqref{eq3}-\eqref{eq4}, using an approximation $\tilde{u}_N$ of $u$ with $N$ neurons as an ansatz. The difference is that, similarly to FD methods, the equations are constructed by collocating the solution on $M_{\Omega}$ points $x_i \in \Omega$ and $M_l$ points $x_k \in \partial \Omega_l$, where $\Omega_l$ are the parts of the boundary where boundary conditions are posed, see e.g. \cite{auricchio2012isogeometric,quarteroni2008numerical}:
\begin{equation*}
\begin{split}
    L\tilde{u}_N(\bm{x}_i;\bm{w})&=f(\tilde{u}_N(\bm{x}_i;\bm{w}),\lambda), \quad i=1,\dots,M_{\Omega}\\
    B_l\tilde{u}_N(\bm{x}_k;\bm{w})&=g_l(\bm{x}_k), \quad k=1,\dots,M_{l}, \quad l=1,\dots,m.
\end{split}
\end{equation*}
Then, if we denote $M=M_{\Omega}+\sum_{l=1}^m M_l$, we have a system of $M$ nonlinear equations with $N$ unknowns that can be rewritten in a compact way as:
\begin{equation*}
    F_k(\bm{w},\lambda)=0, \quad k=1,\dots,M
\end{equation*}
where for $k=1,\dots,M_{\Omega}$, we have:
\begin{equation*}
    F_k(\bm{w},\lambda)=L\biggl( \sum_{i=1}^N w_j \psi(\bm{\alpha}_j \cdot \bm{x}_i+\beta_j)\biggr)-f\biggl(\sum_{i=1}^N w_j \psi(\bm{\alpha}_j \cdot \bm{x}_i+\beta_j) \biggr)=0,
\end{equation*}
while for the $l$-th boundary condition, for $k=1,\dots,M_l$ we have:
\begin{equation*}
    F_k(\bm{w},\lambda)=B_l\biggl( \sum_{i=1}^N w_j \psi(\bm{\alpha}_j \cdot \bm{x}_i+\beta_j)\biggr)-g\biggl(\sum_{i=1}^N w_j \psi(\bm{\alpha}_j \cdot \bm{x}_i+\beta_j) \biggr)=0.
    \label{ELMalgebraic}
\end{equation*}
At this system of non-linear algebraic equations, here we apply Newton's method \eqref{ELMiter}. Notice that the application of the method requires the explicit knowledge of the derivatives of the functions $\psi$; in the ELM case as described, we have explicit formulae for these (see Eq. \eqref{eq:der:SF}, \eqref{eq:der:RBF}).
\begin{remark}
In our case, Newton's method is applied to non-squared systems. When the rank of the Jacobian is small, here we have chosen to solve the problem with the use of Moore–Penrose pseudo inverse of $\nabla_{\bm{w}} F$ computed by the SVD decomposition; as discussed above, another choice would be $QR$-decomposition (\ref{eq:QR_decomposition}).  This means that we cut off all the eigenvectors correlated to small eigenvalues\footnote{The usual algorithm implemented in Matlab is that any singular value less than a tolerance is treated as zero: by default, this tolerance is set to max(size($A$)) * eps(norm($A$))}, so:
\begin{equation*}
    \nabla_{\bm{w}}\bm{F}=U\Sigma V^T, \qquad (\nabla_{\bm{w}}\bm{F})^{+}=V\Sigma^{+}U^T,
\end{equation*}
where $U \in \mathbb{R}^{M\times M}$ and $V \in \mathbb{R}^{N\times N}$ are the unitary matrices of left and right eigenvectors respectively, and $\Sigma \in \mathbb{R}^{M \times N}$ is the diagonal matrix of singular values. Finally, we can select $q\le rank(\nabla F)$ to get
\begin{equation}
    \nabla_{\bm{w}}\bm{F}=U_q\Sigma_q V_q^T, \qquad (\nabla_{\bm{w}}\bm{F})^{+}=V_q\Sigma_q^{+}U_q^T,
    \label{eq:Moore-Penrose_PseudoInv}
\end{equation}
where $U_q \in \mathbb{R}^{M\times q}$ and $V \in \mathbb{R}^{N\times q}$ and $\Sigma_q \in \mathbb{R}^{q \times q}$.
Thus, the solution of Eq.(\ref{ELMiter}) is given by:
\begin{equation*}
    d\bm{w}^{(n)}=-V_q\Sigma_q^+ U_q^T \cdot \bm{F} (\bm{w}^{(n)},\lambda).
\end{equation*}
Branches of solutions past turning points can be traced by solving the augmented, with the pseudo-arc-length condition, problem given by Eq.(\ref{augmentedarclength}). In particular in (\ref{augmentedarclength}), for the ELM framework (\ref{eq:ELMframework}), the term $\nabla_{\bm{w}}N$ becomes:
\begin{equation*}
\nabla_{\bm{w} }{N}=\bm{S}^T\frac{(\tilde{u}(\bm{w})_{-2}-\tilde{u}(\bm{w})_{-1})}{ds},
\end{equation*}
where $\bm{S}$ is the collocation matrix defined in equation \eqref{ESSE}.


\end{remark}

\begin{remark}
The three numerical methods (FD, FEM and ELM) are compared with respect to the dimension of the Jacobian matrix $J$, that in the case of FD and FEM is square and related to the number $N$ of nodes, i.e. $J \in \mathbb{R}^{N \times N}$, and in the case ELM is rectangular and related to both the number $M$ of collocation nodes and the number $N$ of neurons, i.e. $J \in \mathbb{R}^{M\times N}$. Actually, $N$ is the parameter related to the computational cost, i.e. the inversion of the $J$ is $O(N^3)$ and the same is in the ELM case for the inversion of the matrix $J^T J \in \mathbb{R}^{N\times N}$. Finally we make explicit that in all the rest of this work, for the ELM case, we use a number $M$ of collocation points that is half the number $N$ of neurons. Such a choice is justified by our previous work (\cite{calabro2020extreme}) that works better for linear PDEs with steep gradients. In general, we pinpoint that by increasing the number $M$ to be $\frac{2N}{3}, \frac{3N}{4}, etc..$\footnote{The case $M=N$ can be solved only by the use of a (Moore-Penrose) pseudo-inverse (\ref{eq:Moore-Penrose_PseudoInv}), because the invertibility of the Jacobian of the nonlinear PDE operator cannot be guaranteed in advance.} one gets even better results (see e.g. our previous work \cite{calabro2020extreme} on the solution of linear PDEs).
\end{remark}

\section{Numerical Analysis Results: the Case Studies}
The efficiency of the proposed numerical scheme is demonstrated through two benchmark nonlinear PDEs, namely (a) the one dimensional nonlinear Burgers equation with Dirichlet boundary conditions and also mixed boundary conditions, and, (b) the one- and two-dimensional Liouville–Bratu–Gelfand problem. These problems have been widely studied as have been used to model and analyse the behaviour of many physical and chemical systems (see e.g. \cite{allen2013numerical,boyd1986analytical,chan1982arc,glowinski1985continuation,hajipour2018accurate,iqbal2020numerical,raja2013neural}).\\ 
In this section, we present some known properties of the proposed problems and provide details on their numerical solution with FD, FEM and ELM with both logistic and Gaussian RBF transfer functions. 

\subsection{The Nonlinear Viscous Burgers Equation}
Here, we consider the one-dimensional steady state viscous Burgers problem:
\begin{equation}
    \nu\frac{\partial^2u}{\partial x^2}-u \frac{\partial u}{\partial x}=0
    \label{eqnburgers}
\end{equation}
in the unit interval $[0,1]$, where $\nu>0 $ denotes the viscosity. For our analysis, we considered two different sets of boundary conditions:
\begin{itemize}
    \item Dirichlet boundary conditions
\begin{equation}
    u(0)=\gamma\,,\ u(1)=0\,,\ \gamma>0 \ ;
    \label{eqnburgers_Dirichlet}
\end{equation}
\item Mixed boundary conditions: Neumann condition on the left boundary and zero Dirichlet on the right boundary:
\begin{equation}
   \frac{\partial u}{\partial x}(0)=-\vartheta\,,\ u(1)=0\,,\ \vartheta>0 \ .
   \label{eqnboundaryburgers}
\end{equation}
\end{itemize}
The two sets of boundary conditions result to  different behaviours (see \cite{allen2013numerical,benton1972table}).
We summarize in the next two lemmas some of the main results.
\begin{lemma}[Dirichlet case]\label{Lemma1} Consider Eq. (\ref{eqnburgers}) with boundary conditions given by \eqref{eqnburgers_Dirichlet}. Moreover, take (notice that $\gamma \xrightarrow[\nu\rightarrow 0]{}1 $):
\begin{equation*}
    \gamma=\frac{2}{1+\text{exp}(\frac{-1}{\nu})}-1 .
\end{equation*}
\\
Then, the problem \eqref{eqnburgers}-\eqref{eqnburgers_Dirichlet} has a unique solution 
given by:
\begin{equation}\label{BurDir:ExactSol}
    u(x)=\frac{2}{1+\text{exp}(\frac{x-1}{\nu})}-1 \ .
\end{equation}
\end{lemma}
We will use this test problem because the solution has a boundary layer and for this simple case, we can also implement and discuss the efficiency of a fixed point iteration by linearization, while in the mixed-boundaries case, we implement only the Newton's iterative procedure. 
\begin{lemma}[Mixed case]\label{Lemma2_new} Consider Eq.(\ref{eqnburgers}) with boundary conditions given by \eqref{eqnboundaryburgers}. The solution of the problem can be written as \cite{allen2013numerical} :
\begin{equation}
    u(x)=\sqrt{2c}\tanh{\left(\dfrac{\sqrt{2c}}{2\nu}(1-x)\right)},
    \label{eqnsolburgers}
\end{equation}
where $c$ is constant value which can be determined by the imposed Neumann condition.\\ 
Then, for $\vartheta$ sufficiently small the viscous Burgers problem with mixed boundary conditions admits two solutions:
\begin{itemize}
    \item[(a)] a stable lower solution such that $\forall x \in (0, 1)$: \begin{equation*}
        u(x) \to_{\vartheta\to 0} 0\,,\quad \frac{\partial u(x)}{\partial x} \to_{\vartheta\to 0} 0\,;
    \end{equation*}  
\item[(b)] an unstable upper solution $u(x)>0\, \forall x \in (0, 1)$ such that: \begin{equation*}
\frac{\partial u(0)}{\partial x} \to_{\vartheta\to 0} 0\,,\quad \frac{\partial u(1)}{\partial x}\to_{\vartheta\to 0} -\infty \,,
    \end{equation*}
    and
    \begin{equation*}
        \forall x \in (0, 1)\,,\quad  u(x) \to_{\vartheta\to 0} \infty \ .
    \end{equation*} 
\end{itemize}
\end{lemma}
\begin{proof}
The spatial derivative of (\ref{eqnsolburgers}) is given by: 
\begin{equation}
 \frac{\partial u(x)}{\partial x}=-\frac{c}{\nu}\sech^2{\left(\dfrac{\sqrt{2c}}{2\nu}(1-x)\right)}.
 \label{eqnburgersderivative}
\end{equation}
(a) When $c\to0$ then from Eq.(\ref{eqnsolburgers}), we get asymptotically the zero solution, i.e. $u(x)\to 0$, $\forall x \in (0, 1)$ and from Eq.(\ref{eqnburgersderivative}), we get $\frac{\partial u(x)}{\partial x}\to 0$, $\forall x \in (0, 1)$. At $x=1$, the Dirichlet boundary condition $u(1)=0$ is satisfied exactly (see Eq.(\ref{eqnsolburgers})), while at the left boundary $x=0$  the Neumann boundary condition is also satisfied as due to Eq.(\ref{eqnburgersderivative}) and our assumption ($\vartheta \to 0$): $\frac{\partial u(0)}{\partial x}=-\vartheta \to 0$, when $c\to0$.\par
(b) When $\frac{\partial u(1)}{\partial x}\to -\infty$, then (\ref{eqnburgersderivative}) is satisfied $\forall x \in (0, 1)$ when $c \to \infty$. In that case, at $x=0$, the Neumann boundary condition is satisfied as due to Eq.(\ref{eqnburgersderivative}) is easy to prove that $\frac{\partial u(0)}{\partial x}\to 0$. \par
Indeed, from Eq.(\ref{eqnburgersderivative}):
\begin{equation}
 \lim_{c\to\infty} \frac{\partial u(x)}{\partial x}=-\lim_{c\to\infty} \frac{\nu}{\exp{\frac{\sqrt{2c}}{\nu}}}=0.
 \label{eqnburgerslimit}
\end{equation}
Finally Eq.(\ref{eqnsolburgers}) gives $u(x) \to \infty$, $\forall x \in (0, 1)$.
\end{proof}
To better understand the behaviour of the unstable solution with respect to the left boundary condition, we can prove the following.
\begin{corollary}\label{Lemma2} Consider Eq.(\ref{eqnburgers}) with boundary conditions given by \eqref{eqnboundaryburgers}.
For the non-zero solution, when $\vartheta=\epsilon\to 0$ the solution at $x=0$ goes to infinity with values:
 \begin{equation}
    u(0)=\nu \log\left(\frac{\nu}{\epsilon}\right)\tanh\left(\frac{1}{2}\log \left(\frac{\nu}{\epsilon}\right)\right).
    \label{eqnsolburgersinfty}
\end{equation}
\end{corollary}
\begin{proof}
By setting the value of $\vartheta$ in the Neumann boundary condition to be a very small number, i.e. $\vartheta=\epsilon\ll 1$, then from Eq.(\ref{eqnburgerslimit}), we get that the slope of the analytical solution given by Eq.(\ref{eqnburgersderivative}) is equal to $\epsilon$, when 
\begin{equation}
 c=\frac{1}{2}\nu^2 \log^2 \left(\frac{\nu}{\epsilon}\right).
 \label{eqnburgerscoef}
 \end{equation}
Plugging the above into the analytical solution given by Eq.(\ref{eqnsolburgers}), we get Eq.(\ref{eqnsolburgersinfty}).
\end{proof}
The above findings imply also the existence of a limit point bifurcation with respect to $\vartheta$ that depends also on the viscosity. For example, as shown in \cite{allen2013numerical}, for $\vartheta>0$ and $\nu=1/10$, there are two equilibria arising due to a turning point at $\vartheta^*=0.087845767978$.


\subsubsection{Numerical Solution of the Burgers equation with Finite Differences and Finite Elements}

The discretization of the one-dimensional viscous Burgers problem in $N$ points with second-order central finite differences in the unit interval $0\leq x \leq 1$ leads to the following system of $N-2$ algebraic equations $\forall x_j=(j-1)h, j=2,\dots N-1 $, $h=\frac{1}{N-1}$:
 \begin{equation*}
     F_j( \bm{u})=\frac{\nu}{h^2}(u_{j+1}-2u_{j}+u_{j-1})-u_{j}\frac{u_{j+1}-u_{j-1}}{2h}=0\ .
 \end{equation*}
At the boundaries $x_1=0, x_N=1$, we have $u_1=\gamma$, $u_N=0$, respectively for the Dirichlet boundary conditions (\ref{eqnburgers_Dirichlet}) and 
$u_1=(2h\vartheta+4u_2-u_3)/3$, $u_N=0$, respectively for the mixed boundary conditions (\ref{eqnboundaryburgers}).\par
The  above $N-2$ nonlinear algebraic equations are the residual equations \eqref{Res_Eq} that are solved iteratively using Newton's method \eqref{ELMiter}.
The Jacobian \eqref{jacelm} is now triagonal: at each $i$-th iteration, the non-null elements are given by: \begin{equation*}
\frac{\partial F_j}{\partial u_{j-1}}=\frac{\nu}{h^2}+\frac{u_j}{2h} \, ;\ 
    \frac{\partial F_j}{\partial u_j}=-\nu\frac{2}{h^2}-\frac{u_{j+1}-u_{j-1}}{2h} \, ;\ \frac{\partial F_j}{\partial u_{j+1}}=\frac{\nu}{h^2}-\frac{u_j}{2} \ .
\end{equation*}
The Galerkin residuals \eqref{eq6} in the case of the one-dimensional Burgers equation are:
\begin{equation}
R_k=\int_{0}^{1} \biggl(\nu\frac{\partial^2u(x)}{\partial x^2}-u \frac{\partial u(x)}{\partial x}\biggr) \phi_{k}(x) dx.
\label{burgersres}
\end{equation}
By inserting the numerical solution (\ref{eq5}) into Eq.(\ref{burgersres}) and by applying the Green’s formula for integration, we get:
\begin{equation}
\begin{split}
R_k=&\nu\phi_{k}(x)\frac{du}{dx}\Big|_{0}^{1}-\nu\sum_{j=1}^{N} u_j \int_{0}^{1} \frac{d \phi_{j}(x)}{dx}\frac{d\phi_{k}(x)}{dx} dx\\
&-\int_{0}^{1} \sum_{j=1}^{N} u_j \phi_{j}(x) \sum_{j=1}^{N} u_j \frac{d\phi_{j}(x)}{dx}\phi_{k}(x) dx.
\end{split}
\label{eqburgersres2}
\end{equation}
At the above residuals, we have to impose the boundary conditions. If Dirichlet boundary conditions (\ref{eqnburgers_Dirichlet}) are imposed,  Eq. (\ref{eqburgersres2}) becomes:
\begin{equation}
\begin{split}
R_k=& -\nu\sum_{j=1}^{N} u_j \int_{0}^{1} \frac{d \phi_{j}(x)}{dx}\frac{d\phi_{k}(x)}{dx} dx\\&-\int_{0}^{1} \sum_{j=1}^{N} u_j \phi_{j}(x) \sum_{j=1}^{N} u_j \frac{d\phi_{j}(x)}{dx}\phi_{k}(x) dx.
\end{split}
\label{eqburgersresdir}
\end{equation}
In the case of the mixed boundary conditions (\ref{eqnboundaryburgers}), Eq.(\ref{eqburgersres2}) becomes:
\begin{equation}
\begin{split}
R_k=&\nu\vartheta\phi_{k}(0) -\nu\sum_{j=1}^{N} u_j \int_{0}^{1} \frac{d \phi_{j}(x)}{dx}\frac{d\phi_{k}(x)}{dx} dx\\
&-\int_{0}^{1} \sum_{j=1}^{N} u_j \phi_{j}(x) \sum_{j=1}^{N} u_j \frac{d\phi_{j}(x)}{dx}\phi_{k}(x) dx.
\end{split}
\label{eqburgersneumann}
\end{equation}
In this paper, we use a $P^2$ Finite Element space, thus quadratic basis functions using an affine element mapping in the interval $[0, 1]^d$. For the computation of the integrals, we used the Gauss quadrature numerical scheme: for the one-dimensional case, we used the three-points gaussian rule: \begin{equation*}
    \left\{ \left(\frac{1}{2}-\sqrt{\frac{3}{20}}, \frac{5}{18}\right) , \left(0.5, \frac{8}{18}\right), \left(\frac{1}{2}+\sqrt{\frac{3}{20}},\frac{5}{18}\right) \right\} \ .
\end{equation*}
When writing Newton's method \eqref{ELMiter}, the elements of the Jacobian matrix for both \eqref{eqburgersresdir} and \eqref{eqburgersneumann} are given by:
\begin{equation}
\frac{\partial R_i}{\partial u_j}=- \nu\int_{0}^{1} \frac{d \phi_{j}(x)}{dx}\frac{d\phi_{k}(x)}{dx} dx-2   \int_{0}^{1} \sum_{j=1}^{N} u_j \phi_{j}(x) \frac{d\phi_{j}(x)}{dx}\phi_{k}(x) dx.
\label{eqburgersjac}
\end{equation}
Finally, with all the above, the Newton's method \eqref{ELMiter} involves the iterative solution of a linear system. For the Dirichlet problem this becomes:
\begin{equation}
\begin{bmatrix}
1 & 0 & \dots & 0 & \dots & 0\\
\frac{\partial R_2}{\partial u_1} & \frac{\partial R_2}{\partial u_2} & \dots & \frac{\partial R_2}{\partial u_j} & \dots &\frac{\partial R_2}{\partial u_N}\\
\vdots  &\vdots & \ddots & \vdots & \ddots &\vdots\\
\frac{\partial R_k}{\partial u_1} & \frac{\partial R_k}{\partial u_2} & \dots & \frac{\partial R_k}{\partial u_j} & \dots &\frac{\partial R_k}{\partial u_N}\\
\vdots  &\vdots & \ddots & \vdots & \ddots &\vdots\\
0 & 0 & \dots & 0 & \dots & 1
\end{bmatrix}_{\big|{u^{(n)}}} \cdot
\begin{bmatrix}
du^{(n)}_{1}\\
du^{(n)}_{2}\\
\vdots\\
du^{(n)}_{j}\\
\vdots\\
du^{(n)}_{N}
\end{bmatrix}=-
\begin{bmatrix}
0\\
R_{2}\\
\vdots\\
R_{k}\\
\vdots\\
0
\end{bmatrix}_{\big|{u^{(n)}}} \ ,
\label{eqnnewton}
\end{equation}
while for the problem with the mixed boundary conditions, at each iteration, we need to solve the following system:
\begin{equation}
\begin{bmatrix}
\frac{\partial R_1}{\partial u_1}  & \frac{\partial R_1}{\partial u_2}  & \dots & \frac{\partial R_1}{\partial u{j}}  & \dots & \frac{\partial R_1}{\partial u_N} \\
\frac{\partial R_2}{\partial u_1} & \frac{\partial R_2}{\partial u_2} & \dots & \frac{\partial R_2}{\partial u_j} & \dots &\frac{\partial R_2}{\partial u_N}\\
\vdots  &\vdots & \ddots & \vdots & \ddots &\vdots\\
\frac{\partial R_k}{\partial u_1} & \frac{\partial R_k}{\partial u_2} & \dots & \frac{\partial R_k}{\partial u_j} & \dots &\frac{\partial R_k}{\partial u_N}\\
\vdots  &\vdots & \ddots & \vdots & \ddots &\vdots\\
0 & 0 & \dots & 0 & \dots & 1
\end{bmatrix}_{\big|{u^{(n)}}} \cdot
\begin{bmatrix}
du^{(n)}_{1}\\
du^{(n)}_{2}\\
\vdots\\
du^{(n)}_{j}\\
\vdots\\
du^{(n)}_{N}
\end{bmatrix}=-
\begin{bmatrix}
R_{1}\\
R_{2}\\
\vdots\\
R_{k}\\
\vdots\\
0
\end{bmatrix}_{\big|{u^{(n)}}}\ .
\end{equation}

\subsubsection{Numerical Solution of the Burgers equation with Extreme Learning Machine Collocation}
Collocating the ELM network function for the one-dimensional Burgers equation leads to the following nonlinear algebraic system for $i=2,\dots, M-1$:
\begin{equation}
    F_i(\bm{w},\nu)=\nu\sum_{j=1}^Nw_j\alpha_j^2\psi_j''(x_i)-\biggl(\sum_{j=1}^Nw_j\psi_j(x_i)\biggr)\cdot\biggl( \sum_{j=1}^Nw_j \alpha_j \psi_j'(x_i)\biggr)=0 \ .
    \label{eq:ELMburgersequation}
\end{equation}
Then, the imposition of the boundary conditions \eqref{eqnburgers_Dirichlet} gives:
\begin{equation}
    F_1(\bm{w},\nu)=\sum_{j=1}^N w_j \psi_j(0)-\gamma=0, \qquad F_M(\bm{w},\nu)=\sum_{j=1}^N w_j \psi_j(1)=0 \ ,
    \label{eq:ELM_burg_dirichlet}
\end{equation}
while  boundary conditions \eqref{eqnboundaryburgers} lead to:
\begin{equation}
    F_1(\bm{w},\nu)=\sum_{j=1}^N w_j \alpha_j\psi_j'(0)+\vartheta=0, \qquad F_M(\bm{w},\nu)=\sum_{j=1}^N w_j \psi_j(1)=0 \ .
    \label{eq:ELM_burg_mixed}
\end{equation}
These equations are the residual equations \eqref{Res_Eq} that we solve by Newton's method \eqref{ELMiter}. The elements of the Jacobian matrix $\nabla_{\bm{w}}\bm{F}$ are:
\begin{equation*}
    \frac{\partial F_i}{\partial w_j}=\nu\alpha_{j}^2 \psi_j''(x_i)-\psi_j(x_i)\cdot \biggl( \sum_{j=1}^Nw_j \alpha_j \psi_j'(x_i)\biggr)-\biggl( \sum_{j=1}^Nw_j \psi_j(x_i)\biggr)\cdot \alpha_j\psi_j'(x_i)
\end{equation*}
for $ \quad i=2,\dots,M-1$ and due to the Dirichlet boundary conditions \eqref{eq:ELM_burg_dirichlet}, we have:
\begin{equation*}
    \frac{\partial F_1}{\partial w_j}(\bm{w},\lambda)=\psi_j(0) \qquad \frac{\partial F_M}{\partial w_j}(\bm{w},\lambda)=\psi_j(1).
\end{equation*}
On the other hand, due to the mixed boundary conditions given by \eqref{eq:ELM_burg_mixed}, we get:
\begin{equation*}
    \frac{\partial F_1}{\partial w_j}(\bm{w},\lambda)=\alpha_j\psi_j'(0) \qquad \frac{\partial F_M}{\partial w_j}(\bm{w},\lambda)=\psi_j(1).
\end{equation*}
At this point, the application of Newton's method \eqref{ELMiter} using the exact computation of the derivatives of the basis functions is straightforward  (see \eqref{eq:der:SF} and \eqref{eq:der:RBF}). 

\subsubsection{Numerical Results}
In all the computations with FD, FEM and ELMs, the convergence criterion for Newton's iterations was the $L_2$ \footnote{The relative error is the $L_2$--norm of the difference between two successive solutions $||u(\bm{w})_{-2}-u(\bm{w})_{-1}||_2$. In particular for the ELM framework is given by $||S^T \cdot (\bm{w}_{-2}-\bm{w}_{-1})||_2$, where $S$ is the collocation matrix defined in eq. \eqref{ESSE}.} norm of the relative error between the solutions resulting from successive iterations; the convergence tolerance was set to $10^{-6}$. In fact, for all methods, Newton's method converged quadratically also up to the order of $10^{-10}$, when the bifurcation parameter was not close to zero where the solution of both Burgers with mixed boundary conditions and Bratu problems goes asymptotically to infinity. The exact solutions that are available for the one-dimensional Burgers and Bratu problems are derived using Newton's method with a convergence tolerance of $10^{-12}$.\par
First, we present the numerical results for the Burgers equation \eqref{eqnburgers} with Dirichlet boundary conditions \eqref{eqnburgers_Dirichlet}.  Recall that for this case, the exact solution is available (see equation \eqref{BurDir:ExactSol}). For our illustrations, we have selected two different values for the viscosity, namely $\nu=0.1$ and $\nu=0.007$.
Results were obtained with Newton's iterations starting from an initial guess that is a linear segment that satisfies the boundary conditions.
Figure \ref{fig:burg_dirichlet} shows the corresponding computed solutions for a fized size $N=40$ as well as the relative errors with respect to the exact solution. As it is shon the proposed ELM scheme outperforms both the FD and FEM schemes for medium to large sizes of the grid; from low to medium sizes of the grid, all methods perform equivalently. However, as shown in Figure \ref{fig:burg_dirichlet}(c), for $\nu=0.007$, and the particualr choice of the size ($N=40$), the FD scheme fails to approximate sufficiently the steep-gradient appearing at the right boundary.\par
\begin{figure}[th]
    \centering
    \subfigure[]{
   \includegraphics[width=0.45 \textwidth]{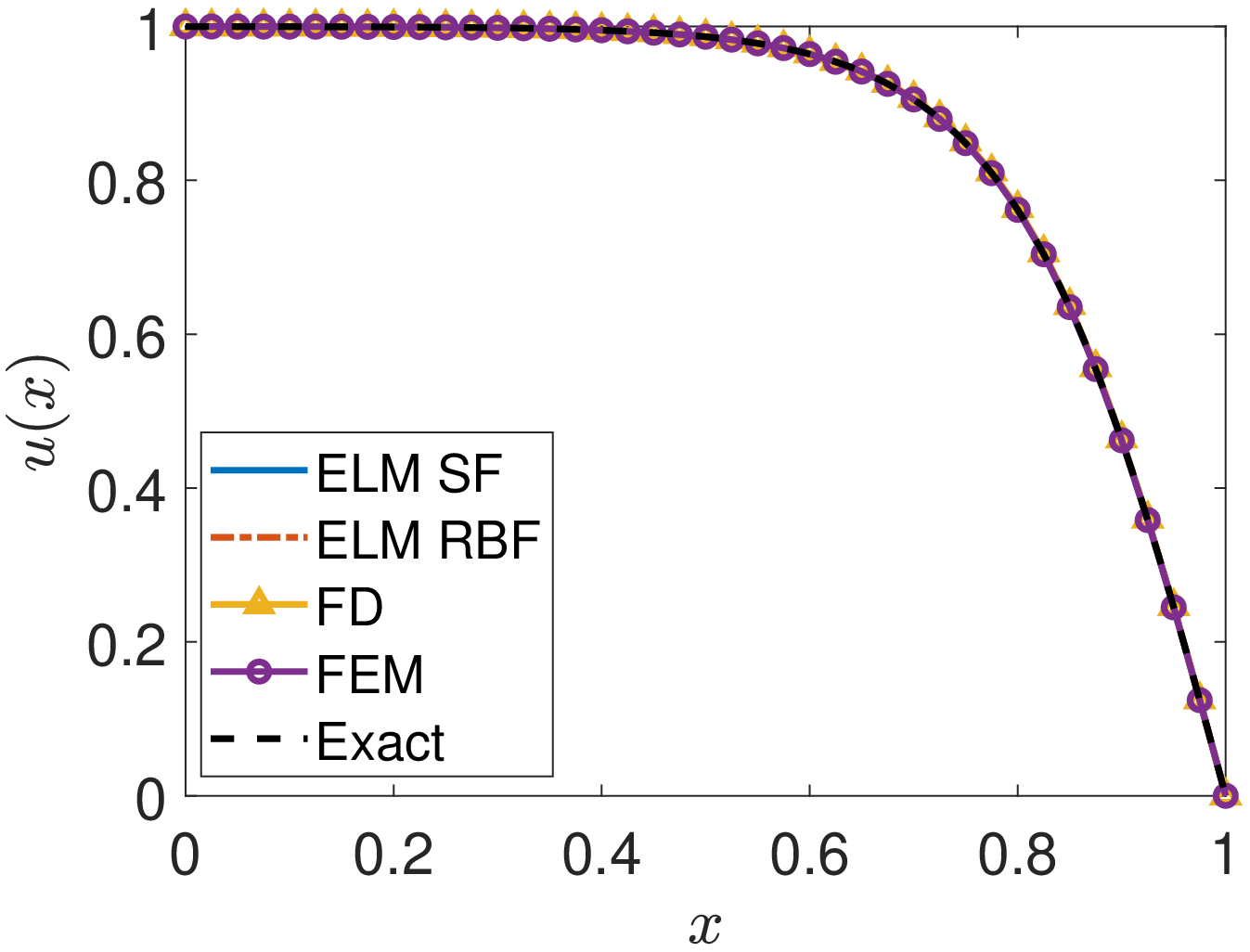}
}
~\subfigure[]{
    \includegraphics[width=0.45 \textwidth]{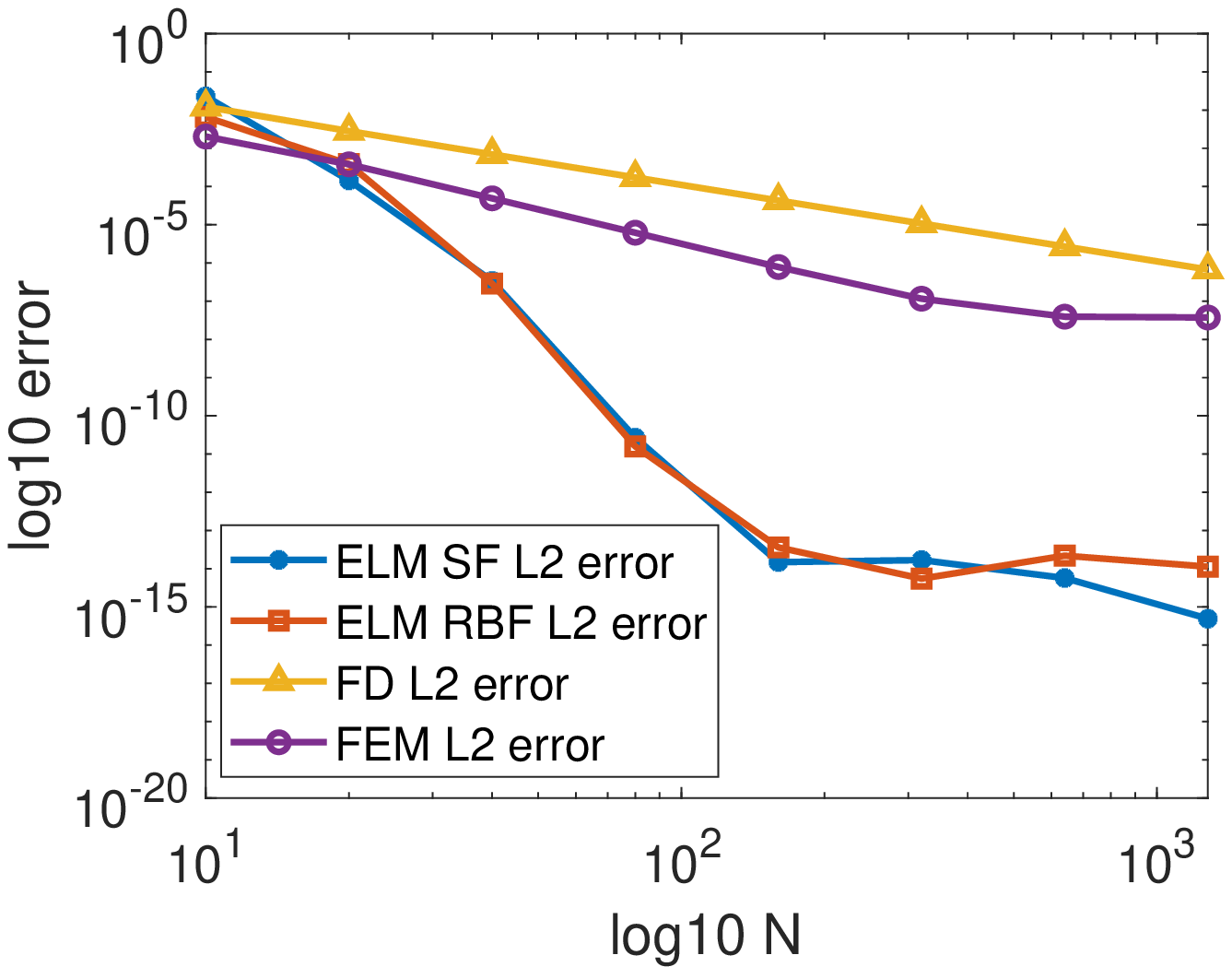}
}
~\subfigure[]{
    \includegraphics[width=0.45 \textwidth]{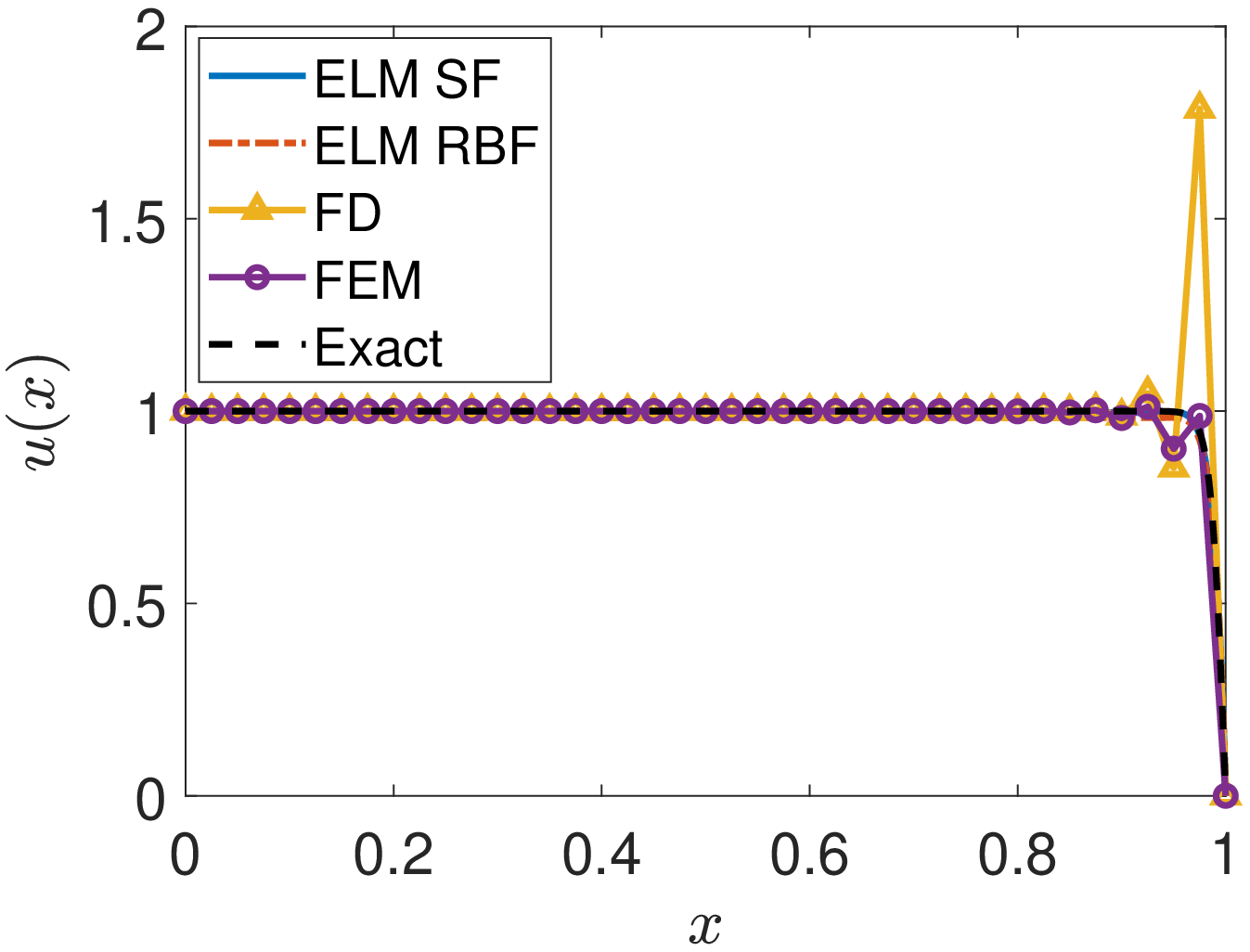}
}
~\subfigure[]{
    \includegraphics[width=0.45 \textwidth]{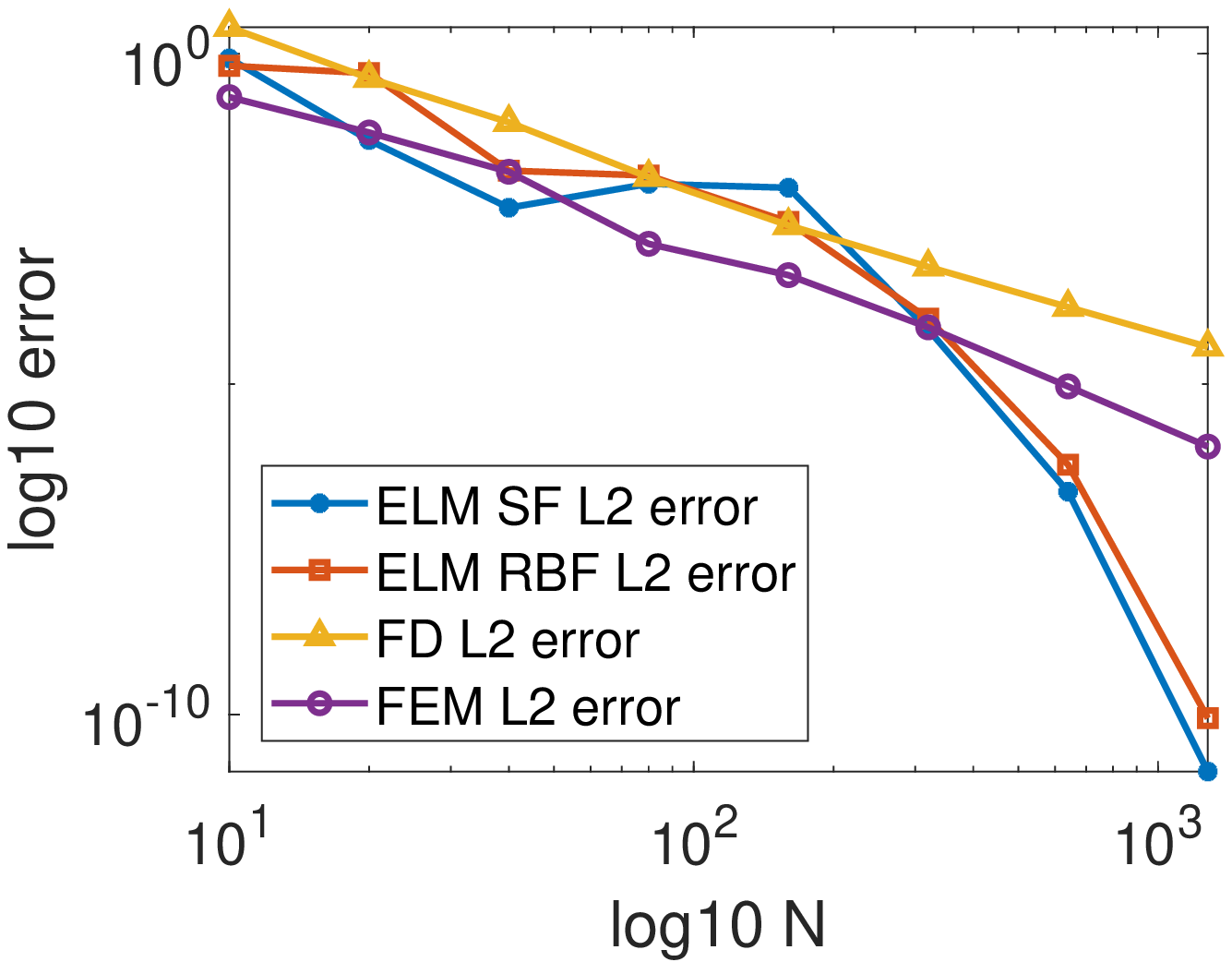}
}
\caption{Numerical solution and accuracy of the FD, FEM and ELM schemes for the one-dimensional viscous Burgers problem with Dirichlet boundary conditions (\ref{eqnburgers}), (\ref{eqnburgers_Dirichlet}), (a,b) with viscosity $\nu=0.1$: (a) Solutions for a fixed problem size $N=40$; (b) $L_2$--norm of differences with respect to the exact solution \eqref{BurDir:ExactSol} for various problem sizes. (c,d) with viscosity $\nu=0.007$: (c) Solutions for a fixed problem size $N=40$; (d) $L_2$--norm errors with respect to the exact solution for various problem sizes.}
\label{fig:burg_dirichlet}
\end{figure}
Then, we considered the case of the non-homogeneous Neumann condition on the left boundary \eqref{eqnburgers}- \eqref{eqnboundaryburgers}; here, we have set $\nu=1/10$. In this case, the solution is not unique and the resulting bifurcation diagram obtained with FD, FEM and ELM is depicted in Fig.(\ref{figbifdiagburgers}). In Table 1, we report the error between the value of the bifurcation point as computed with FD, FEM and ELM for various problem sizes $N$, with respect to the exact value of the bifurcation point (occurring for the particular choice of viscosity at $\vartheta^*=0.087845767978$). The location of the bifurcation point for all numerical methods was estimated by fitting a parabola around the four points (two on the lower and two on the upper branch) of the largest values of $\lambda$ as obtained by the pseudo-arc-length continuation. As shown, the proposed ELM scheme performs equivalently to FEM for low to medium sized of the grid, thus outperforming FEM for medium to large grid sizes; both methods FEM and ELM) outperform FD for all sizes of the grid.
\FloatBarrier
\begin{figure}[th]
    \centering
    \subfigure[]{
    \includegraphics[width=0.45 \textwidth]{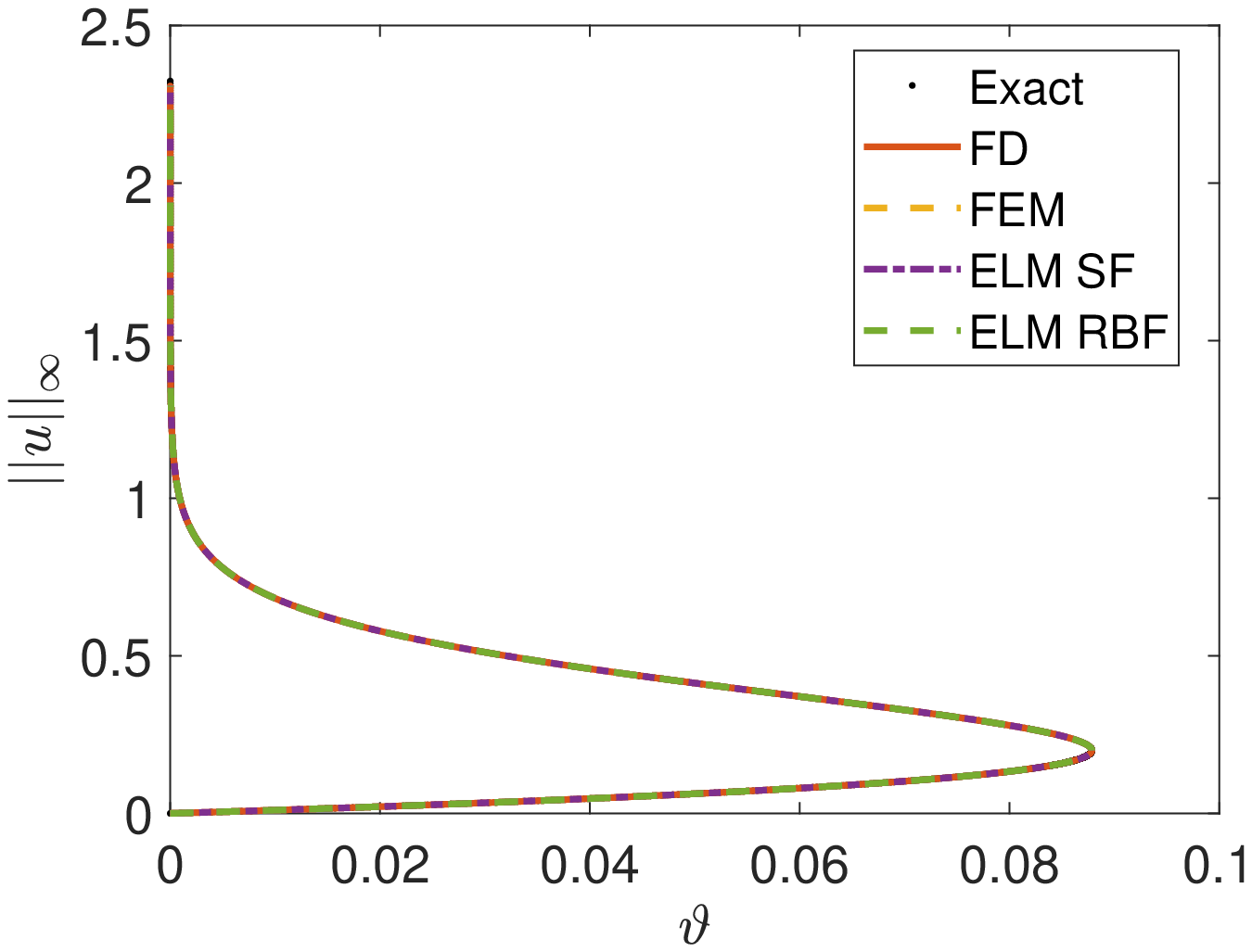}
    }
    \subfigure[]{
    \includegraphics[width=0.45 \textwidth]{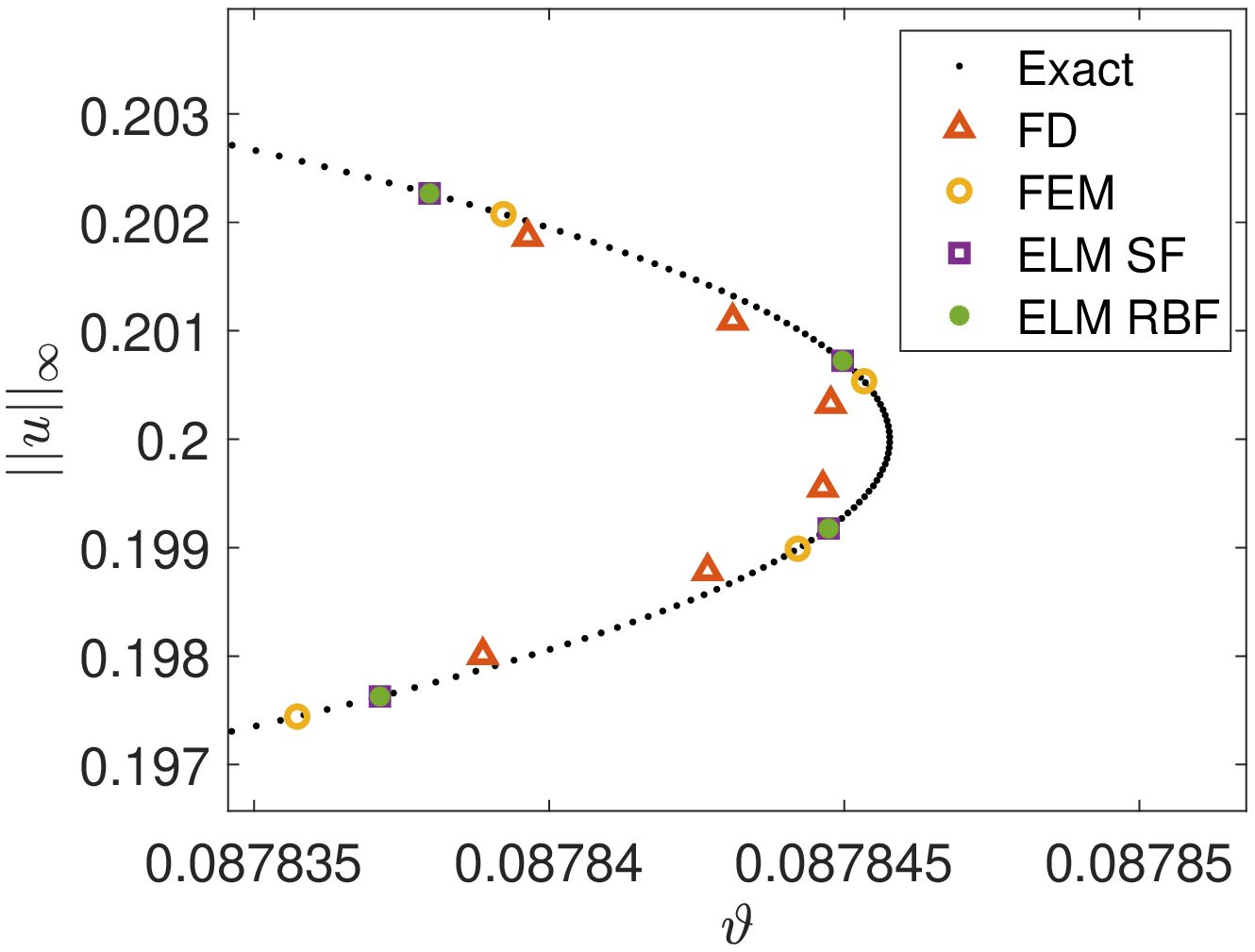}
    }
    \caption{(a) One-dimensional Burgers equation (\ref{eqnburgers}) with mixed boundary conditions (\ref{eqnboundaryburgers}). Bifurcation diagram with respect to the Neumann boundary value $\theta$ as obtained for $\nu=1/10$, with FD, FEM and ELM schemes with a fixed problem size $N=400$; (b) Zoom near the turning point.}
    \label{figbifdiagburgers}
\end{figure}
\begin{table}[th]
\begin{tabular}{lrrrr}
\hline
\hline
N & FD & FEM & ELM SF & ELM RBF\\
\hline
20  & -3.3230e-04 & -4.8557e-09 & 2.7506e-08  & -4.3683e-06 \\
50  & -5.3487e-05 & -7.6969e-09 & -2.0571e-09 & -2.1431e-09 \\
100 & -1.3370e-05 & -2.1575e-09 & -9.8439e-09 & -9.8483e-09 \\
200 & -3.3420e-06 & -5.9262e-09  & -9.6156e-09 & -9.6095e-09 \\
400 & -8.3473e-07 & 4.1474e-09 & 9.3882e-10  & 9.3338e-10 \\
\hline
\hline
\end{tabular}
\label{burgerscombifpoint}
\caption{One-dimensional Burgers equation (\ref{eqnburgers}) with mixed boundary conditions (\ref{eqnboundaryburgers}). Comparative results with respect to the error between the estimated value of the turning point as obtained with FD, FEM and ELMs schemes and the exact value of the turning point at $\vartheta^*=0.087845767978$ for $\nu=1/10$. The value of the turning point was estimated by fitting a parabola around the four points with the largest $\lambda$ values as obtained by the arc-length continuation.}
\end{table}
In this case, steep gradients arise at the right boundary related to the presence of the upper unstable solution, as discussed in Lemma \ref{Lemma2_new} and Corollary \ref{Lemma2}. In Table \ref{Table1e6}, we report the error between the numerically computed and the exact analytically obtained value (see Eq. \eqref{eqnsolburgers}) at $x=0$ when the value of boundary condition $\vartheta$ at the left boundary is $\vartheta=10^{-6}$. Again as shown, near the left boundary, the proposed ELM scheme outperforms both FEM and FD for medium to larger sizes of the grid.
\FloatBarrier
\begin{table}[th]
\begin{tabular}{lrrrr}
\hline
\hline
N & FD & FEM & ELM SF & ELM RBF\\
\hline
20  & -1.8099e-01 & 2.0532e-02 & -6.5492e-01 & -6.1366e-01 \\
50  & -2.6632e-02 & 7.6660e-04 & -5.8353e-01 & -6.0850e-01 \\
100 & -6.5179e-03 & 1.5752e-04 & -1.9976e-01 & -1.0504e-01 \\
200 & -1.6105e-03 & 8.9850e-05 & -2.4956e-06 & -5.0483e-06 \\
400 & -3.9992e-04 & 6.2798e-05 & -3.4737e-06 & -9.5189e-06 \\
\hline
\hline
\end{tabular}
\caption{One-dimensional Burgers equation (\ref{eqnburgers}) with mixed boundary conditions (\ref{eqnboundaryburgers}). Comparative results with respect to the error between the computed solution (at $x=0$) with FD, FEM and ELMs (with both sigmoidal and radial basis functions) and the exact solution $u(0)=1.798516682636303$ (see Eq. \eqref{eqnsolburgers}) for $\vartheta=1e-6$ (the value of the Neumann condition at the left boundary).}
\label{Table1e6}
\end{table}

\begin{remark}[Linearization of the Burgers equation for its numerical solution.]\label{remark:burg_lin}
For the numerical solution of the Burgers equation \eqref{eqnburgers} with boundary conditions given by \eqref{eqnburgers_Dirichlet}, one can also consider the following simple iterative procedure that linearizes the equation:
\begin{equation*}
\left\{\begin{array}{l}
\text{Given }  u^{(0)}, \text{ do until convergence}    \\
\text{find } u^{(k)} \text{ such that } \nu\dfrac{\partial^2u^{(k)}}{\partial x^2}-u^{(k-1)} \dfrac{\partial u^{(k)}}{\partial x}=0 \ .
\end{array}\right.
\end{equation*}
In this way, the nonlinear term becomes a linear advection term with a non-constant coefficient given by the evaluation of $u$ at the previous iteration. This results to a fixed point scheme.
Such linearized equations can be easily solved, being linear elliptic equations, and thus in this case one can perform the analysis for linear systems presented in \cite{calabro2020extreme}. The results of this procedure are depicted in Figure \ref{fig:burg_fixpoint}.
\begin{figure}[th]
    \centering
    \subfigure[]{
   \includegraphics[width=0.45 \textwidth]{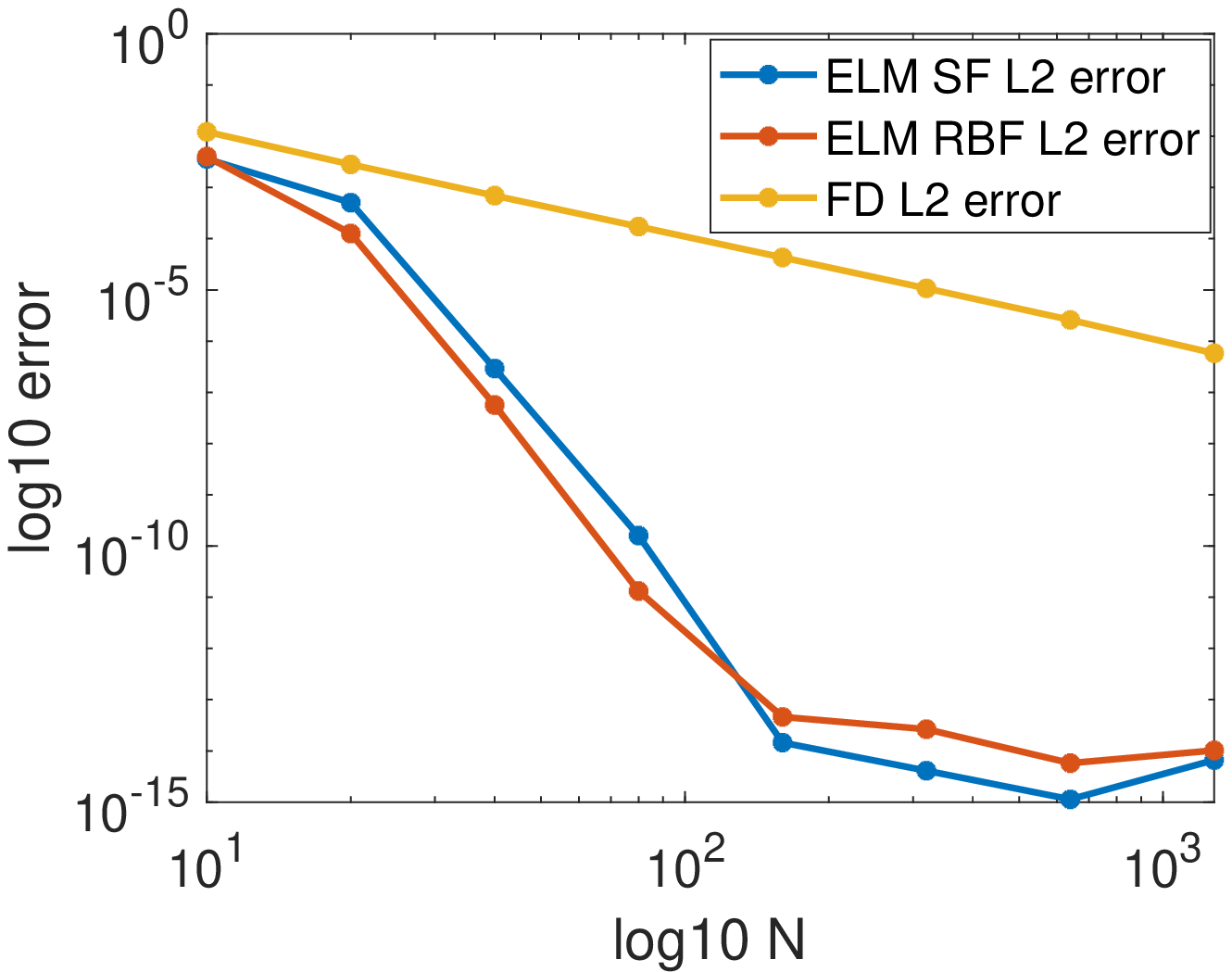}
}
~\subfigure[]{
    \includegraphics[width=0.45 \textwidth]{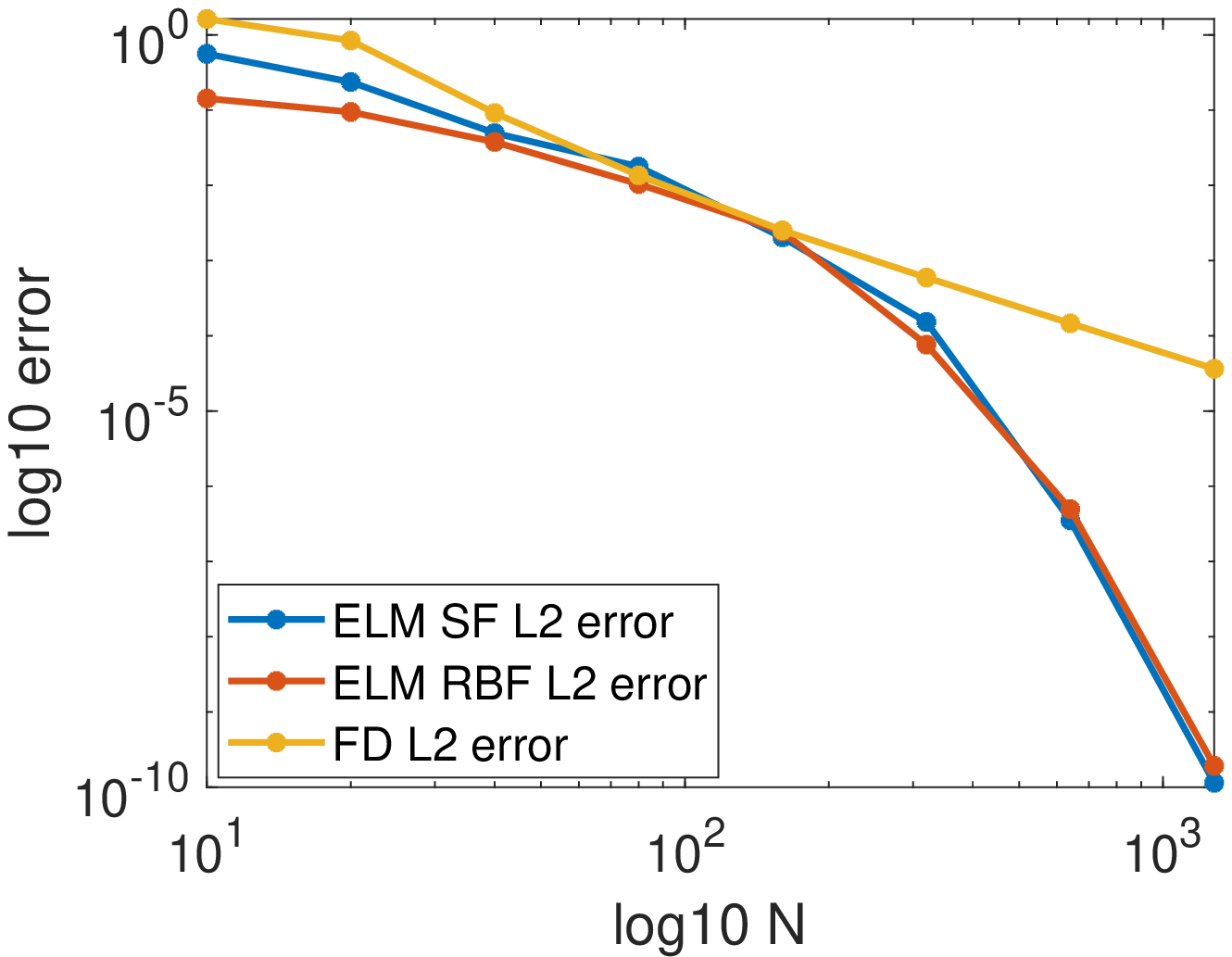}
}
\caption{Numerical accuracy of FD and ELM schemes with respect to the exact solution, for the case of the one-dimensional Burgers equation \eqref{eqnburgers} with Dirichlet boundary conditions given by \eqref{eqnburgers_Dirichlet}, as obtained by the fixed point scheme described in Remark \ref{remark:burg_lin} for (a) $\nu=0.1$ and (b) $\nu=0.007$. We depict the $L_2$--norm of the difference between the solutions obtained with FD and ELMs and the exact solution \eqref{BurDir:ExactSol}. }
\label{fig:burg_fixpoint}
\end{figure}
We point out that such iterations converge generally very slowly and, what is most important from our point of view, is that convergence is obtained only for a very ``good" guess of the solution. 
\end{remark}

\subsection{The one- and two-dimensional Liouville–Bratu–Gelfand Problem}\label{Sect:sez4.2}
The Liouville–Bratu–Gelfand model arises in many physical and chemical systems. It is an elliptic partial differential equation which in its general form is given by \cite{boyd1986analytical}:
\begin{equation}\label{eq:Bratu}
\Delta u(\bm{x}) +\lambda e^{u(\bm{x})}=0\,\ \bm{x} \in \Omega,
\end{equation}
with homogeneous Dirichlet conditions
\begin{equation}\label{eq:Bratu:BC}
u(\bm{x}) =0\,,\ \bm{x} \in \partial \Omega.
\end{equation}
The domain that we consider here is the $\Omega=[0, 1]^d$ in $R^d\,,\ d=1,2$.\par
The one-dimensional problem admits an analytical solution given by \cite{mohsen2014simple}:
\begin{equation}
u(x)=2\ln\frac{\cosh{\theta}}{\cosh{\theta (1-2x)}}, \text{ where $\theta$ is such that } \cosh{\theta} = \frac{4\theta}{\sqrt{2 \lambda}}.
\label{eq:sys1}
\end{equation}
It can be shown that when $0<\lambda <\lambda_c$ 
the problem admits two branches of solutions that meet at $\lambda_c \sim 3.513830719$, a limit point (saddle-node bifurcation) that marks the onset of two branches of solutions with different stability, while beyond that point no solutions exist.\par
For the two-dimensional problem, to the best of our knowledge, no such (as in the one-dimensional case) exact analytical solution exist that is verified by the numerical results that have been reported in the literature (e.g. \cite{chan1982arc,hajipour2018accurate}), in which the authors 
report the value of the turning at $\lambda_c \sim 6.808124$.
\par
\subsubsection{Numerical Solution with Finite Differences and Finite Elements}

The discretization of the one-dimensional problem in $N$ points with central finite differences at the unit interval $0\leq x \leq 1$ leads to the following system of $N-2$ algebraic equations $\forall x_j=(j-1)h, j=2,\dots N-1 $, $h=\frac{1}{N-1}$:
 \begin{equation*}
     F_j(u)=\frac{1}{h^2}(u_{j+1}-2u_{j}+u_{j-1})+\lambda e^{u_j}=0,
 \end{equation*}
where, at the boundaries $x_1=0, x_N=1$, we have $u_1=u_N=0$.\par
The solution of the above $N-2$ nonlinear algebraic equations is obtained iteratively using the Newton-Raphson method. The Jacobian is now triagonal; at each $n$-th iteration, the elements at the main diagonal are given by ${\frac{\partial F_j}{\partial u_j}}^{(n)}=-\frac{2}{h^2}+\lambda e^{u_{j}^{(n)}}$ and the elements of the first diagonal above and the first diagonal below are given by ${\frac{\partial F_{j+1}}{\partial u_j}}^{(n)}={\frac{\partial F_{j}}{\partial u_{j+1}}}^{(n)}=\frac{1}{h^2}$, respectively.\par
The discretization of the two-dimensional Bratu problem in $N\times N$ points with central finite differences on the square grid $0\leq x,y\leq 1$ with zero boundary conditions leads to the following system of $(N-2)\times (N-2)$ algebraic equations $	\forall (x_i=(i-1)h, y_j=(j-1)h), i,j=2,\dots N-1 $, $h=\frac{1}{N-1}$:
 \begin{equation*}
     F_{i,j}(u)=\frac{1}{h^2}(u_{i+1,j}+u_{i,j+1}-4u_{i,j}+u_{i,j-1}+u_{i-1,j})+\lambda e^{u_{i,j}}=0.
 \end{equation*}
The Jacobian is now a $(N-2)^2 \times (N-2)^2$ block diagonal matrix of the form:
\begin{equation*}
\nabla F = \frac{1}{h^2} \begin{bmatrix}
T_2 & I & 0 & 0 & \dots & \dots & 0\\
 I & T_3 & I & 0 & \dots & \dots & 0\\
0 & I & T_4 & I & 0 & \dots & 0\\
\vdots & \vdots & \ddots & \ddots & \ddots & \ddots & \vdots\\
0 & \dots & \dots & \dots & \dots & I & T_{N-1}
\end{bmatrix},
 \end{equation*}
where $I$ is the $(N-2)\times (N-2)$ identity matrix and $T_i$ is the $(N-2)\times (N-2)$ tridiagonal matrix with non null elements on the $j$-th row:
\begin{equation*}
    1\,,\quad -4 +h^{2} \lambda e^{u_{i+j,i+j}} \,,\quad 1
\end{equation*}

Regarding the FEM solution, for the one-dimensional Bratu problem, Eq. \eqref{eq6} gives:
\begin{equation}
R_k=\int_{\Omega} \left(\frac{\partial^2u}{\partial x^2} +\lambda e^{u(x)}\right) \phi_{k}(x) dx.
\label{eqbratres}
\end{equation}
By inserting Eq.(\ref{eq5}) into Eq.(\ref{eqbratres}) and by applying the Green’s formula for integration, we get:
\begin{equation}
R_k=\phi_{k}(x)\frac{du}{dx}\Big|_{0}^{1}-\sum_{j=1}^{N} u_j \int_{0}^{1} \frac{d \phi_{j}(x)}{dx}\frac{d\phi_{k}(x)}{dx} dx+\lambda \int_{0}^{1} e^{ \sum_{j=1}^{N} u_j \phi_{j}(x)}\phi_{k}(x) dx
\label{eq8}
\end{equation}
and because of the zero Dirichlet boundary conditions, Eq.(\ref{eq8}) becomes:
\begin{equation*}
R_k=-\sum_{j=1}^{N} u_j \int_{0}^{1} \frac{d \phi_{j}(x)}{dx}\frac{d\phi_{k}(x)}{dx} dx+\lambda\int_{0}^{1} e^{\sum_{j=1}^{N} u_j \phi_{j}(x)}\phi_{k}(x)dx.
\label{eq9}
\end{equation*}
The elements of the Jacobian matrix are given by:
\begin{equation}
\frac{\partial R_i}{\partial u_j}=- \int_{0}^{1} \frac{d \phi_{j}(x)}{dx}\frac{d\phi_{k}(x)}{dx} dx+\lambda\int_{0}^{1} e^{\sum_{j=1}^{N} u_j \phi_{j}(x)}  \phi_{j}(x) \phi_{k}(x)dx
\label{eq10b}
\end{equation}
Due to the Dirichlet boundary conditions, Eq.(\ref{eq10b}) becomes:
\begin{equation}
\begin{bmatrix}
1 & 0 & \dots & 0 & \dots & 0\\
\frac{\partial R_2}{\partial u_1} & \frac{\partial R_2}{\partial u_2} & \dots & \frac{R_2}{\partial u_j} & \dots &\frac{R_2}{\partial u_N}\\
\vdots  &\vdots & \ddots & \vdots & \ddots &\vdots\\
\frac{\partial R_k}{\partial u_1} & \frac{\partial R_k}{\partial u_2} & \dots & \frac{R_k}{\partial u_j} & \dots &\frac{R_k}{\partial u_N}\\
\vdots  &\vdots & \ddots & \vdots & \ddots &\vdots\\
0 & 0 & \dots & 0 & \dots & 1
\end{bmatrix}_{\big|{u^{(n)}}} \cdot
\begin{bmatrix}
du^{(n)}_{1}\\
du^{(n)}_{2}\\
\vdots\\
du^{(n)}_{j}\\
\vdots\\
du^{(n)}_{N}
\end{bmatrix}=-
\begin{bmatrix}
0\\
R_{2}\\
\vdots\\
R_{k}\\
\vdots\\
0
\end{bmatrix}_{\big|{u^{(n)}}}.
\label{eqnnewtonFE}
\end{equation}
For the two-dimensional Bratu problem
, the residuals are given by:
\begin{equation*}
R_k=\iint_{\Omega} (\frac{\partial^2u(x,y)}{\partial x^2}+\frac{\partial^2u(x,y)}{\partial y^2} +\lambda e^{u(x,y)}) \phi_k(x,y) dx dy.
\label{eq11}
\end{equation*}
By applying the Green’s formula for integration, we get:
\begin{equation*}
\begin{split}
R_k=\oint_{\partial{\Omega}} \nabla u(x,y) d\ell -&\iint_{\Omega} \nabla u(x,y) \nabla \phi_{k}(x,y) dx dy\\ +&\iint_{\Omega} \lambda e^{u(x,y)} \phi_k(x,y) dx dy.
\end{split}
\label{eq12}
\end{equation*}
By inserting Eq.(\ref{eq5}) and the zero Dirichlet boundary conditions, we get:
\begin{equation*}
\begin{split}
R_k= -&\sum_{j=1}^{N} u_j \iint_{\Omega} \nabla \phi_{j}(x,y) \nabla \phi_{k}(x,y) dx dy\\ 
+&\iint_{\Omega} \lambda e^{\sum_{j=1}^{N} u_{j} \phi_{j}(x,y)} \phi_k(x,y) dx dy.
\end{split}
\label{eq13}
\end{equation*}
Thus, the elements of the Jacobian matrix for the two-dimensional Bratu problem are given by:
\begin{equation*}
\begin{split}
\frac{\partial R_k}{\partial u_j} =-&\iint_{\Omega} \nabla \phi_{j}(x,y) \nabla \phi_{k}(x,y) dx dy\\ +&\iint_{\Omega} \lambda e^{\sum_{j=1}^{N} u_{j} \phi_{j}(x,y)}  \phi_{j}(x,y) \phi_k(x,y) dx dy.
\end{split}
\label{eq14}
\end{equation*}
As before, for our computations we have used 
quadratic basis functions using an affine  element mapping in the domain $[0, 1]^2$. 

\subsubsection{Numerical Solution with Extreme Learning Machine Collocation}

Collocating the ELM network function (\ref{eq:ELMframework}) in the 1D Bratu problem (\ref{eq:Bratu}) leads to the following system:
\begin{equation*}
    F_i(\bm{w},\lambda)=\sum_{j=1}^N w_j \alpha_{j}^2 \psi_j''(x_i)+\lambda \text{exp}\biggl(\sum_{j=1}^N w_j \psi_j(x_i)\biggr)=0, \quad i=2,\dots,M-1
\end{equation*}
with boundary conditions:
\begin{equation*}
    F_1(\bm{w},\lambda)=\sum_{j=1}^N w_j \psi_j(0)=0, \qquad F_M(\bm{w},\lambda)=\sum_{j=1}^N w_j \psi_j(1)=0.
\end{equation*}
Thus, the elements of the Jacobian matrix $\nabla_{\bm{w}} \bm{F}$ are given by:
\begin{equation*}
    \frac{\partial F_i}{\partial w_j}=\alpha_{j}^2 \psi_j''(x_i)+\lambda \psi_j(x_i)\text{exp}\biggl(\sum_{j=1}^N w_j \psi_j(x_i)\biggr), \qquad i=2,\dots,M-1
\end{equation*}
and
\begin{equation*}
    \frac{\partial F_1}{\partial w_j}(\bm{w},\lambda)=\psi_j(0) \qquad \frac{\partial F_M}{\partial w_j}(\bm{w},\lambda)=\psi_j(1).
\end{equation*}
The application of Newton's method \eqref{ELMiter} is straightforward using the exact computation of derivatives of the basis functions (see \eqref{eq:der:SF} and \eqref{eq:der:RBF}).\par
For the two-dimensional Bratu problem (\ref{eq:Bratu}), we have:
\begin{equation*}
\begin{split}
    F_i(\bm{w},\lambda)=&\sum_{j=1}^N w_j \alpha_{j,1}^2 \psi_j''(x_i,y_i)+\sum_{j=1}^N w_j \alpha_{j,2}^2 \psi_j''(x_i,y_i)\\
    &+\lambda\, \text{exp}\biggl(\sum_{j=1}^N w_j \psi_j(x_i,y_i)\biggr)=0, \quad i=1,\dots,M_{\Omega}
\end{split}
\end{equation*}
with boundary conditions:
\begin{equation*}
    F_k(\bm{w},\lambda)= \sum_{j=1}^N w_j \psi_j(x_k,y_k)=0, \quad k=1,\dots,M_{1}.
\end{equation*}
Thus, the elements of the Jacobian matrix $\nabla_{\bm{w}} F$ read:
\begin{equation*}
\begin{split}
    \frac{\partial F_i}{\partial w_j}&=\alpha_{j,1}^2 \psi_j''(x_i,y_i)+\alpha_{j,2}^2 \psi_j''(x_i,y_i)\\
    &+\lambda \psi_j(x_i,y_i)\, \text{exp}\biggl(\sum_{j=1}^N w_j \psi_j(x_i,y_i)\biggr), \qquad i=1,\dots,M_{\Omega}
\end{split}
\end{equation*}
and
\begin{equation*}
    \frac{\partial F_k}{\partial w_j}(\bm{w},\lambda)=\psi_j(x_k,y_k)=0, \quad k=1,\dots,M_{1} \ .
\end{equation*}
Also in this case, with the above computations the application of Newton's method \eqref{ELMiter} is straightforward.

\subsubsection{Numerical results for the one-dimensional problem}
\begin{figure}
    \centering
    \subfigure[]{
   \includegraphics[width=0.45 \textwidth]{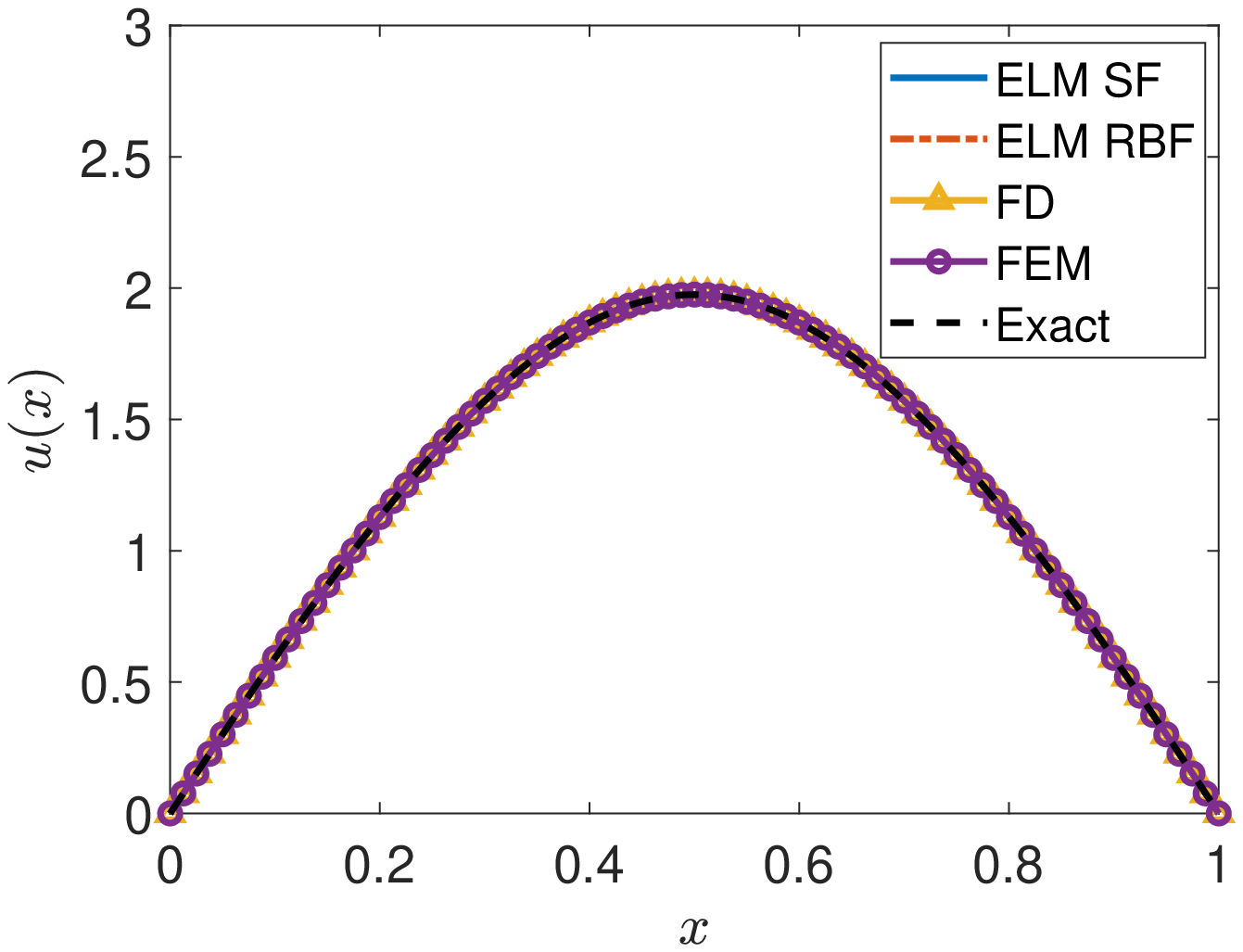}
}
~\subfigure[]{
    \includegraphics[width=0.45 \textwidth]{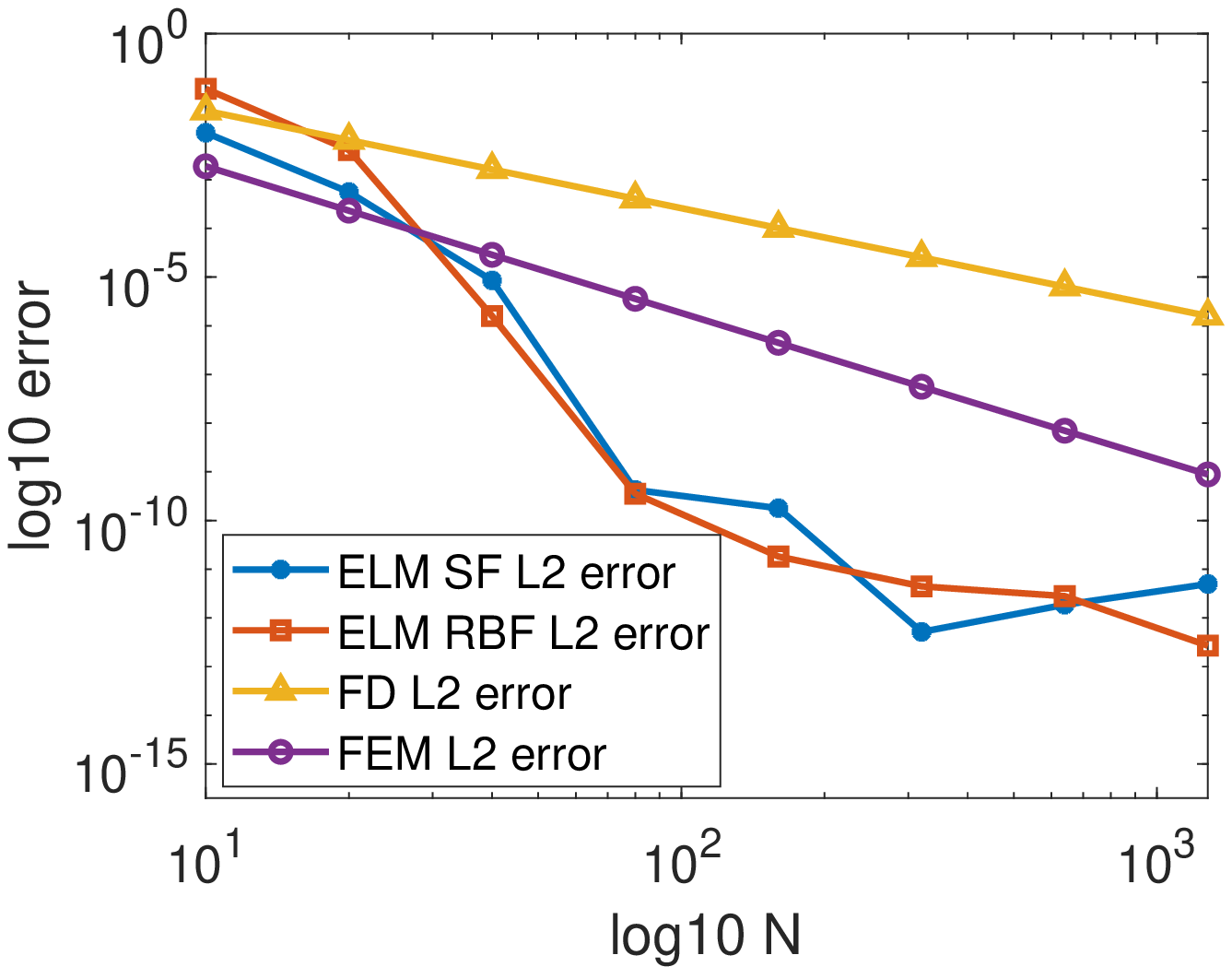}
    }
~\subfigure[]{
    \includegraphics[width=0.45 \textwidth]{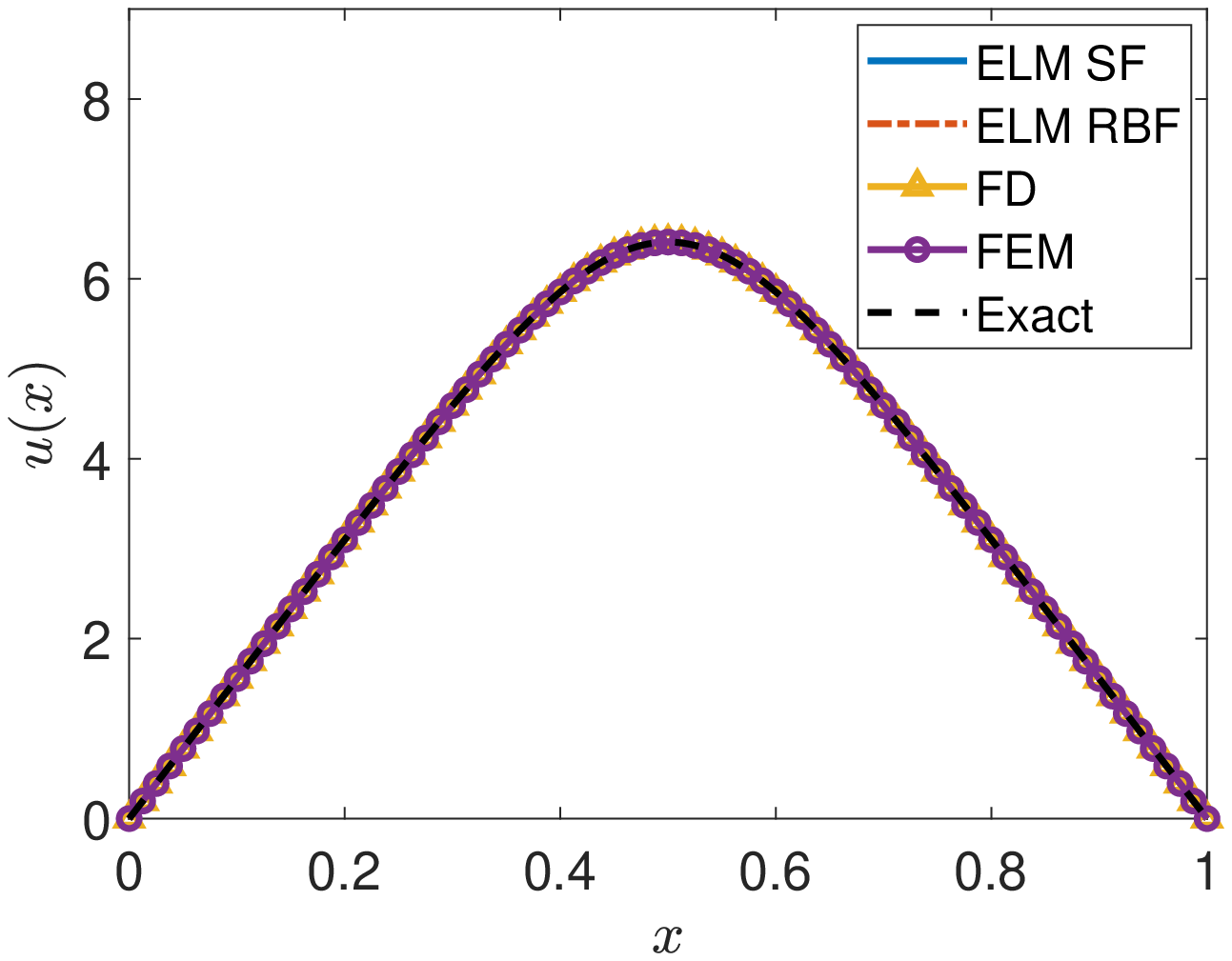}
    }
    ~\subfigure[]{
    \includegraphics[width=0.45 \textwidth]{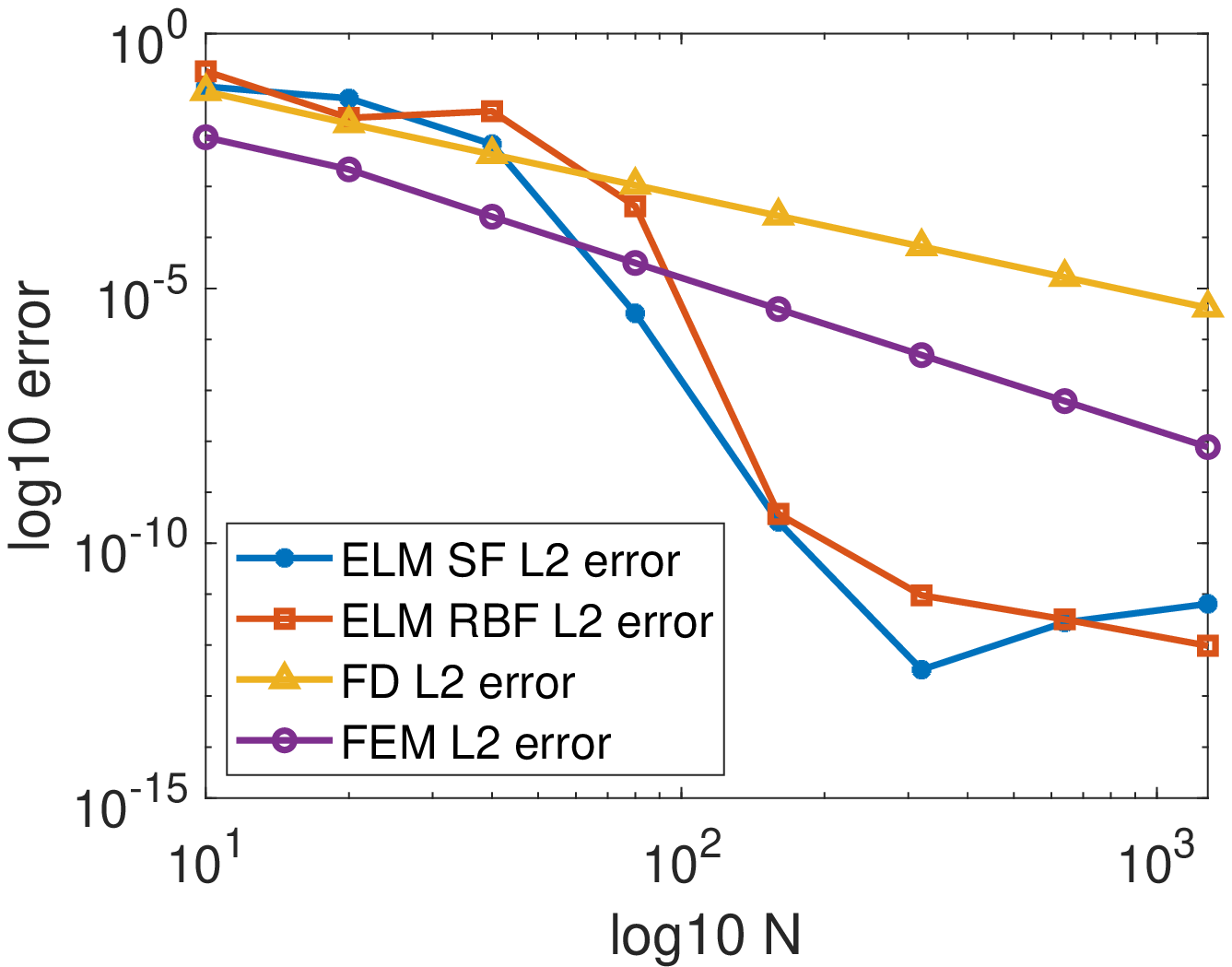}
    }
  \caption{Numerical solutions and accuracy of the FD, FEM and ELMs schemes for the one-dimensional Bratu problem (\ref{eq:Bratu}). (a) Computed solutions at the upper-branch unstable solution at $\lambda=3$ for a fixed problem size $N=40$. (b) $L_2$--norm of differences with respect to the exact unstable solution \eqref{eq:sys1} at $\lambda=3$ for various values of $N$. (c) Computed solutions at the upper-branch unstable solution at $\lambda=0.2$ with a fixed problem size $N=40$. (d) $L_2$--norm of differences with respect to the exact unstable solution \eqref{eq:sys1} at $\lambda=0.2$ for various values of $N$. The initial guess of the solutions was a parabola satisfying the homogeneous boundary conditions with a fixed $L_{\infty}$--norm $||u||_{\infty}=l_0$ close to the one resulting from the exact solution.}
    \label{fig:bratu_sol}
\end{figure}

First, we show the numerical results for the one-dimensional Liouville–Bratu–Gelfand equation (\ref{eq:Bratu}) with homogeneous Dirichlet boundary conditions  \eqref{eq:Bratu:BC}. Recall that an exact solution, although in implicit form, is available in this case (see equation \eqref{eq:sys1}); thus, as discussed,  the exact  solutions are derived using Newton's method with a convergence tolerance of $10^{-12}$.
Figure \ref{fig:bratu_sol} depicts the comparative results between the exact, FD, FEM and ELM solutions on the upper-branch as obtained by applying Newton's iterations, for two values of the parameter $\lambda$ and a fixed $N=40$, namely for $\lambda=3$ close to the turning point (occurring at $\lambda_c \sim 3.513830719$) and for $\lambda=0.2$. For our illustrations, we have set as initial guess $u_0(x)$ a parabola that satisfies the homogeneous boundary conditions, namely: 
\begin{equation*}
    u_0(x)=4l_0(x-x^2),
\end{equation*}
with a fixed $L_{\infty}$--norm $||u||_{\infty}=l_0$ close to the one obtained from the exact solution.\par
In particular, for $\lambda=3$, we used as initial guess a parabola with $l_0=2.2$; in all cases Newton's iterations converge to the correct unstable upper-branch solution. For $\lambda=0.2$, we used as initial guess a parabola with $l_0=6.4$ (the exact solution has $l_0 \sim 6.5)$; again in all cases, Newton's iterations converged to the correct unstable upper-branch solution. To clarify more the behaviour of the convergence, in Figure \ref{fig:conv_zone}, we report the regimes of convergence for a grid of $L_{\infty}$ norms of the initial guesses (parabolas) and $\lambda$s.

\begin{figure}[h!]
    \centering
    \subfigure[]{
   \includegraphics[width=0.45 \textwidth]{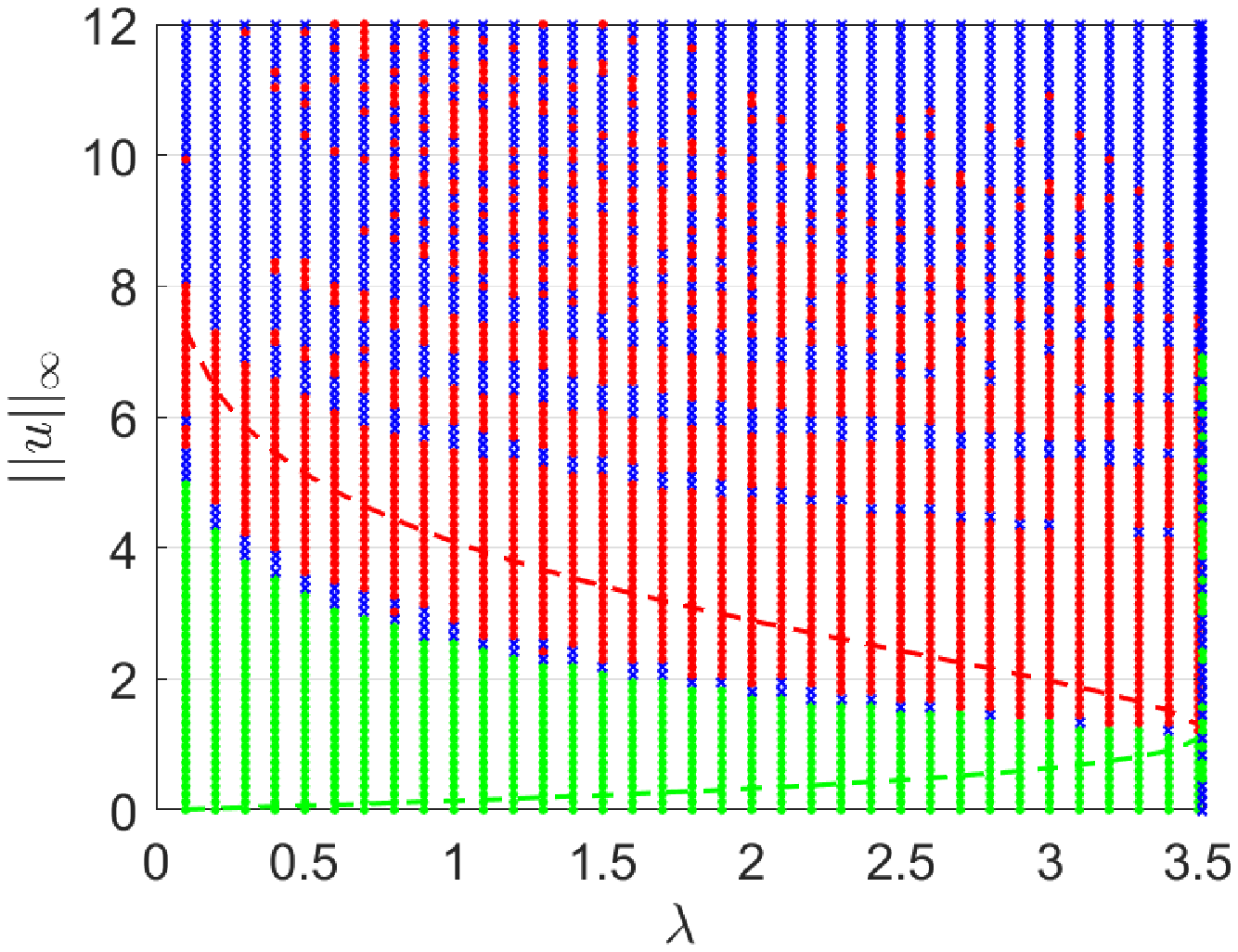}
}
~\subfigure[]{
    \includegraphics[width=0.45 \textwidth]{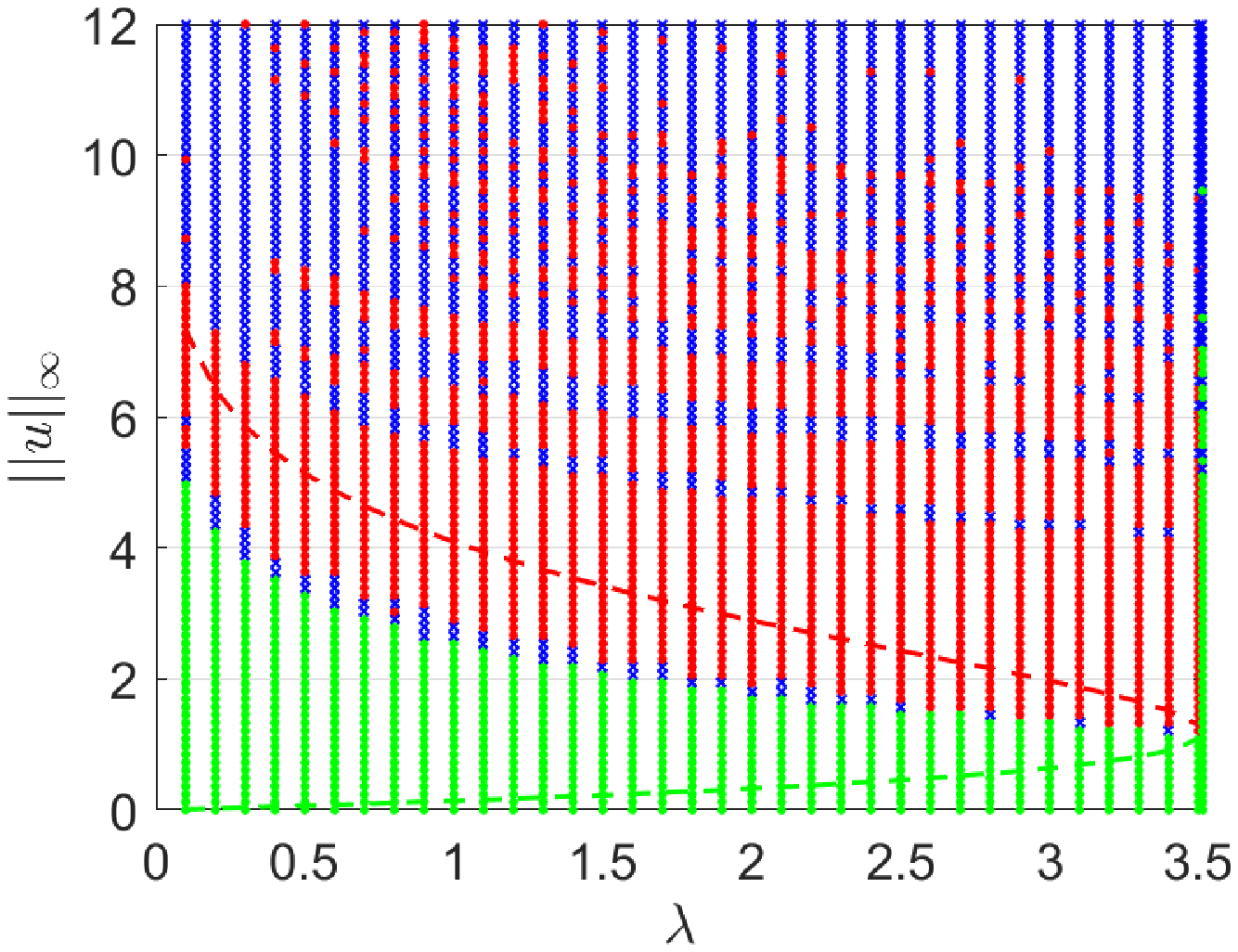}
}
~\subfigure[]{
    \includegraphics[width=0.45 \textwidth]{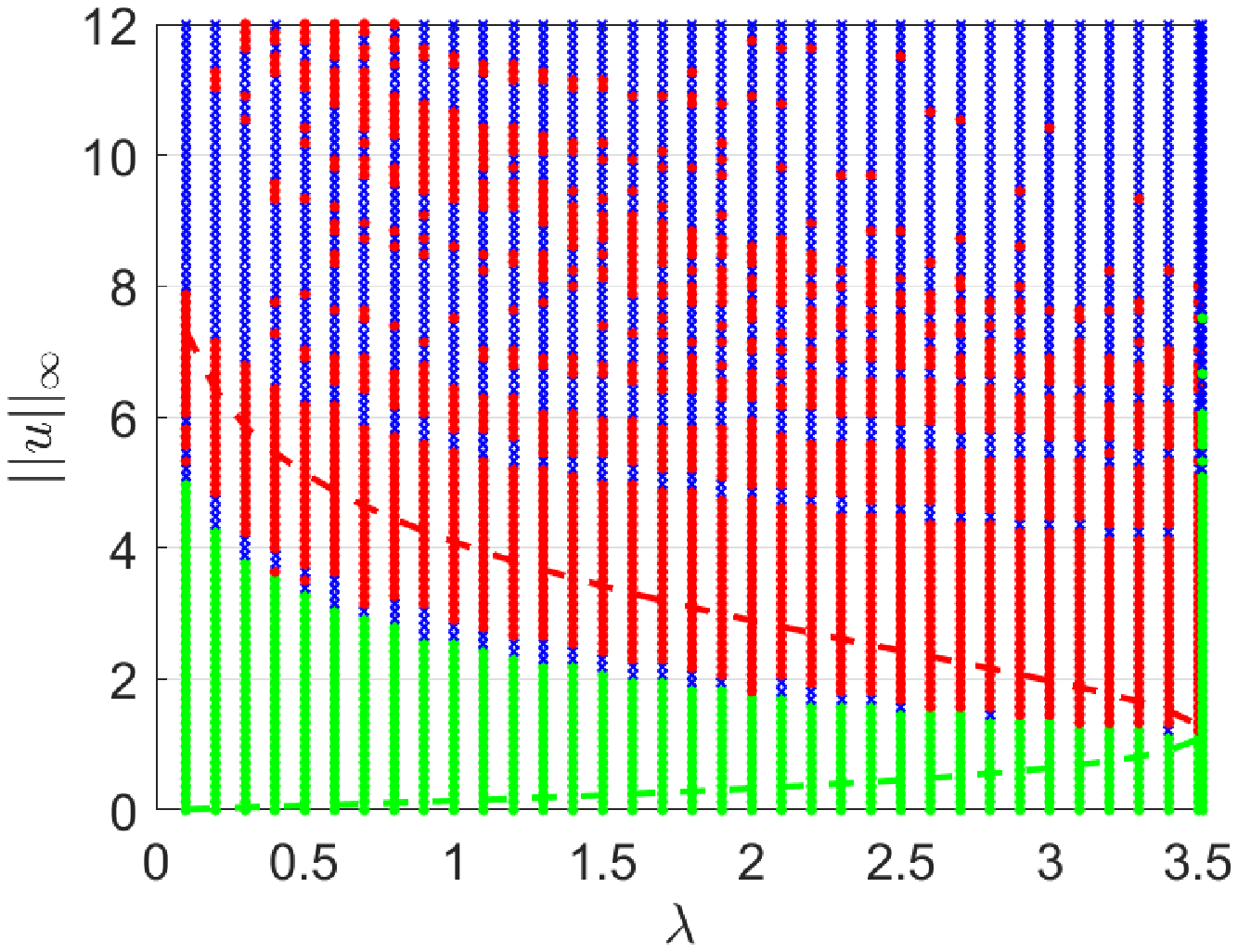}
}
~\subfigure[]{
    \includegraphics[width=0.45 \textwidth]{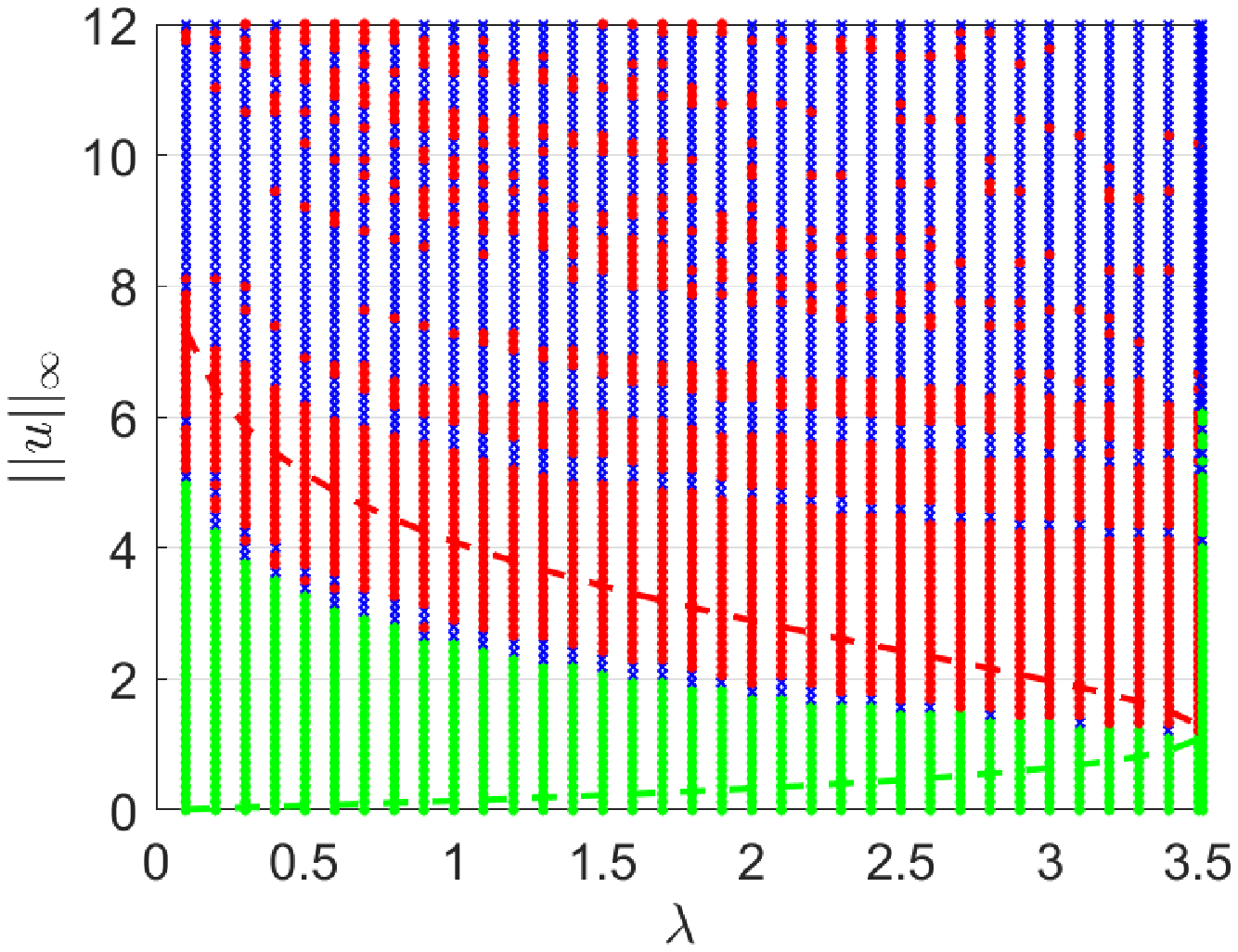}
}
\caption{Convergence regimes (basin of attraction) of Newton's method with the (a) FD, (b) FEM, and (c-d) ELM numerical schemes for the one-dimensional Bratu problem (\ref{eq:Bratu}) for a grid of initial  guesses ($L_{\infty}$--norms of parabolas that satisfy the boundary conditions (\ref{eq:Bratu:BC})) and $\lambda$s. Green points indicate convergence to the lower-branch solutions; Red points indicate convergence to the upper-branch solutions; Blue points indicate divergence. (c) ELM with logistic SF (\ref{eq:logisticSF}) (d) ELM with Gaussian RBF (\ref{eq:gaussianRBF}).}
\label{fig:conv_zone}
\end{figure}

\begin{remark}[Linearization of the equation for the numerical solution of the Liouville–Bratu–Gelfand problem]\label{remark:bratu_lin}

\begin{figure}[th]
    \centering
~\subfigure[]{
    \includegraphics[width=0.45 \textwidth]{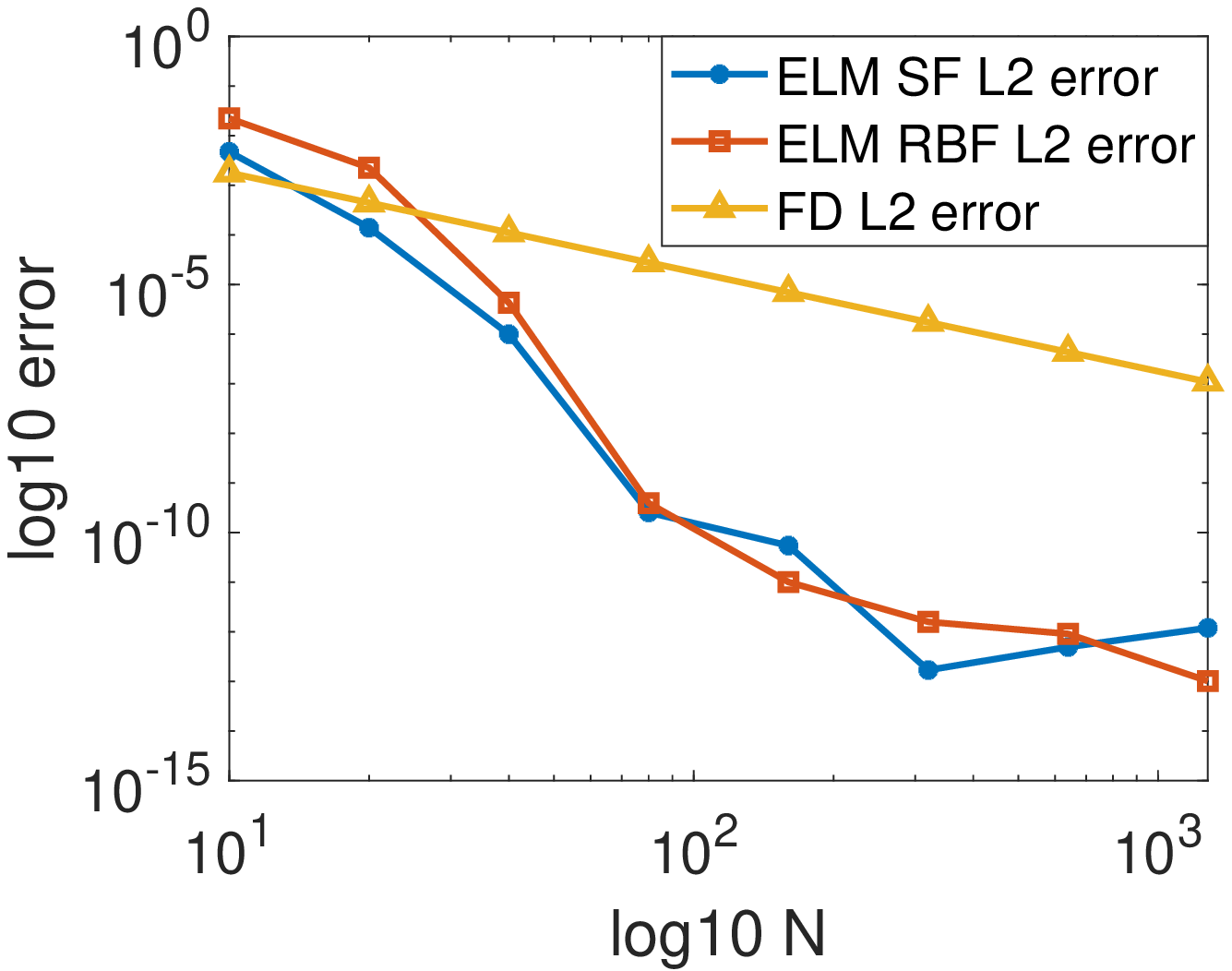}}
    ~\subfigure[]{
    \includegraphics[width=0.45 \textwidth]{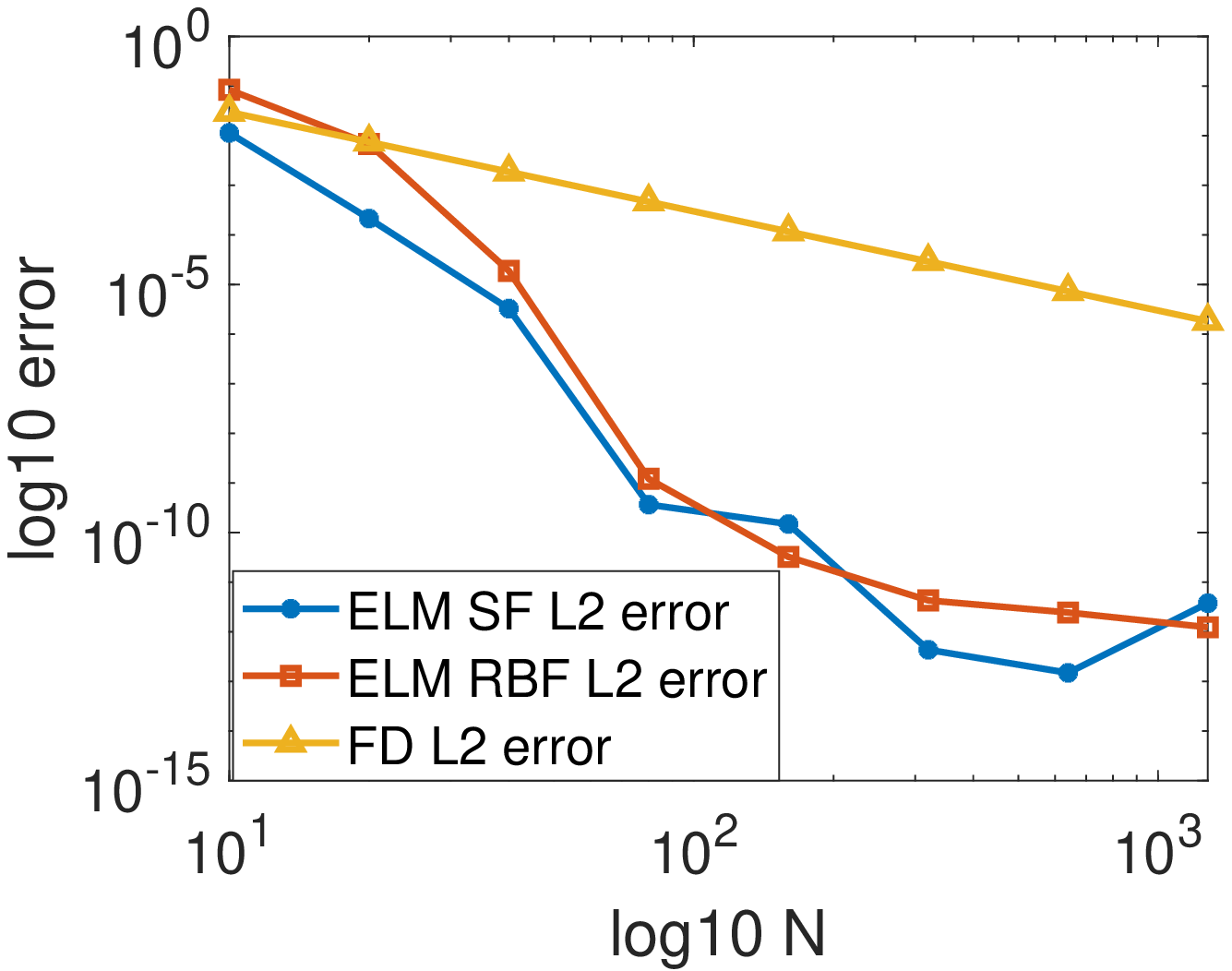}}
    
    \caption{Fixed point iterations: $L_2$--norm of the difference errors for the low and up branch Liouville–Bratu–Gelfand solution \eqref{eq:sys1} for $\lambda=2$:  (a) $L_2$ errors with respect to $N$ of the low branch solution (b) $L_2$ errors with respect to $N$ of the upper branch.}
    \label{fig:bratu_sol_fixedpoint_iteration}
\end{figure}

For the solution of the equation \eqref{eq:Bratu} with boundary conditions given by \eqref{eq:Bratu:BC}, one can consider the following iterative procedure that linearizes the equation:
\begin{equation*}
\left\{\begin{array}{l}
\text{Given }  u^{(0)}, \text{ do until convergence}    \\
\text{find } u^{(k)} \text{ such that } \Delta u^{(k)} +  \lambda e^{u^{(k-1)}}  u^{(k)} = \lambda ( u^{(k-1)}-1) e^{u^{(k-1)}} \ .
\end{array}\right.
\end{equation*}
In this way, the nonlinear term becomes a linear reaction term with a non-constant coefficient given by the evaluation of the nonlinearity at the previous step. Then, we implemented fixed point iterations until convergence. Such a linearization procedure is used, for example, in \cite{iqbal2020numerical}.
In Figure \ref{fig:bratu_sol_fixedpoint_iteration}, we report some results on the application of this method.
We note that this scheme converges more slowly and it is not so robust compared to Newton's method. 
\end{remark}


\subsubsection{Bifurcation diagram and numerical accuracy}
In this section, we report the numerical results obtained by the bifurcation analysis of the one-dimensional Bratu problem \eqref{eq:Bratu}. Figure \ref{fig:bratu_bifurcation} shows the constructed bifurcation diagram with respect to the parameter $\lambda$ and in Table \ref{table_bratu1D} we report  the accuracy of the computed value as obtained with FD, FEM and ELMS, versus the exact value of the turning point. As shown, the ELMs provide a bigger numerical accuracy for the value of the turning point for medium to large sizes of the grid, and equivalent results (the ELM with SF) to FEM, both outperforming the FD scheme.
\begin{table}[]
\begin{tabular}{lrrrr}
\hline
\hline
N & FD & FEM & ELM SF & ELM RBF\\
\hline
20  & -4.5737e-03 & 3.4410e-05 & 8.7618e-05  & 2.9953e-02  \\
50  & -7.3137e-04 & 8.4422e-07 & 2.9818e-07  & 6.6092e-05  \\
100 & -1.8282e-04 & 5.0597e-08 & -3.7086e-08 & 6.1302e-08  \\
200 & -4.5683e-05 & 2.3606e-08 & -4.5484e-09 & -2.6770e-09 \\
400 & -1.1412e-05 & 1.3557e-08 & 2.0169e-09  & 2.0275e-09 \\
\hline
\hline
\end{tabular}
\caption{One-dimensional Bratu problem (\ref{eq:Bratu}). Accuracy of FD, FEM and ELMS in the approximation of the value of the turning point with respect to the exact value $\lambda=3.513830719125162$. Values express the difference with the computed turning point and the exact one. The value of the turning point was estimated by fitting a parabola around the four points with the largest $\lambda$ values as obtained with arc-length continuation.}
\label{table_bratu1D}
\end{table}

\begin{figure}[ht]
    \centering
    \subfigure[]{
    \includegraphics[width=0.45 \textwidth]{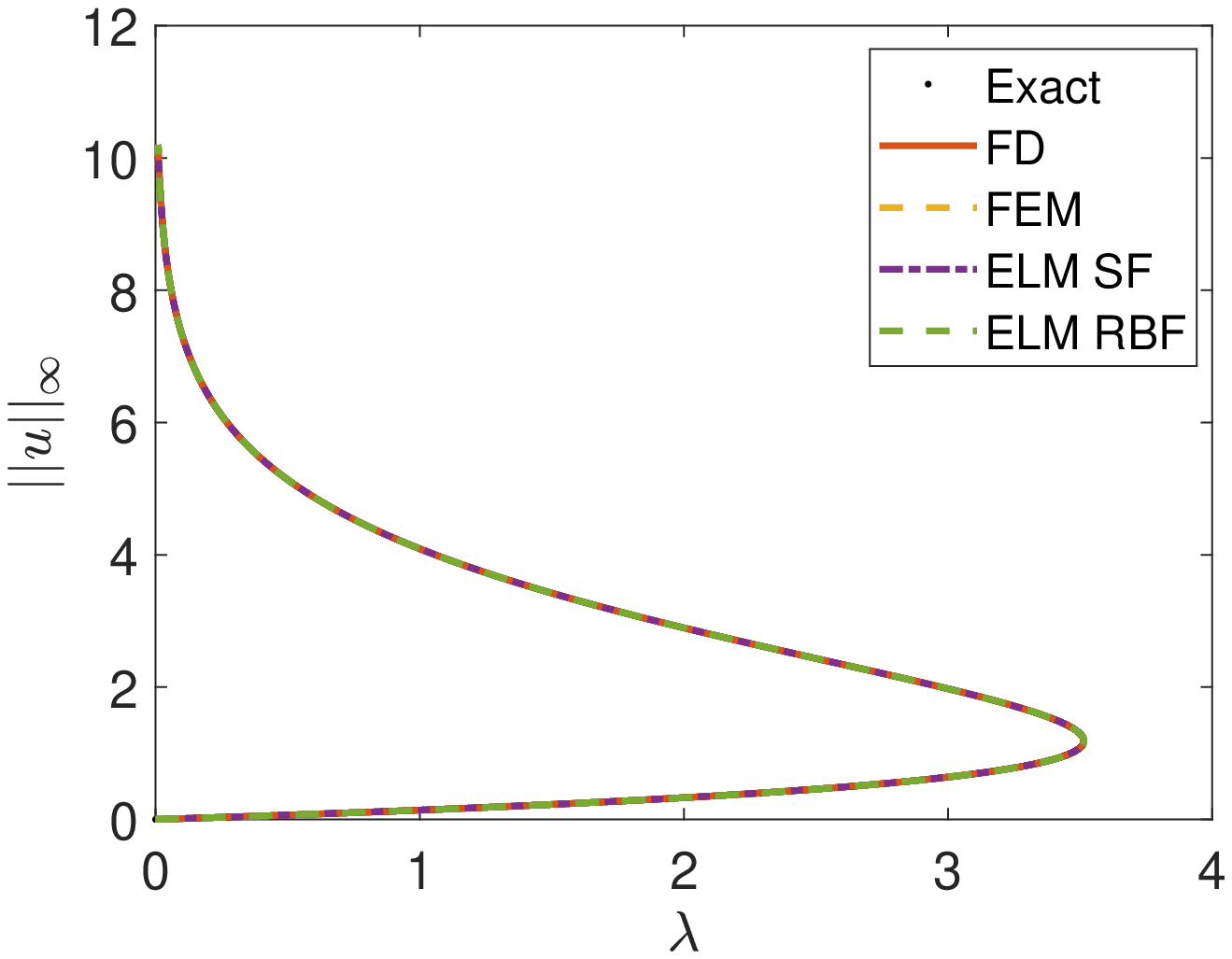}
    }
    \subfigure[]{
    \includegraphics[width=0.45 \textwidth]{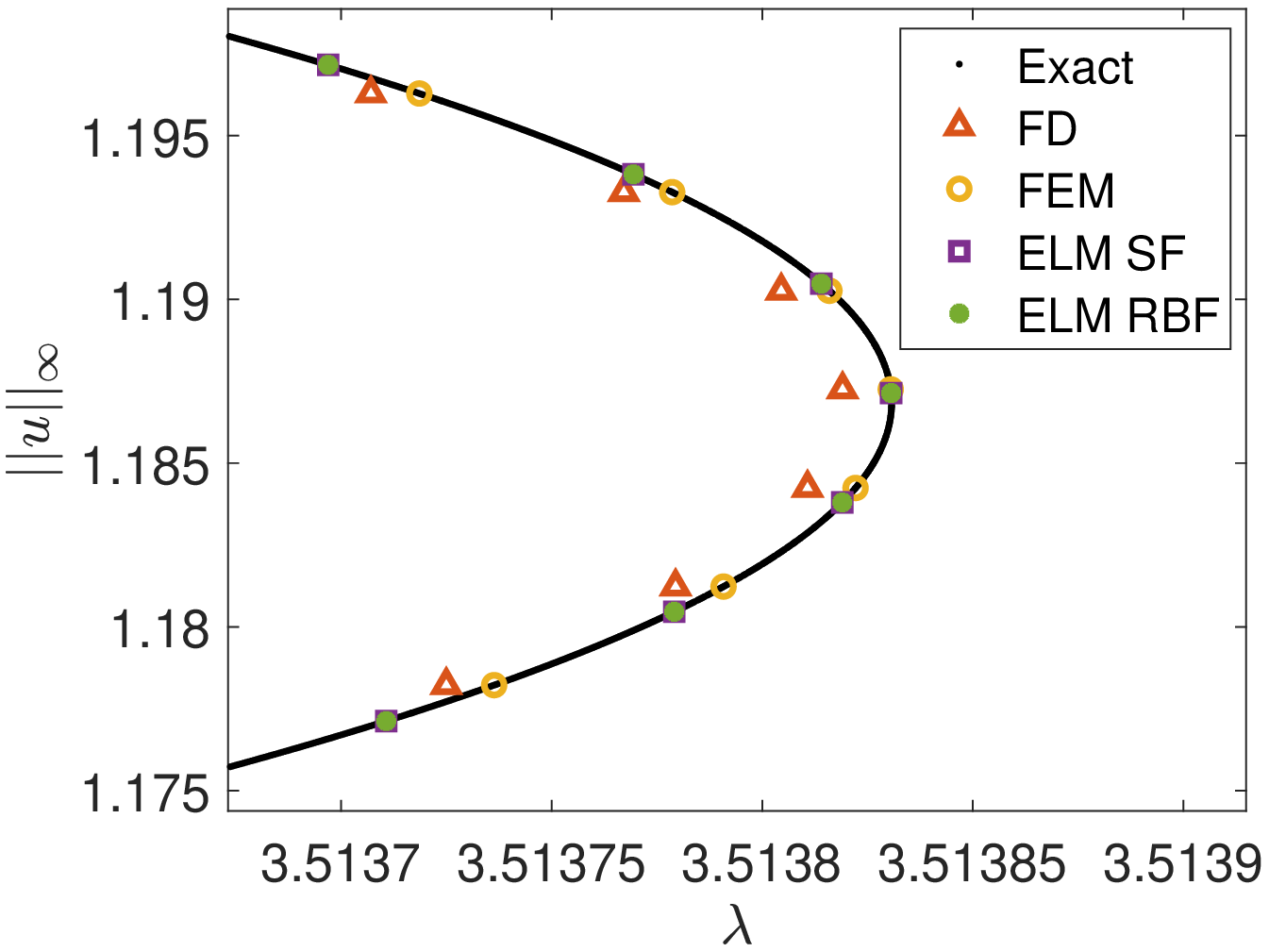}
    }
    \caption{(a) Bifurcation diagram for the one-dimensional Bratu problem (\ref{eq:Bratu}), with a fixed problem size $N=400$. (b) Zoom near the turning point.}
    \label{fig:bratu_bifurcation}
\end{figure}
In Figures \ref{fig:bratu_contourplot_down} and \ref{fig:bratu_contourplot_up}, we report the contour plots of the $L_{\infty}$--norms of the differences between the computed solutions by FD, FEM and ELMs and the exact solutions for the lower- (\ref{fig:bratu_contourplot_down}) and upper-branch (\ref{fig:bratu_contourplot_up}), respectively with respect to $N$ and $\lambda$.\par
\begin{figure}[ht]
    \centering
    \subfigure[]{
   \includegraphics[width=0.45 \textwidth]{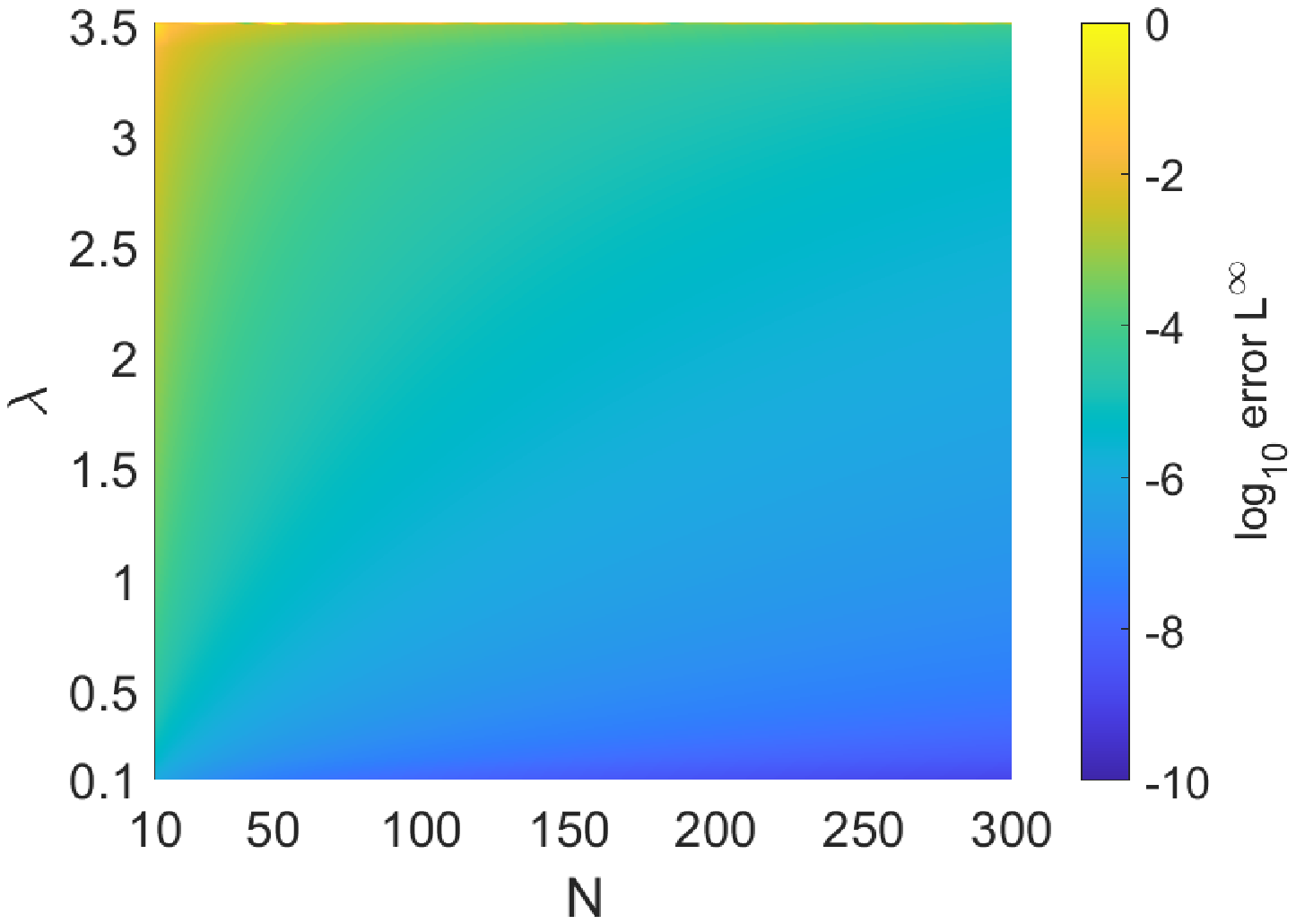}
}
~\subfigure[]{
    \includegraphics[width=0.45 \textwidth]{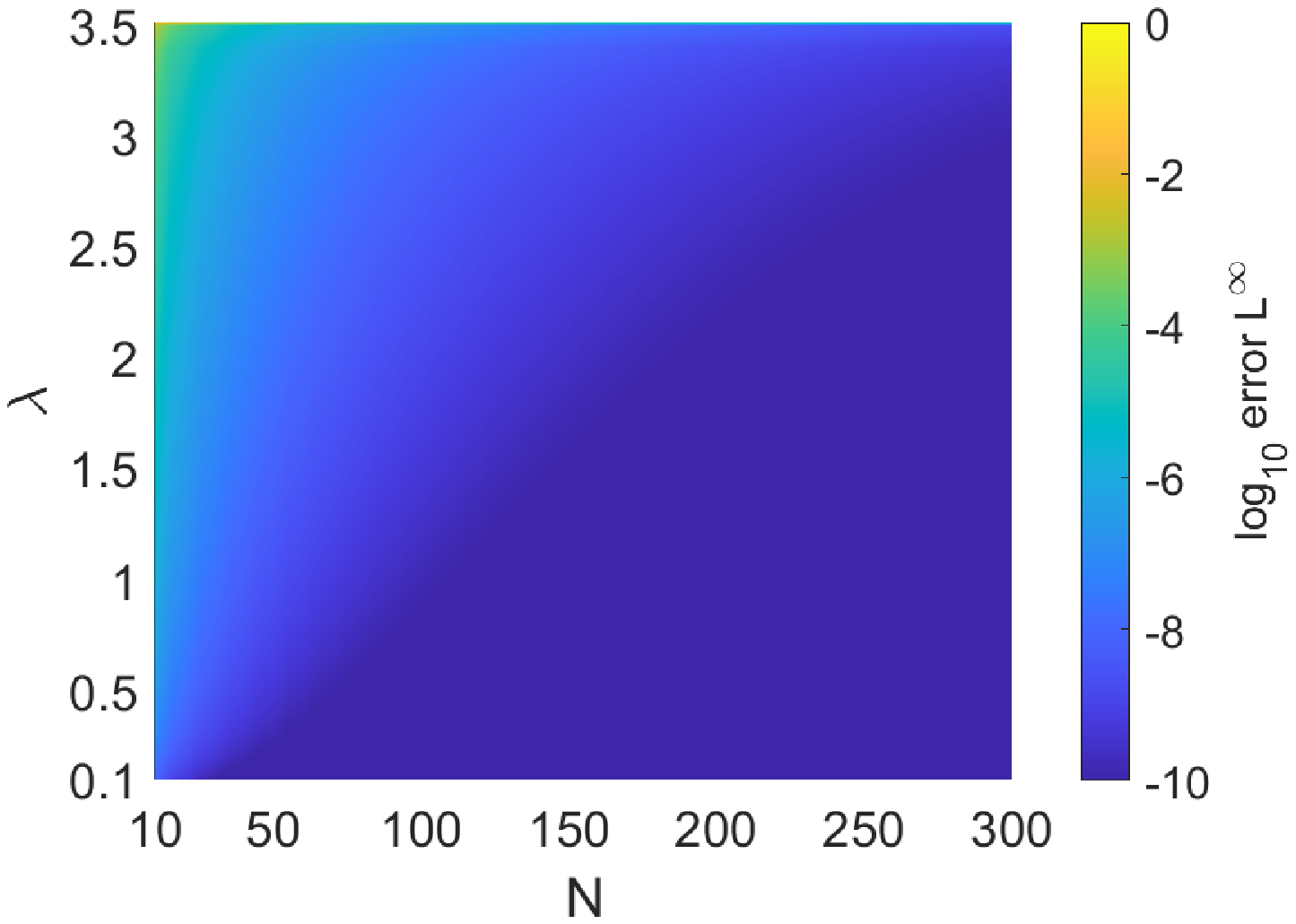}}
    \subfigure[]{
   \includegraphics[width=0.45 \textwidth]{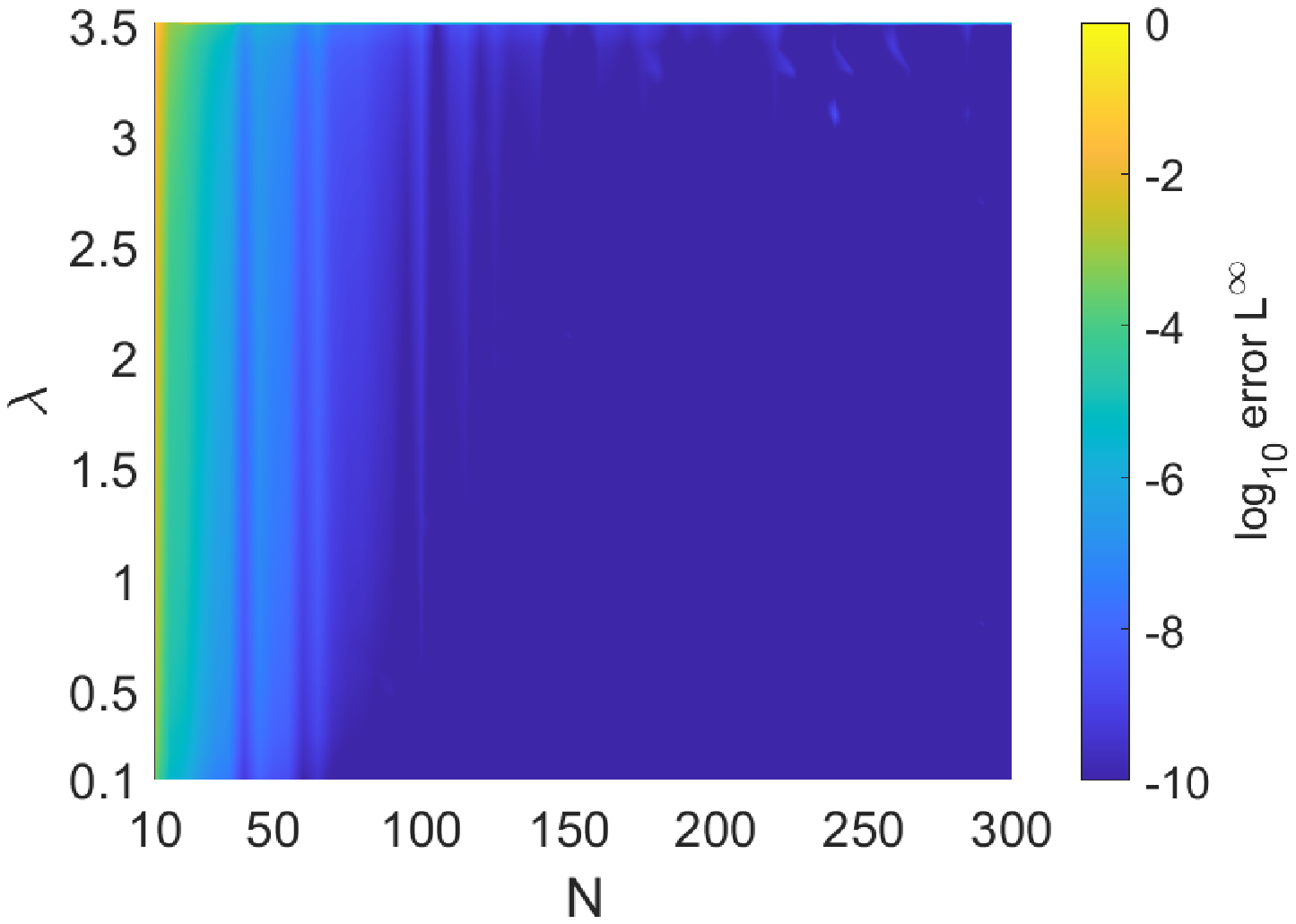}
}
~\subfigure[]{
    \includegraphics[width=0.45 \textwidth]{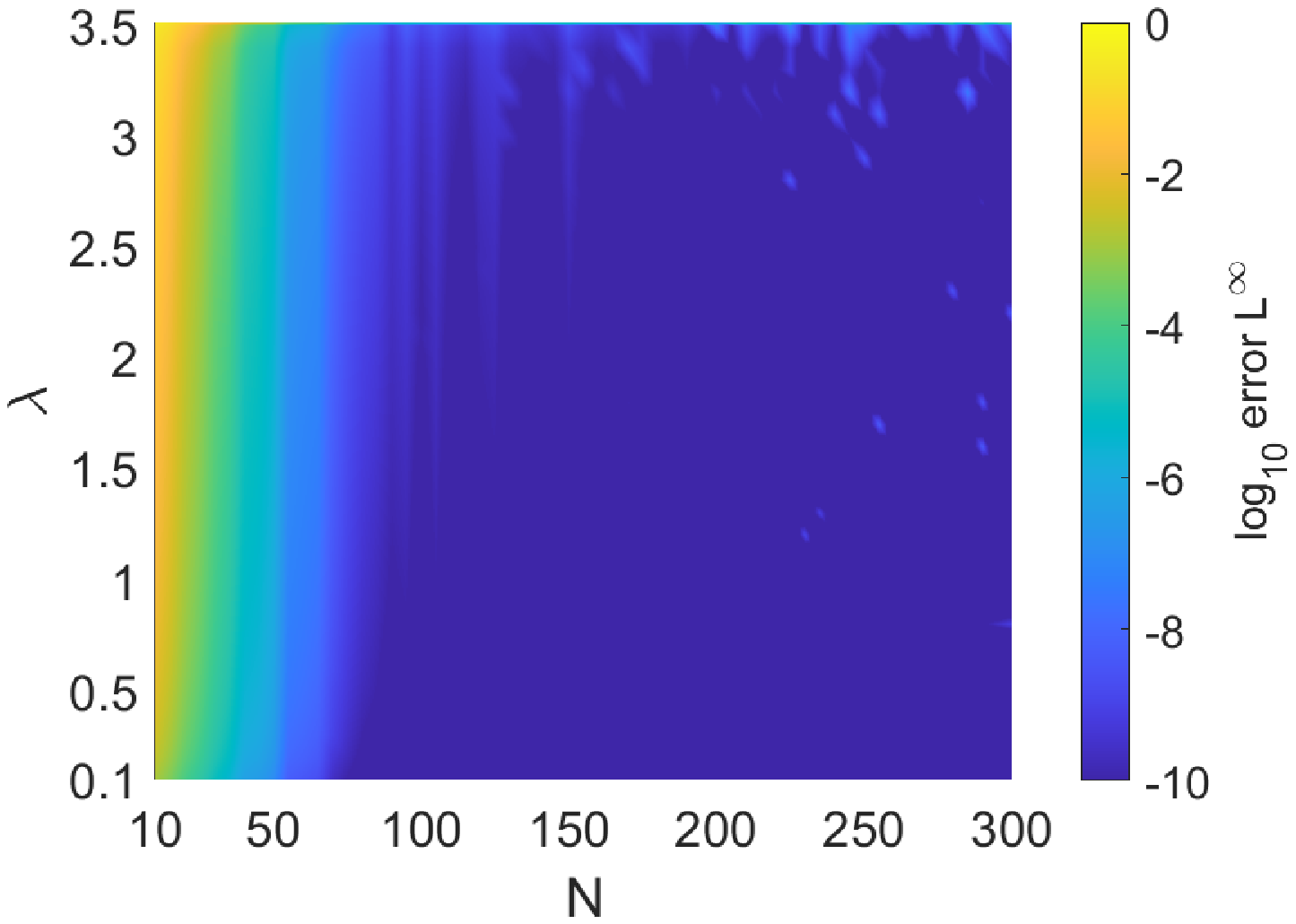}}
    \caption{One-dimensional Bratu problem (\ref{eq:Bratu}). Contour plots of the $L_{\infty}$--norms of the differences between the computed and exact \eqref{eq:sys1} solutions for the lower stable branch: (a) FD, (b) FEM, (c) ELM with logistic SF (\ref{eq:logisticSF}), (d) ELM with Gaussian RBF (\ref{eq:gaussianRBF}).}
    \label{fig:bratu_contourplot_down}
\end{figure}

\begin{figure}[ht]
    \centering
    \subfigure[]{
   \includegraphics[width=0.45 \textwidth]{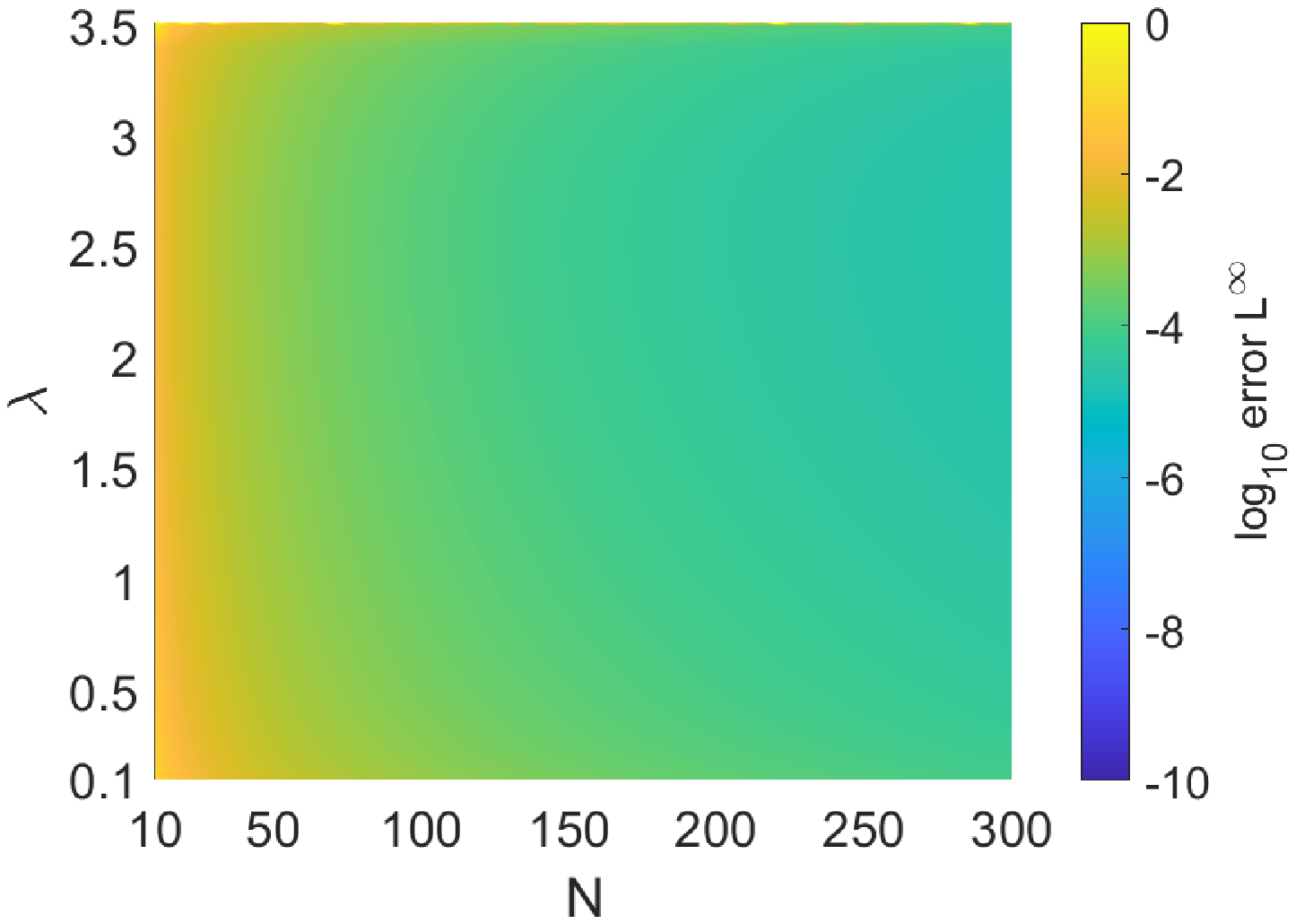}
}
~\subfigure[]{
    \includegraphics[width=0.45 \textwidth]{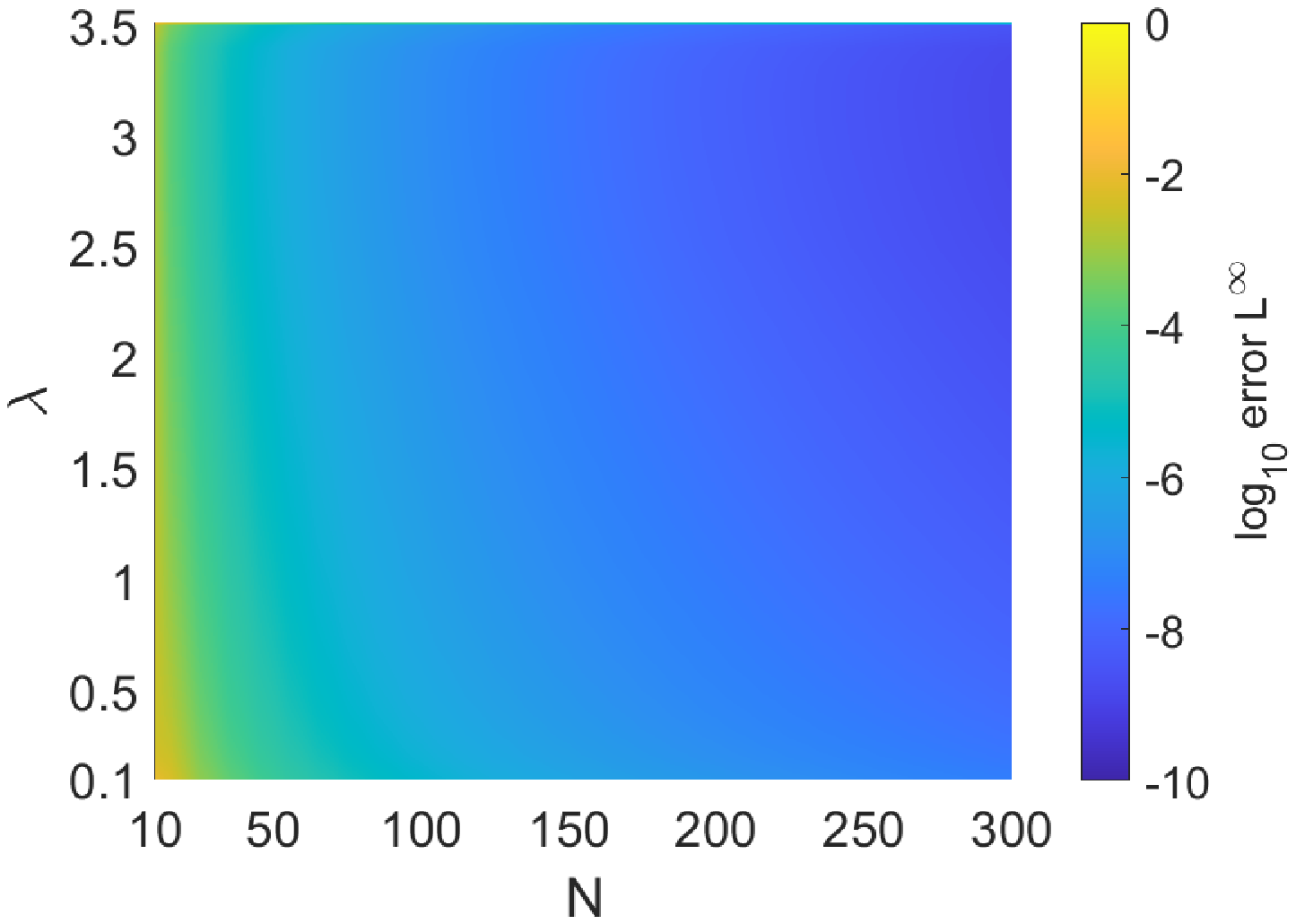}}
    \subfigure[]{
   \includegraphics[width=0.45 \textwidth]{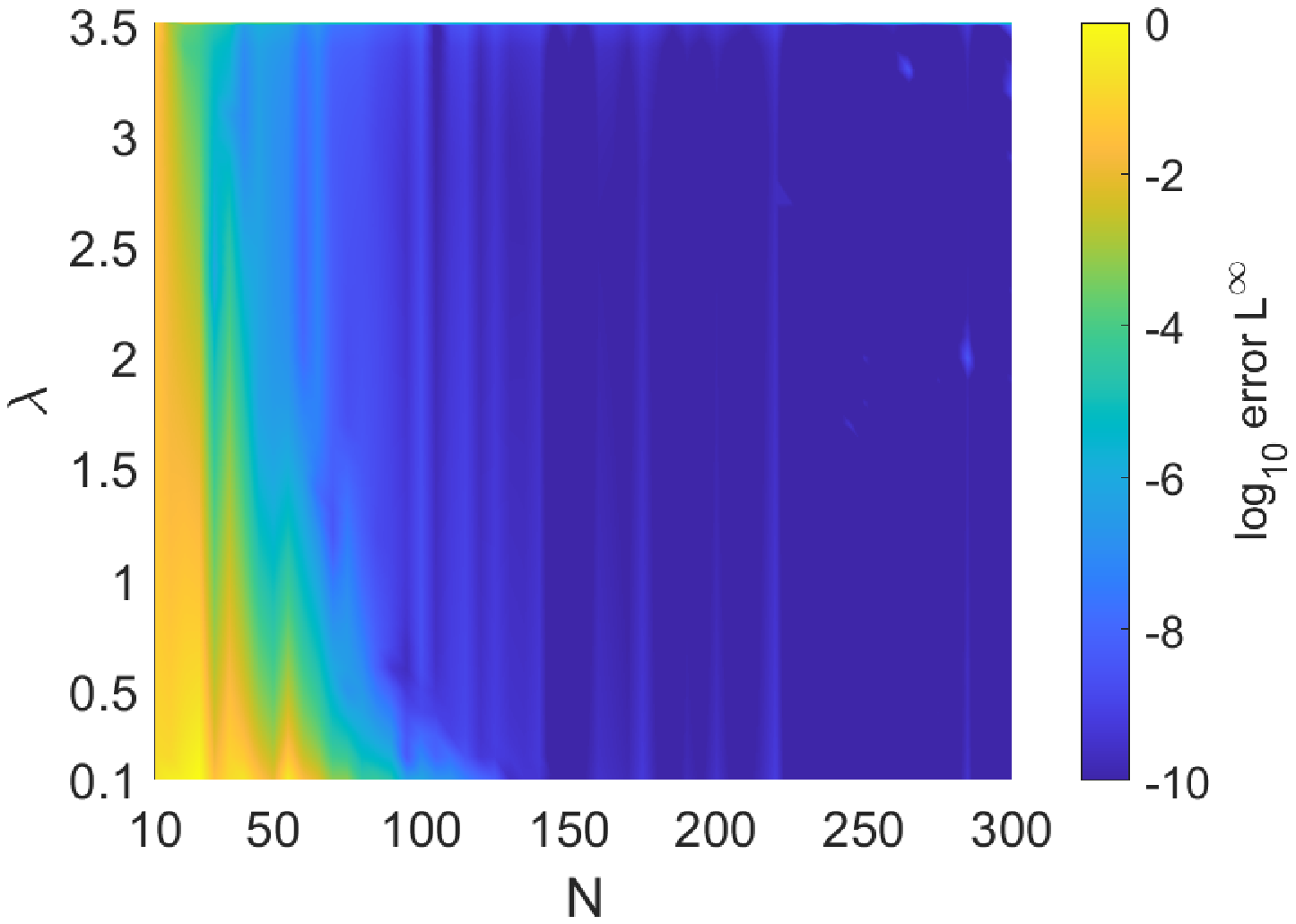}
}
~\subfigure[]{
    \includegraphics[width=0.45 \textwidth]{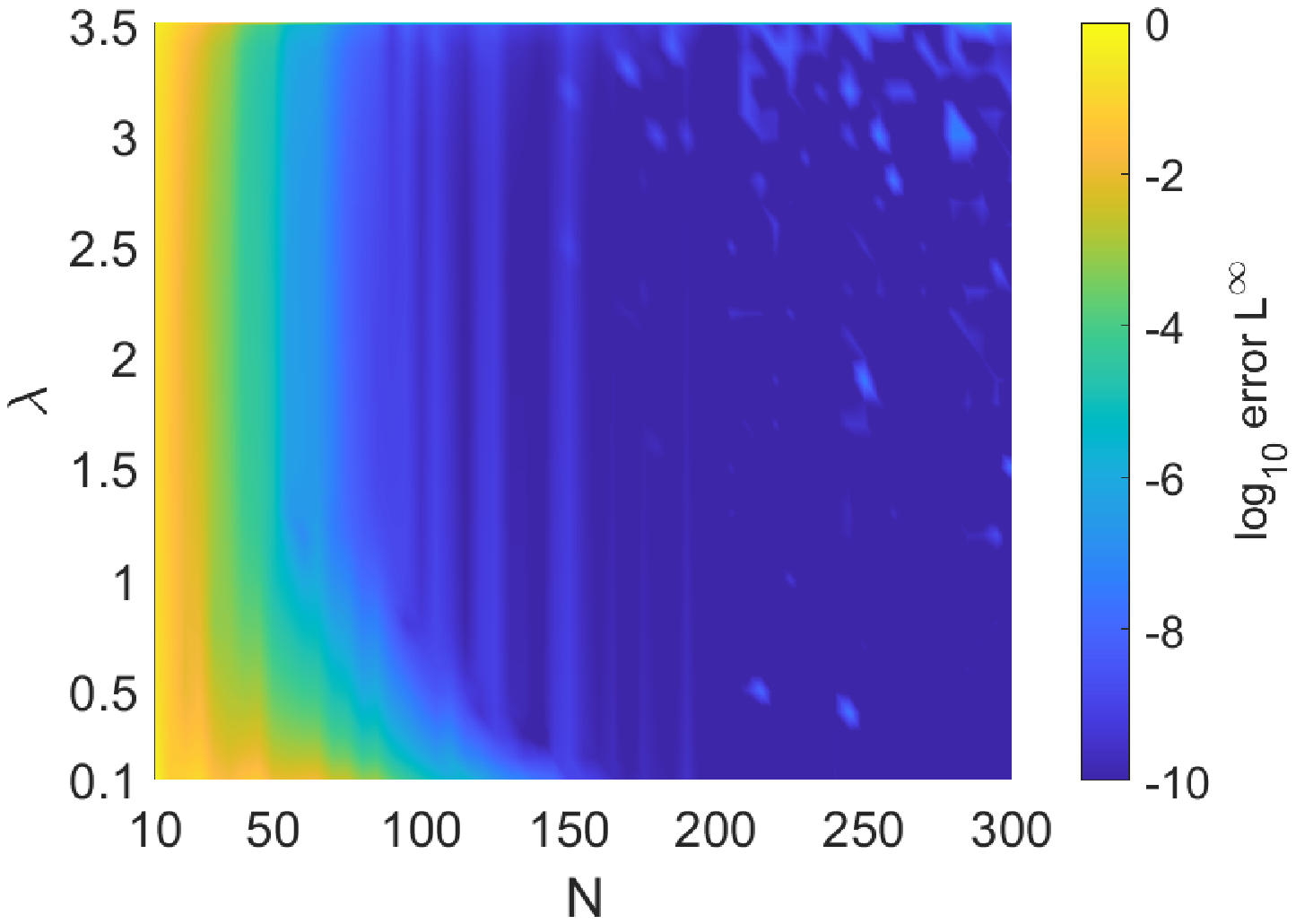}}
    \caption{One-dimensional Bratu problem (\ref{eq:Bratu}). Contour plots of the $L_{\infty}$--norms of the differences between the computed and exact \eqref{eq:sys1} solutions for the upper unstable branch: (a) FD, (b) FEM, (c) ELM with logistic SF (\ref{eq:logisticSF}), (d) ELM with Gaussian RBF (\ref{eq:gaussianRBF}).}
    \label{fig:bratu_contourplot_up}
\end{figure}

As it is shown, the ELM schemes outperform both FD and FEM methods for medium to large problem sizes $N$, and provide equivalent results with FEM for low to medium problem sizes, ths both (FEM and ELMs) outperforming the FD scheme.

\FloatBarrier
\subsubsection{Numerical results for the two-dimensional problem}
For the two-dimensional problem \eqref{eq:Bratu}-\eqref{eq:Bratu:BC}, no exact analytical solution is available. Thus, for comparing the numerical accuracy of the FD, FEM and ELM schemes, we considered the value of the bifurcation point that has been reported in key works as discussed in Section \ref{Sect:sez4.2}. Figure \ref{fig:bratu2D_bifurcation} depicts the computed bifurcation diagram as computed via pseudo-arc-length continuation (see section \ref{Sect:sez2}). Table 5, summarizes the computed values of the turning point as estimated with the FD, FEM and ELM schemes for various sizes $N$ of the grid.

\begin{figure}
    \centering
    \subfigure[]{
    \includegraphics[width=0.45 \textwidth]{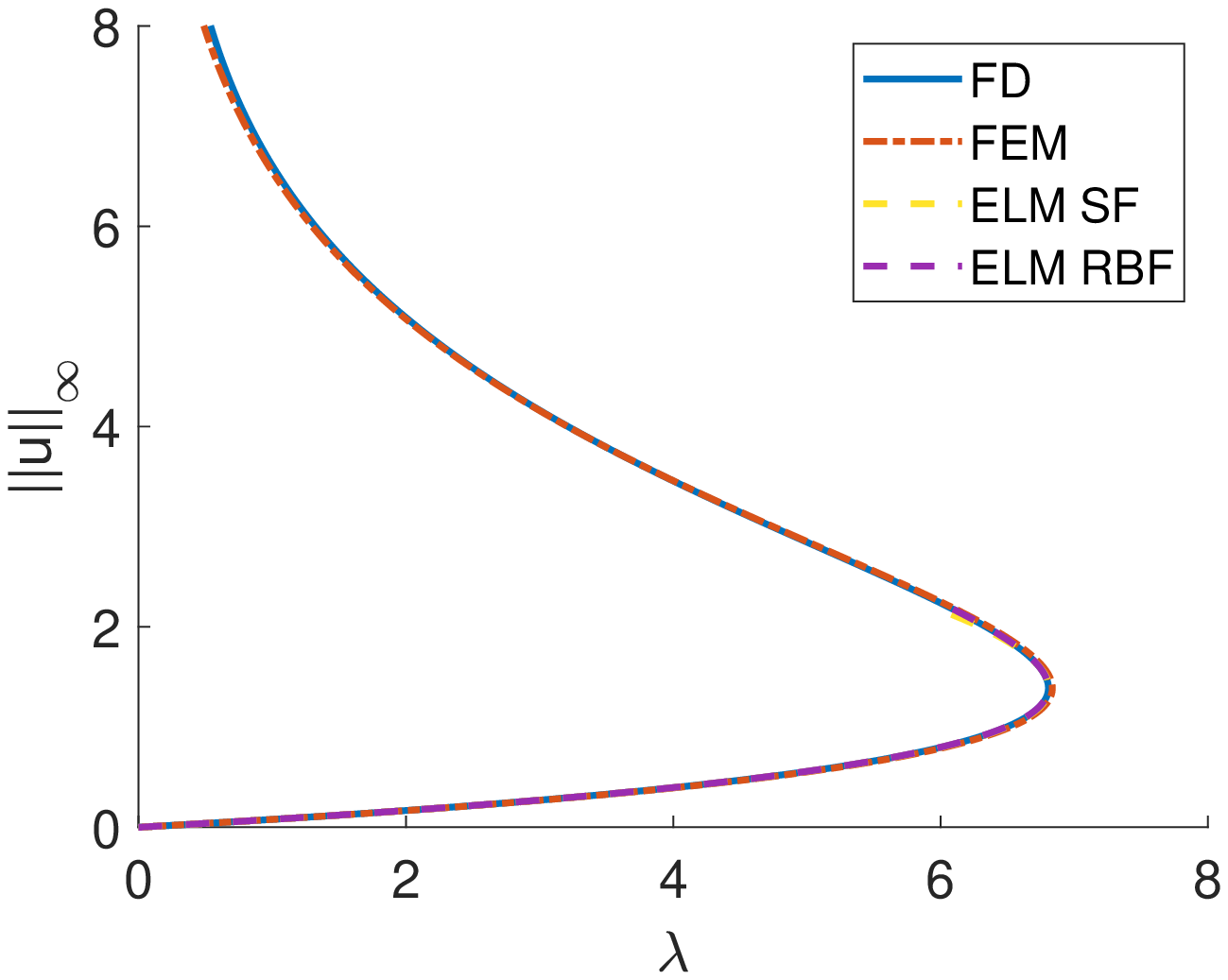}}
    \subfigure[]{\includegraphics[width=0.45 \textwidth]{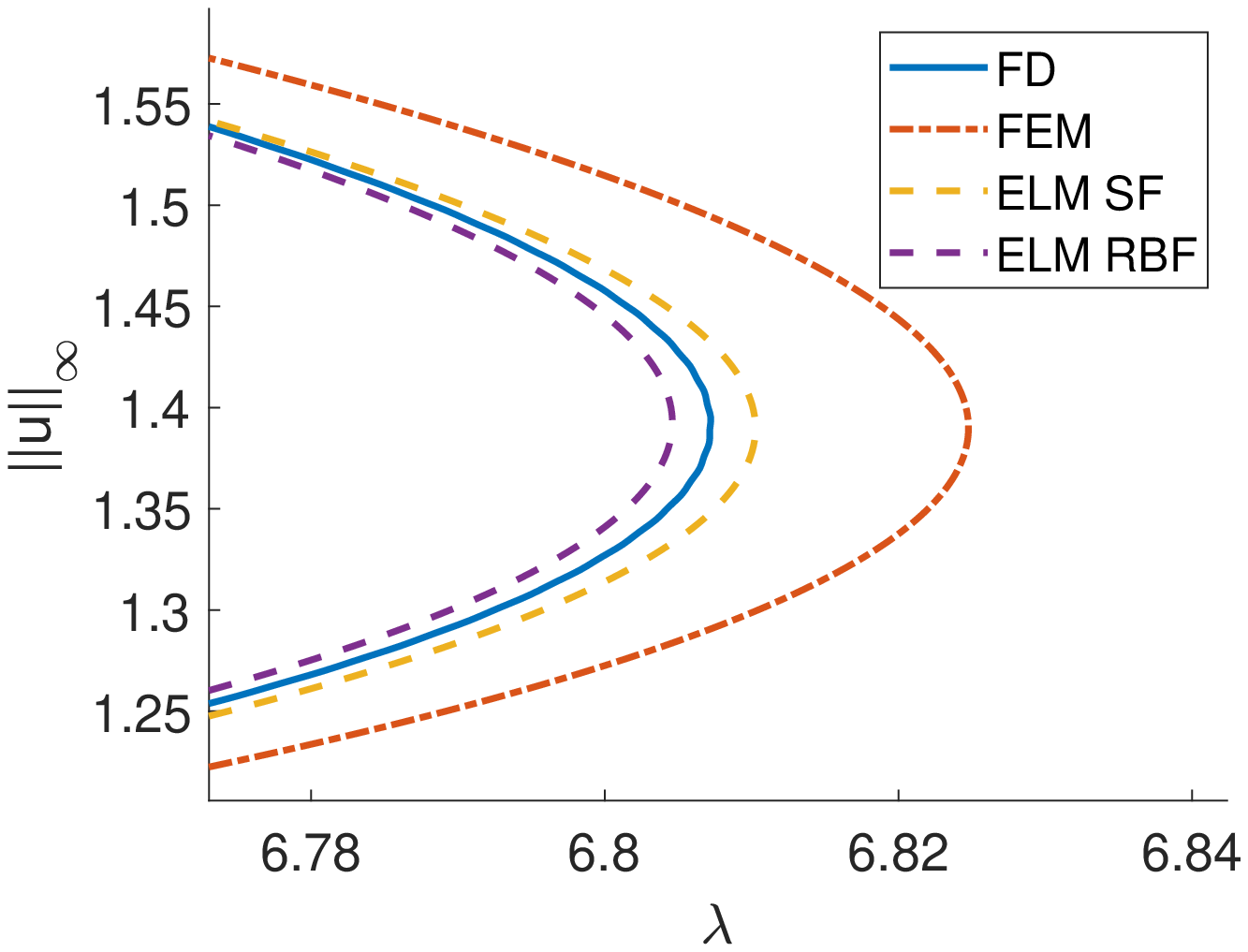}}
    \caption{(a) Computed bifurcation diagram for the two-dimensional Bratu problem (\ref{eq:Bratu}), with a grid of $40\times 40$ points. b) Zoom near the turning point.}
    \label{fig:bratu2D_bifurcation}
    
\end{figure}

\begin{table}[]
\begin{tabular}{ccrrrr}
\hline
\hline
N & Grid & FD & FEM & ELM SF & ELM RBF \\
\hline

64 & 8x8  & 6.783434 & 7.083742 & 6.845015 & 7.207203 \\
100 & 10x10  & 6.792626 & 6.984260 & 6.723902 & 6.930798 \\
196 & 14x14 & 6.800361 & 6.900313  & 6.855055 & 6.882435 \\
400 & 20x20 & 6.804392 & 6.856401 & 6.799440  & 6.829754 \\
784 & 28x28 & 6.806235 & 6.835771 & 6.801689 & 6.806149 \\
1600 & 40x40 & 6.807220 & 6.824770 & 6.806899 & 6.804600 \\
\hline
\hline
\end{tabular}
\caption{Turning point estimation of the two-dimensional Bratu problem. The value that has been reported in the literature in key works (see e.g. \cite{chan1982arc}) is $\lambda^*=6.808124$. The value of the turning point was estimated by fitting a parabola around the four points with the largest $\lambda$ values as obtained by the arc-length continuation.}
\end{table}

\begin{remark} [The Gelfand-Bratu model]\label{Remark:Bratu:Ball}
The Liouville–Bratu–Gelfand equation \eqref{eq:Bratu} in a unitary ball $B \subset \mathbb{R}^d $ with homogeneous Dirichlet boundary conditions is usually refereed as Gelfand-Bratu model. Such equation posses radial solutions $u({r})$ of the one-dimensional non-linear boundary-value problem \cite{syam2007modified}:
\begin{equation}
\left\{
    \begin{array}{rl}
        &u''(r)+ \dfrac{d-1}{r}u'(r)+ \lambda e^{u(r)} =0 \quad 0< r < 1  \\
        &u(1)=u'(0)=0 
    \end{array} \right.
    \label{eq:Bratu_radial}
\end{equation}
In the case $d=2$ this equation gives multiple solutions if $\lambda<\lambda_c=2$.
For example, in \cite{raja2013neural}, the authors have used Mathematica to give analytical solutions at various values of $\lambda$; for our tests we consider:
\begin{equation}
    \begin{array}{ll}
      \lambda=\frac{1}{2}  \,\rightarrow\   & u(r)=\log\left( \dfrac{ 16\left(7+4\sqrt{3}\right) }{\left( 7+4\sqrt{3}+ r^2 \right)^2  }\right) \\
      \lambda=1 \,\rightarrow\  & u(r)=\log\left( \dfrac{ 8\left(3+2\sqrt{2}\right) }{\left(3+2\sqrt{2}+ r^2 \right)^2  }\right) \ . \\
      \end{array}
      \label{eq:mathematica_solutions}
\end{equation}
Figure \ref{Fig:Bratu_2d_BALL} depicts the numerical accuracy of the ELM collocation schemes with respect to the exact solutions for two values of $\lambda$, namely for $\lambda=1/2$ and for $\lambda=1$. Because no meshing procedure is involved, and because the collocation equation seeks no other point, the implementation of the Newton's method is straightforward when changing the geometry of the domain.
\end{remark}
\begin{figure}[ht]
    \centering
    \subfigure[]{
   \includegraphics[width=0.45 \textwidth]{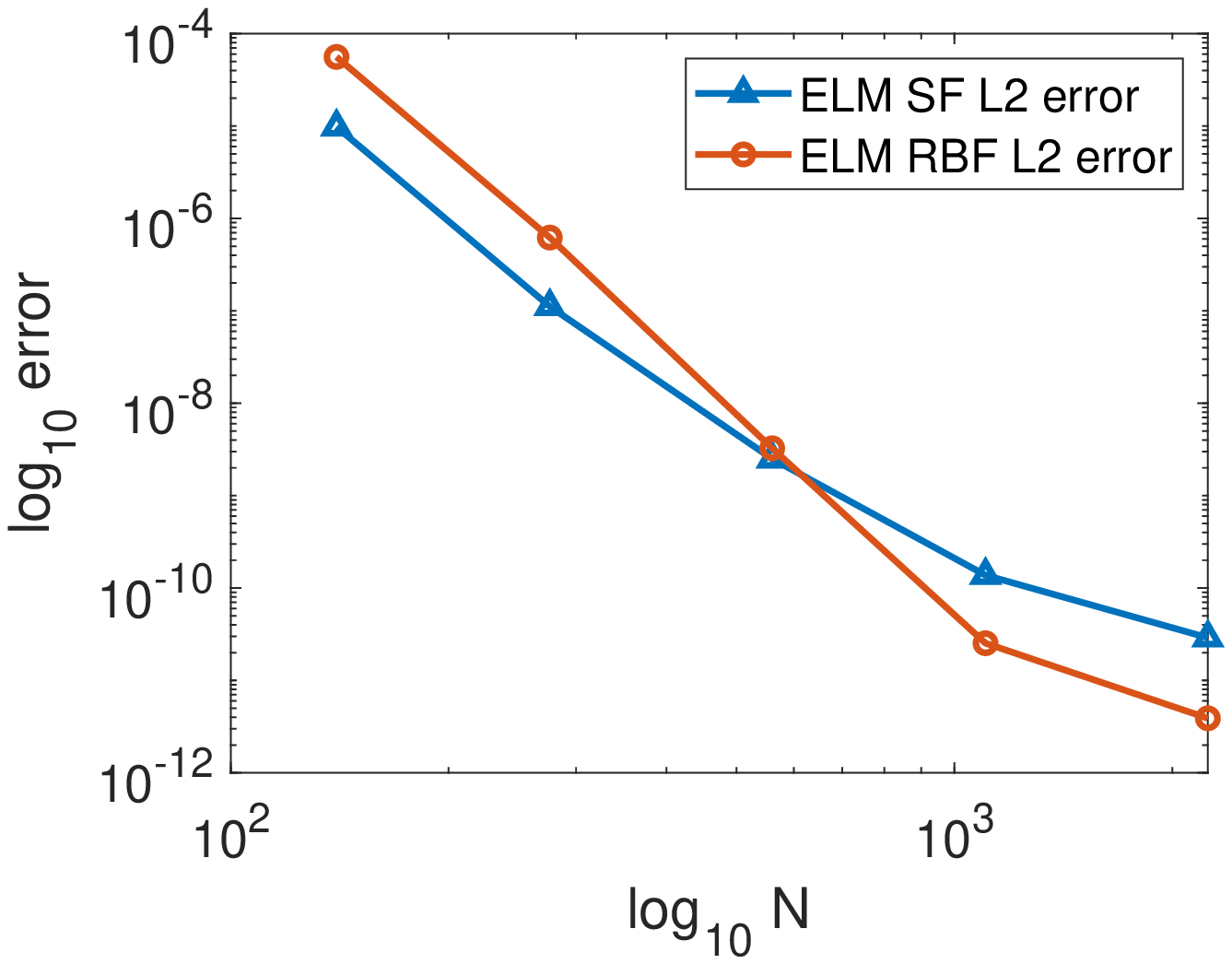}
}
~\subfigure[]{
    \includegraphics[width=0.45 \textwidth]{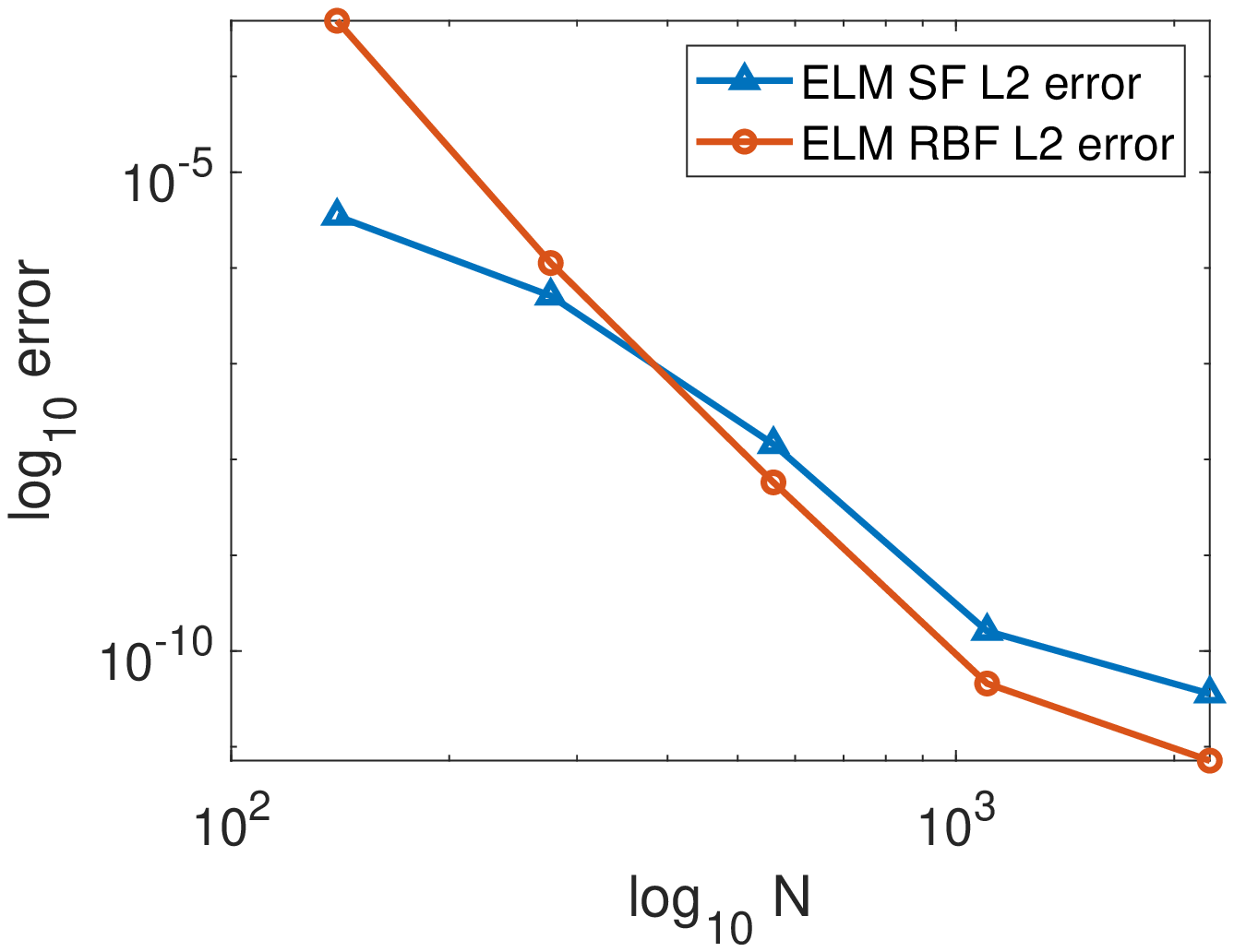}
}
\caption{Numerical accuracy of ELMs for the radial two-dimensional Gelfand-Bratu problem (\ref{eq:Bratu_radial}). $L_2$--norm of differences of the analytical solutions (\ref{eq:mathematica_solutions}) w.r.t. the number of neurons $N$ in ELMs with both logistic SF (\ref{eq:logisticSF}) and Gaussian RBF (\ref{eq:gaussianRBF}): (a) $\lambda=1/2$, (b) $\lambda=1$.}
\label{Fig:Bratu_2d_BALL}
\end{figure}
\FloatBarrier
\section{Conclusions}
We proposed a numerical approach based on Extreme Learning Machines (ELMs) and collocation for the approximation of steady-state solutions of non-linear PDEs. The proposed scheme takes advantage of the property of the ELMs as universal function approximators, bypassing the need of the computational very expensive - and most-of-the times without any guarantee for convergence of-the training phase of other types of machine learning such as single or multilayer ANNs and Deep-learning networks. The base of the approximation subspace on which a solution of the PDE is sought are the (unknown) weights of the hidden to output layer. For linear PDEs, these can be computed by solving a linear regularization problem in one step. In our previous work \cite{calabro2020extreme}, we demonstrated that ELMs can provide robust and accurate approximations of the solution of benchmark linear PDEs with steep gradients, for which analytical solutions were available. Here, building on this work, we make a step change by showing how ELMs can be used to solve non-linear PDEs, and by bridging them with continuation methods, we show how one can exploit the arsenal of numerical bifurcation theory to trace branches of solutions past critical points. For our demonstrations, we considered two celebrated classes of nonlinear PDEs whose solutions bifurcate as parameter values change: the one-dimensional viscous Burgers equation (a fundamental representative of advection-diffusion PDEs) and the one- and two-dimensional Liouville–Bratu–Gelfand equation (a fundamental representative of reaction-diffusion PDEs). By coupling the proposed numerical scheme with Newton-Raphson iterations and the ``pseudo" arc-length continuation method, we constructed the corresponding bifurcation diagrams past turning points. The efficiency of the proposed numerical ELM collocation scheme was compared against two of the most established numerical solution methods, namely central Finite Differences and Galerkin Finite Elements. By doing so, we showed that (for the same problem size) the proposed machine-learning approach outperforms FD and FEM schemes for relatively medium to large sizes of the grid, both with respect to the accuracy of the computed solutions for a wide range of the bifurcation parameter values and the approximation accuracy of the turning points. Hence, the proposed approach arises as an alternative and powerful new numerical technique for the approximation of steady-state solutions of non-linear PDEs. Furthermore, its implementation is far simpler than the implementation of FEM, thus providing equivalent or even better numerical accuracy, and in all cases is shown to outperform the simple FD scheme, which fails to approximate steep gradients as here arise near the boundaries.
Of course there are many open problems linked to the implementation of the proposed scheme that ask for further and deeper investigation, such as the theoretical investigation of the impact of the type of transfer functions and the probability distribution of their parameter values functions to the approximation of the solutions. Further directions could be towards the extension of the scheme for the solution of time-dependent non-linear PDEs as well as the solution of inverse-problems in PDEs.

\section*{Acknowledgments}
Francesco Calabr\`o and Constantinos Siettos were partially supported by INdAM, through GNCS research projects. This support is gratefully acknowledged.

\bibliographystyle{spmpsci}
\bibliography{references}  

\begin{thebibliography}{10}
\providecommand{\url}[1]{{#1}}
\providecommand{\urlprefix}{URL }
\expandafter\ifx\csname urlstyle\endcsname\relax
  \providecommand{\doi}[1]{DOI~\discretionary{}{}{}#1}\else
  \providecommand{\doi}{DOI~\discretionary{}{}{}\begingroup
  \urlstyle{rm}\Url}\fi

\bibitem{allen2013numerical}
Allen, E.J., Burns, J.A., Gilliam, D.S.: Numerical approximations of the
  dynamical system generated by burgers' equation with neumann-dirichlet
  boundary conditions.
\newblock ESAIM: Mathematical Modelling and Numerical Analysis-Mod{\'e}lisation
  Math{\'e}matique et Analyse Num{\'e}rique \textbf{47}(5), 1465--1492 (2013)

\bibitem{asprone2010particle}
Asprone, D., Auricchio, F., Manfredi, G., Prota, A., Reali, A., Sangalli, G.:
  Particle methods for a 1 d elastic model problem: Error analysis and
  development of a second-order accurate formulation.
\newblock Computer Modeling in Engineering \& Sciences(CMES) \textbf{62}(1),
  1--21 (2010)

\bibitem{auricchio2012isogeometric}
Auricchio, F., Da~Veiga, L.B., Hughes, T.J., Reali, A., Sangalli, G.:
  Isogeometric collocation for elastostatics and explicit dynamics.
\newblock Computer methods in applied mechanics and engineering \textbf{249},
  2--14 (2012)

\bibitem{bai2014sparse}
Bai, Z., Huang, G.B., Wang, D., Wang, H., Westover, M.B.: Sparse extreme
  learning machine for classification.
\newblock IEEE transactions on cybernetics \textbf{44}(10), 1858--1870 (2014)

\bibitem{benton1972table}
Benton, E.R., Platzman, G.W.: A table of solutions of the one-dimensional
  burgers equation.
\newblock Quarterly of Applied Mathematics \textbf{30}(2), 195--212 (1972)

\bibitem{boyd1986analytical}
Boyd, J.P.: An analytical and numerical study of the two-dimensional bratu
  equation.
\newblock Journal of Scientific Computing \textbf{1}(2), 183--206 (1986)

\bibitem{brezzi1982finite}
Brezzi, F., Rappaz, J., Raviart, P.A.: Finite dimensional approximation of
  nonlinear problems.
\newblock Numerische Mathematik \textbf{38}(1), 1--30 (1982)

\bibitem{calabro2020extreme}
Calabr{\`o}, F., Fabiani, G., Siettos, C.: Extreme learning machine collocation
  for the numerical solution of elliptic pdes with sharp gradients.
\newblock arXiv preprint arXiv:2012.05871  (2020)

\bibitem{chan1982arc}
Chan, T.F., Keller, H.: Arc-length continuation and multigrid techniques for
  nonlinear elliptic eigenvalue problems.
\newblock SIAM Journal on Scientific and Statistical Computing \textbf{3}(2),
  173--194 (1982)

\bibitem{chan2019machine}
Chan-Wai-Nam, Q., Mikael, J., Warin, X.: Machine learning for semi linear pdes.
\newblock Journal of Scientific Computing \textbf{79}(3), 1667--1712 (2019)

\bibitem{chaturvedi2018bayesian}
Chaturvedi, I., Ragusa, E., Gastaldo, P., Zunino, R., Cambria, E.: Bayesian
  network based extreme learning machine for subjectivity detection.
\newblock Journal of The Franklin Institute \textbf{355}(4), 1780--1797 (2018)

\bibitem{chen2020unsupervised}
Chen, J., Zeng, Y., Li, Y., Huang, G.B.: Unsupervised feature selection based
  extreme learning machine for clustering.
\newblock Neurocomputing \textbf{386}, 198--207 (2020)

\bibitem{cliffe2000numerical}
Cliffe, K., Spence, A., Tavener, S.: The numerical analysis of bifurcation
  problems with application to fluid mechanics.
\newblock Acta Numerica \textbf{9}(00), 39--131 (2000)

\bibitem{dai2019multilayer}
Dai, H., Cao, J., Wang, T., Deng, M., Yang, Z.: Multilayer one-class extreme
  learning machine.
\newblock Neural Networks \textbf{115}, 11--22 (2019)

\bibitem{dhooge2008new}
Dhooge, A., Govaerts, W., Kuznetsov, Y.A., Meijer, H.G.E., Sautois, B.: New
  features of the software matcont for bifurcation analysis of dynamical
  systems.
\newblock Mathematical and Computer Modelling of Dynamical Systems
  \textbf{14}(2), 147--175 (2008)

\bibitem{doedel2012numerical}
Doedel, E., Tuckerman, L.S.: Numerical methods for bifurcation problems and
  large-scale dynamical systems, vol. 119.
\newblock Springer Science \& Business Media (2012)

\bibitem{doedel2007auto}
Doedel, E.J., Champneys, A.R., Dercole, F., Fairgrieve, T.F., Kuznetsov, Y.A.,
  Oldeman, B., Paffenroth, R., Sandstede, B., Wang, X., Zhang, C.: Auto-07p:
  Continuation and bifurcation software for ordinary differential equations
  (2007)

\bibitem{DWIVEDI202096}
Dwivedi, V., Srinivasan, B.: Physics informed extreme learning machine
  ({PIELM}) - {A} rapid method for the numerical solution of partial
  differential equations.
\newblock Neurocomputing \textbf{391}, 96 -- 118 (2020)

\bibitem{fresca2020comprehensive}
Fresca, S., Dede, L., Manzoni, A.: A comprehensive deep learning-based approach
  to reduced order modeling of nonlinear time-dependent parametrized pdes.
\newblock Journal of Scientific Computing \textbf{87}(61) (2021)

\bibitem{gebhardt2020framework}
Gebhardt, C.G., Steinbach, M.C., Schillinger, D., Rolfes, R.: A framework for
  data-driven structural analysis in general elasticity based on nonlinear
  optimization: The dynamic case.
\newblock International Journal for Numerical Methods in Engineering
  \textbf{121}(24), 5447--5468 (2020)

\bibitem{glowinski1985continuation}
Glowinski, R., Keller, H.B., Reinhart, L.: Continuation-conjugate gradient
  methods for the least squares solution of nonlinear boundary value problems.
\newblock SIAM journal on scientific and statistical computing \textbf{6}(4),
  793--832 (1985)

\bibitem{gonzalez1998identification}
Gonz{\'a}lez-Garc{\'\i}a, R., Rico-Mart{\`\i}nez, R., Kevrekidis, I.G.:
  Identification of distributed parameter systems: A neural net based approach.
\newblock Computers \& chemical engineering \textbf{22}, S965--S968 (1998)

\bibitem{govaerts2000numerical}
Govaerts, W.J.: Numerical methods for bifurcations of dynamical equilibria.
\newblock SIAM (2000)

\bibitem{hadash2018estimate}
Hadash, G., Kermany, E., Carmeli, B., Lavi, O., Kour, G., Jacovi, A.: Estimate
  and replace: A novel approach to integrating deep neural networks with
  existing applications.
\newblock arXiv preprint arXiv:1804.09028  (2018)

\bibitem{hajipour2018accurate}
Hajipour, M., Jajarmi, A., Baleanu, D.: On the accurate discretization of a
  highly nonlinear boundary value problem.
\newblock Numerical Algorithms \textbf{79}(3), 679--695 (2018)

\bibitem{han2018solving}
Han, J., Jentzen, A., Weinan, E.: Solving high-dimensional partial differential
  equations using deep learning.
\newblock Proceedings of the National Academy of Sciences \textbf{115}(34),
  8505--8510 (2018)

\bibitem{huang2015trends}
Huang, G., Huang, G.B., Song, S., You, K.: Trends in extreme learning machines:
  A review.
\newblock Neural Networks \textbf{61}, 32--48 (2015)

\bibitem{huang2013}
Huang, G., Kasun, L., Zhou, H., Vong, C.: Representational learning with
  extreme learning machine for big data.
\newblock IEEE Intelligent Systems \textbf{28}(6), 31--34 (2013)

\bibitem{Huang}
{Huang}, G., {Zhou}, H., {Ding}, X., {Zhang}, R.: Extreme learning machine for
  regression and multiclass classification.
\newblock IEEE Transactions on Systems, Man, and Cybernetics, Part B
  (Cybernetics) \textbf{42}(2), 513--529 (2012).
\newblock \doi{10.1109/TSMCB.2011.2168604}

\bibitem{huang2010optimization}
Huang, G.B., Ding, X., Zhou, H.: Optimization method based extreme learning
  machine for classification.
\newblock Neurocomputing \textbf{74}(1-3), 155--163 (2010)

\bibitem{huang2011extreme}
Huang, G.B., Zhou, H., Ding, X., Zhang, R.: Extreme learning machine for
  regression and multiclass classification.
\newblock IEEE Transactions on Systems, Man, and Cybernetics, Part B
  (Cybernetics) \textbf{42}(2), 513--529 (2011)

\bibitem{huang2006extreme}
Huang, G.B., Zhu, Q.Y., Siew, C.K.: Extreme learning machine: theory and
  applications.
\newblock Neurocomputing \textbf{70}(1-3), 489--501 (2006)

\bibitem{iqbal2020numerical}
Iqbal, S., Zegeling, P.A.: A numerical study of the higher-dimensional
  gelfand-bratu model.
\newblock Computers \& Mathematics with Applications \textbf{79}(6), 1619--1633
  (2020)

\bibitem{kelley_2018}
Kelley, C.T.: Numerical methods for nonlinear equations.
\newblock Acta Numerica \textbf{27}, 207–287 (2018).
\newblock \doi{10.1017/S0962492917000113}

\bibitem{krauskopf2007numerical}
Krauskopf, B., Osinga, H.M., Gal{\'a}n-Vioque, J.: Numerical continuation
  methods for dynamical systems, vol.~2.
\newblock Springer (2007)

\bibitem{kuznetsov2013elements}
Kuznetsov, Y.A.: Elements of applied bifurcation theory, vol. 112.
\newblock Springer Science \& Business Media (2013)

\bibitem{lagaris1998artificial}
Lagaris, I.E., Likas, A., Fotiadis, D.I.: Artificial neural networks for
  solving ordinary and partial differential equations.
\newblock IEEE transactions on neural networks \textbf{9}(5), 987--1000 (1998)

\bibitem{mohsen2014simple}
Mohsen, A.: A simple solution of the bratu problem.
\newblock Computers \& Mathematics with Applications \textbf{67}(1), 26--33
  (2014)

\bibitem{olson1991efficient}
Olson, L.G., Georgiou, G.C., Schultz, W.W.: An efficient finite element method
  for treating singularities in laplace's equation.
\newblock Journal of Computational Physics \textbf{96}(2), 391--410 (1991)

\bibitem{pinkus1999approximation}
Pinkus, A.: Approximation theory of the mlp model.
\newblock Acta Numerica 1999: Volume 8 \textbf{8}, 143--195 (1999)

\bibitem{quarteroni2008numerical}
Quarteroni, A., Valli, A.: Numerical approximation of partial differential
  equations, vol.~23.
\newblock Springer Science \& Business Media (2008)

\bibitem{raissi2018numerical}
Raissi, M., Perdikaris, P., Karniadakis, G.E.: Numerical gaussian processes for
  time-dependent and nonlinear partial differential equations.
\newblock SIAM Journal on Scientific Computing \textbf{40}(1), A172--A198
  (2018)

\bibitem{raissi2019physics}
Raissi, M., Perdikaris, P., Karniadakis, G.E.: Physics-informed neural
  networks: A deep learning framework for solving forward and inverse problems
  involving nonlinear partial differential equations.
\newblock Journal of Computational Physics \textbf{378}, 686--707 (2019)

\bibitem{raja2013neural}
Raja, M.A.Z., Samar, R., et~al.: Neural network optimized with evolutionary
  computing technique for solving the 2-dimensional bratu problem.
\newblock Neural Computing and Applications \textbf{23}(7), 2199--2210 (2013)

\bibitem{samaniego2020energy}
Samaniego, E., Anitescu, C., Goswami, S., Nguyen-Thanh, V.M., Guo, H., Hamdia,
  K., Zhuang, X., Rabczuk, T.: An energy approach to the solution of partial
  differential equations in computational mechanics via machine learning:
  Concepts, implementation and applications.
\newblock Computer Methods in Applied Mechanics and Engineering \textbf{362},
  112790 (2020)

\bibitem{schilder2017continuation}
Schilder, F., Dankowicz, H.: Continuation core and toolboxes (coco).
\newblock Source-Forge. net, project cocotools  (2017)

\bibitem{syam2007modified}
Syam, M.I.: The modified broyden-variational method for solving nonlinear
  elliptic differential equations.
\newblock Chaos, Solitons \& Fractals \textbf{32}(2), 392--404 (2007)

\bibitem{tang2015extreme}
Tang, J., Deng, C., Huang, G.B.: Extreme learning machine for multilayer
  perceptron.
\newblock IEEE transactions on neural networks and learning systems
  \textbf{27}(4), 809--821 (2015)

\bibitem{tissera2016deep}
Tissera, M.D., McDonnell, M.D.: Deep extreme learning machines: supervised
  autoencoding architecture for classification.
\newblock Neurocomputing \textbf{174}, 42--49 (2016)

\bibitem{wang2011study}
Wang, Y., Cao, F., Yuan, Y.: A study on effectiveness of extreme learning
  machine.
\newblock Neurocomputing \textbf{74}(16), 2483--2490 (2011)

\bibitem{wei2018machine}
Wei, Q., Jiang, Y., Chen, J.Z.: Machine-learning solver for modified diffusion
  equations.
\newblock Physical Review E \textbf{98}(5), 053304 (2018)

\end{thebibliography}


%



\end{document}